\documentclass[11pt]{amsart}

\usepackage{amsmath,amsfonts,amssymb,amsthm,epsfig,color}

\usepackage{enumitem}
\usepackage{esint}
\usepackage{color}

\usepackage{import}

\usepackage[final,allcolors=blue,colorlinks=true]{hyperref}
\usepackage[leqno]{amsmath}

\newtheorem{theorem}{Theorem}[section]
\newtheorem{lemma}[theorem]{Lemma}
\newtheorem{proposition}[theorem]{Proposition}

\theoremstyle{definition}
\newtheorem{definition}[theorem]{Definition}

\newtheorem{remark}[theorem]{Remark}
\numberwithin{equation}{section}

\newcommand{\R}{\mathbb{R}}

\newcommand{\N}{\mathbb{N}}
\newcommand{\medint}{-\kern  -,395cm\int}
\newcommand{\medintdue}{-\kern  -,31cm\int}

\newcommand{\eps}{\varepsilon}

\newcommand{\pa}{\partial}

\newcommand{\Ha}{{\mathcal{H}}}
\newcommand{\A}{\mathcal{A}}

\newcommand{\fh}{\mathfrak{h}}
\newcommand{\ft}{\mathfrak{t}}

\newcommand{\trace}{\operatorname{Tr}}

\renewcommand{\MR}[1]{\null}

\def\beq{\begin{equation}}
\def\eeq{\end{equation}}
\begin{document}

\title[Eventual regularity]{Eventual regularity of  the volume-preserving mean curvature flow in three and two dimensions}

\author{Vedansh Arya}
\address{Department of Mathematics and Statistics, Indian Institute of Technology Kanpur \\ 
India}\email[Vedansh Arya]{vedansh@iitk.ac.in}

\author{Seongmin Jeon}
\address{Department of Mathematics Education \\
Hanyang University \\ Republic of Korea}\email[Seongmin Jeon]{seongminjeon@hanyang.ac.kr}

\author{Vesa Julin}
\address{Department of Mathematics and Statistics\\
University of Jyväskylä\\ Finland}\email[Vesa Julin]{vesa.julin@jyu.fi}

\begin{abstract}
The recent work of Morini-Oronzio-Spadaro and the third author \cite{JuMoOrSpa} shows that, in three dimensions, a flat‑flow solution of the volume‑preserving mean curvature flow that converges to a single ball, which is the case for instance when the initial perimeter is smaller than that of two disjoint balls, converges  exponentially fast in  Hausdorff distance.  In this paper we strengthen this result by proving that  after a finite time the flow becomes smooth, satisfies the equation in the classical sense and converges exponentially fast to the limiting ball in every $C^k$‑norm.  In the  proof we develop a version of Brakke’s $\eps$-regularity theorem adapted to our setting and derive the necessary nonlinear PDE estimates directly at the level of the discrete minimizing-movements scheme.  The same result holds in the planar case.
\end{abstract}
\maketitle
\tableofcontents
\section{Introduction }

We continue the study of the asymptotic behavior of the volume-preserving mean curvature flow of sets $\{E_t\}_{t \geq 0}$, $E_t \subset \R^{n+1}$
\begin{equation}\label{eq:VMCF}
V_t = -\mathrm{H}_{E_t} + \mathrm{\bar H}_{E_t}  \quad\text{on }\pa E_t\subset\R^{n+1},
\end{equation}
where $V_t$ denotes the outer normal velocity, $\mathrm{H}_{E_t}$ is the mean curvature   and $\mathrm{\bar H}_{E_t} :=\fint_{\partial E_t} \mathrm{H}_{E_t}\, d\mathcal{H}^2$. If the flow does not develop singularities, one can show that it converges exponentially fast to a ball, or their union. This is the case for instance when the initial set is convex \cite{Hui} or close to  the ball \cite{ES}. Similar stability results hold also in the flat torus when the initial set is near a stable critical configuration \cite{DKK, Joonas}.  The situation becomes more challenging in the presence of singularities, in which case one has to  extend the notion of solution.  From physical point of view it is natural to require that  weak solution of \eqref{eq:VMCF} is given by family  of sets  of finite perimeter in $\R^{n+1}$  defined for all times.  For start-shaped sets Kim-Kwon \cite{KK} construct a weak global-in-time solution using the level-set approach and prove its convergence to a ball. We use the notion of  a volume‑preserving flat flow  $\{E(t)\}_{t \geq 0}$, originally introduced for mean curvature flow by Almgren-Taylor-Wang \cite{ATW} and Luckhaus-St\"urzenhecker \cite{LS}, and adapted to the volume-preserving setting by Mugnai-Seis-Spadaro \cite{MSS}. For technical convenience we use the variant construction from \cite{Vesa}; see Section \ref{sec:flatflow} for the precise definition.

An important  result on possible asymptotic limits of \eqref{eq:VMCF} is due to Delgadino-Maggi \cite{DM}, who showed that any set of finite perimeter  with constant distributional mean curvature is a (possibly tangent) union of balls. 
 Morini-Posiglione-Spadaro \cite{MoPoSpa} proved that the discrete approximation scheme for the (volume-preserving) flat flow converges exponentially fast to a union of balls in every dimension.  In a recent work by Bonforte-Maggi-Restrepo \cite{BMR}  the convergence  is proven in every dimension for the phase-field approximation of \eqref{eq:VMCF}. For the flat flow itself, Niinikoski and the third author established convergence (up to translations of components) to a union of balls in low dimensions $n\leq 2$  \cite{JN}. This is refined in \cite{JuMoOrSpa}, where it is shown that if the limiting balls remain at positive distances from each other,  in particular when there is a single limiting ball, then the convergence is exponential in the Hausdorff distance. Exponential Hausdorff convergence in the planar case appears in \cite{JuMoPoSpa}; related anisotropic case and  result in the flat torus  appear in \cite{KK2} and \cite{ADK, DK}, respectively. 

Our main result improves these results by upgrading the Hausdorff convergence to convergence in every $C^k$-norm. In particular, we prove that whenever the flow converges to a single ball, it becomes smooth after  finite time. Here is our main theorem.

\begin{theorem}\label{mainthm}
Let $\{E(t)\}_{t\geq 0}$ be a volume-preserving flat flow in $\R^{n+1}$, for $n\leq 2$, starting  from a bounded set of finite perimeter $E_0 \subset\R^{n+1}$ with volume $|E_0| = |B_r|$ and assume 
\[
\liminf_{t \to \infty} P(E(t)) <   2^{\frac{1}{n+1}} P(B_r). 
\]
Then there exists a time $T_0 >0$ such that the sets $E(t)$ are smooth for all $t \in (T_0,\infty)$, they solve  \eqref{eq:VMCF} in the classical sense  and converge exponentially fast to  $ B_r(x_0)$, for some $x_0 \in \R^{n+1}$,  in $C^k$ for every $k \in \N$, i.e., there are functions $u(\cdot, t) : \mathbb{S}^{n} \to \R$ such that for all $t >T_0$
\[
\partial E(t) = \{ \big(r+u(x,t)\big)x + x_0 : x \in \mathbb{S}^{n}  \} \quad \text{and} \quad \|u(\cdot,t)\|_{C^k( \mathbb{S}^{n})} \leq C_k e^{-c_k t}. 
\]
\end{theorem}

The condition on the perimeter  is needed to exclude bubbling, i.e., that the flow converges to a union of balls which may be tangent to each other. Although it sounds feasible that the flow generically converges to a one single ball, as it is the only stable critical point \cite{BDC}, we cannot  exclude the possibility of bubbling as the example in \cite{FJM} shows. We may, however, provide similar result as Theorem \ref{mainthm} in the case when the limiting set is a union of balls which have positive distance to each other. We state this as a remark as it follows from the same argument as Theorem \ref{mainthm}. 

\begin{remark}
\label{rem:bubbling}
Let $\{E(t)\}_{t\geq 0}$ be a volume-preserving flat flow in $\R^{n+1}$, $n \leq 2$, starting  from a bounded set of finite perimeter $E_0 \subset\R^{n+1}$ with volume $|E_0| = |B_r|$ which converges exponentially fast in $L^1$-sense to  $F = \bigcup_{i =1}^N B_\rho(x_i)$, for some $x_1,\dots, x_N\in \R^{n+1}$ with  $|x_i -x_j| >2\rho + \delta_0$ for $i \neq j$.  Then there exists time $T_0 >0$ such that the sets $E(t)$ are smooth for all $t \in (T_0,\infty)$,  they solve  \eqref{eq:VMCF} in the classical sense  and converge exponentially fast to  $F$ in $C^k$ for every $k \in \N$.
\end{remark}
 
We remark that in spite of the title \emph{eventual regularity},   Theorem \ref{mainthm}   differs from  the result by Cameron \cite{Cam} for the nonlocal mean curvature flow for graphs, since the latter holds for very rough initial sets. As the author points out, the result in \cite{Cam} does not hold for the mean curvature flow. We also mention the result by Otto-Schubert-Westdickenberg \cite{OSW} where the authors study the asymptotic behavior of the Mullins-Sekerka flow for graphs.

\subsection{Outline of the proof}
We begin by recalling Brakke’s $\eps$‑regularity theory for mean curvature flow \cite{Bra, KT,  Ton2, Ton}, which  roughly states that a Brakke flow that is trapped between two hyperplanes close to each other and has multiplicity close to one is locally regular. 
There is an elegant  recent proof for this   by DePhilippis-Gasparetto-Schulze \cite{DGS} and  Gasparetto \cite{Ga}   (the latter for varifolds with boundary), which is based on regularity theory for fully nonlinear PDEs  first introduced by Savin \cite{Sa} in the elliptic case,  and then adapted to the parabolic case  by Wang \cite{Wa}. If an analogous statement were available for our volume‑preserving flat flow, Theorem \ref{mainthm} would follow relatively directly from the a priori estimates in \cite{JuMoOrSpa}.

There are, however, three obstructions to applying this argument. First, we are not aware of any result which implies that a volume-preserving flat flow is a Brakke solution of the equation \eqref{eq:VMCF}. Neither do we know that the flat flow is a distributional solution of  \eqref{eq:VMCF}:  the result in \cite{MSS} is conditional because strong BV‑convergence of the minimizing‑movement scheme is not available in general (it is known only under mean‑convexity \cite{DPL}).  Second,  because we do not have the notion of Brakke solution, and since the equation \eqref{eq:VMCF} is nonlocal, we do not have Huisken's monotonicity formula  \cite{Hui2, MantegazzaBook}, which is crucial in the proof of Brakke regularity in \cite{DGS, Ga, KT,  Ton2}. The third issue is  the nonlocality  of the flow  caused by the Lagrange multipliers $\lambda(t) := \mathrm{\bar H}_{E_t} $ in \eqref{eq:VMCF}. The a priori estimates (see Proposition \ref{prop:apriori-est}) imply  that $\lambda(t)$ is only locally $L^2$-integrable in time which is  insufficient for standard parabolic regularity arguments as in \cite{KT}.

Starting from the Hausdorff convergence to a ball provided by \cite{JuMoOrSpa}, we first remove the obstruction posed by the Lagrange multiplier by adding a suitable artificial vertical motion to the flow that essentially cancels $\lambda(t)$. While conceptually simple, this manipulation introduces several technical complications. The more serious issue, the absence of a Brakke formulation for the flat flow, leads us to a different strategy: instead of passing to a varifold limit, we prove the required nonlinear PDE estimates directly for the discrete minimizing‑movement approximations, with constants independent of the time step. This direct, discretized PDE analysis is the most technically demanding part of the paper and may be of independent interest; accordingly, we present these estimates in full generality in  $\R^{n+1}$.

Finally, we overcome the lack of monotonicity formula in $\R^3$ by using again the  a priori estimates from \cite{JuMoOrSpa}, which show that the Willmore energy is asymptotically close to $4\pi$ for most times. By the classical Li–Yau inequality \cite{LiYa82}, an immersed surface in $\R^3$ with Willmore energy below $8\pi$ is embedded.  Combining our estimate with a quantitative version of the Li–Yau inequality (Proposition \ref{prop:kuwert-schaztle})  yields the following: by choosing a time $T_0$ at which the set
\[
\{ t >0 : \frac14  \|\mathrm{H}_{E(t)}\|_{L^2}^2 \leq 8 \pi - \delta_0  \}  
\]
has density close to one in a quantitative sense, we obtain a uniform bound away from multiplicity two at every spatial scale. We remark that this argument is purely three‑dimensional and does not extend to higher dimensions. The information above suffices to establish our version of flatness decay (Theorem \ref{thm:decay-flatness}). We treat the planar case separately, where the argument is simpler.

Theorem \ref{thm:decay-flatness}  provides spatial regularity of the flow at time $T_0$. We then use  the well‑posedness of the Cauchy problem for $C^2$-initial sets from \cite{JN2} to propagate smoothness to a  time interval $(T_0,T_0+\delta)$. Iterating this procedure gives full regularity for $(T_0,\infty)$, and exponential convergence in all $C^k$-norms follows by standard interpolation inequality.

\section{Set-up and statement of the flatness decay}

We  begin by introducing the relevant notation. 
We denote the open ball with radius $r$ centered at $x$ by $B_r(x)$ and by $B_r$ if it is centered at the origin. The closed ball is denoted by $\bar B_r(x)$.   If we need to emphasize the dimension of the ball we denote  $B_r^{n}(x) \subset \R^n$. The unit sphere is the boundary of the unit ball $\mathbb{S}^{n} = \partial B_1^{n+1}$. 
Given a unit vector $\omega \in \mathbb{S}^{n}$, we denote  the projection to its orthogonal plane as 
\begin{equation}
\label{def:projection}
\pi_\omega (x) = x - (x \cdot  \omega) \, \omega \quad \text{and} \quad \Pi_\omega = \{ x \in \R^{n+1} : x \cdot \omega = 0 \}.   
\end{equation}
Above $x\cdot y$ denotes the inner product between $x$ and $y$.  We often associate  $\Pi_\omega$ with $\R^n$ if there is no danger for confusion,    and therefore   $\pi_\omega (x)$  is often associated with  a vector in $\R^{n}$. For  $\omega \in \mathbb{S}^{n} $  we denote the cylinder in space tilted in direction $\omega$ as 
\begin{equation}
\label{def:cylinder-space}
C_{\rho,r}(\omega) = \{ x \in \R^{n+1}: |\pi_\omega (x)| <\rho, \,\, |x \cdot \omega| < r \}  
\end{equation}
and $C_{\rho,r}(x_0, \omega)= C_{\rho,r}(\omega) + x_0$. 
We denote the space-time cylinders by $Q_r = B_r \times (-r^2,r^2)$,  $Q_r^- = B_r \times (-r^2,0]$,  $Q_r^+ = B_r \times [0,r^2)$ and $\bar Q_r^- = \bar B_r \times [-r^2,0]$, while the  cylinders centered at $(x_0,t_0)$ are $Q_r(x_0,t_0) = Q_r + (x_0,t_0)$ and $Q_r^\pm(x_0,t_0) = Q_r^\pm + (x_0,t_0)$. If the size and the position of the cube is not relevant we denote $Q = Q_r(x_0,t_0)$ and $Q^\pm= Q_r^\pm(x_0,t_0)$. 

For a given set $E \subset \R^{n+1}$ we denote the distance function by $\text{dist}(\cdot, E)$ and the signed distance function by 
\begin{equation}
\label{def:sign-dist}
d_E(x) =  \text{dist}(x,E) - \text{dist}(x,\R^{n+1} \setminus E). 
\end{equation}
If $E$ is measurable we denote by $|E|$ its Lebesgue measure, $P(E)$ its perimeter, and $P(E;U)$ the perimeter relative to an open set $U$. We will mostly deal with regular sets and therefore it is enough to recall that for Lipschitz domains 
it holds $P(E) = \Ha^{n}(\partial E)$ and $P(E; U) = \Ha^{n}(\partial E \cap U)$.  If the set $E$ is $C^2$-regular,  the mean curvature $\mathrm{H}_E$  is the sum of the principal curvatures, while $B_E$ denotes the second fundamental form. We use the orientation given by the outer unit  normal which means that $\mathrm{H}_E$ is non-negative for convex sets. In the planar case, we also use the notation $\kappa_E = \mathrm{H}_E$. Finally we denote the Laplace-Beltrami operator as $\Delta_{\pa E}$ and the tangential gradient as $\nabla_{\pa E}$, which for functions $f:\R^{n+1} \to \R$ can be written as $\nabla_{\pa E} f = \big(I - \nu_E \otimes \nu_E\big) \nabla f$ on $\pa E$. 

For  a function $u : \partial E \to \R$ we define the $L^2$-norm $\|u\|_{L^2(\pa E)}$ as usual, but use the notation $\|u\|_{C^0(\pa E)}= \sup_{x \in \pa E}|u(x)|$ for the sup-norm. For a function $u: U \subset \R^n \to \R$ and $k \in \N$ we denote  the $C^k$-norm as $\|u\|_{C^k(U)}$ and the $C^{k+\alpha}$ norm for  $\alpha \in (0,1)$ by $\|u\|_{C^{k+\alpha}(U)}$, which are defined as usual  \cite{CC}.

\subsection{The flat flow solution} \label{sec:flatflow}

We begin by recalling the definition of the flat flow. This is first given in \cite{MSS}, following the associated scheme proposed in \cite{ATW, LS}, but here we use the variant from \cite{Vesa} since it simplifies the forthcoming analysis. We refer to  \cite{Vesa, MoPoSpa, MSS} for a more detailed introduction.     

We fix a small time step $h>0$, and  given a bounded set of finite perimeter $E$ with $|E| = v$,  we consider the minimization problem 
\beq
\label{def:min-prob}
\min  \Big{\{} P(F) + \frac{1}{h}\int_F  d_E \, dx : \,  |F| = v\Big{\}}
\eeq
 and note that the minimizer exists but might not be unique. Above $d_E$ is the signed distance function defined in \eqref{def:sign-dist}. 
We note that by a simple scaling argument we may reduce to the case $v = |B_1|$ by changing the value of $h$.

Let $E_0 \subset \R^{n+1}$  be a bounded set of finite perimeter with $|E_0|= v$ and which coincides with its Lebesgue representative. 
We construct discrete-in-time evolution $\{E_k^{h}\}_{k}$, $k =0,1, \dots$ recursively  such  that $E_0^h=E_0$ and, assuming that $E_{k-1}^{h}$ is defined, we set $E_{k}^{h}$ to be a minimizer of \eqref{def:min-prob} with $E= E_{k-1}^{h}$. Notice that by the standard regularity theory $\partial E_k^{h}$ is $C^{2+\alpha}$-regular for all $\alpha\in (0,1)$ outside a small singular set \cite{MSS, Vesa} and therefore in \eqref{def:min-prob} we use this exact representative in order to compute $d_{E_{k-1}^{h}}$. We define the {\em approximate volume-preserving flat flow}  $\{E^{h}(t)\}_{t \geq 0}$ by setting $t_k = kh$ and
\begin{equation}
 \label{def:app-flat-flow} 
E^{h}(t) = E_k^{h} \qquad \text{for }\, t \in [t_k,  t_{k+1}).
\end{equation}
In view that $\partial E^{h}(t)$ is $C^{2+\alpha}$-regular, it satisfies the Euler-Lagrange equation
 \beq
\label{eg:Euler-Lag}
\frac{d_{E^{h}(t -h) }}{h} = - \mathrm{H}_{E^{h}(t) } + \lambda^{h}(t)  \qquad \text{on }\, \partial E^{h}(t) ,
\eeq
in  classical sense outside the singular set, where $\lambda^h(t) =  \lambda^{h}(t_k)$ for $t \in [t_k,  t_{k+1})$ is the Lagrange multiplier due to the volume constraint. We now recall the definition of {\em flat flow}.
\begin{definition}\label{def:flat-flow}
A flat flow solution of \eqref{eq:VMCF}, or volume-preserving  flat flow,  is any  family of sets $\{E(t)\}_{t \geq 0}$ which is a cluster point of $\{E^{h}(t)\}_{t \geq 0}$ defined in \eqref{def:app-flat-flow}, i.e., 
\[
E^{h_n}(t) \to E(t) \quad \text{as } \, h_n \to 0 \quad \text{in } \, L^1 \quad \text{for almost every }\, t >0 .
\]
\end{definition}
By \cite[Theorem 1]{Vesa} there exists a flat flow starting from $E_0$.  We recall the following a priori estimates for the volume-preserving  flat flow, which can be found in \cite{MSS, Vesa} and in  \cite[Propositions 4.1 and 4.2]{JN}. 

\begin{proposition}
\label{prop:apriori-est}
Let  $\{E^{h}(t)\}_{t \geq 0}$ be  an approximate volume-preserving flat flow starting from a bounded set of finite perimeter $E_0 \subset \R^{n+1}$ with $|E_0| \geq c_0 >0$ and $P(E_0) \leq C_0$. There is a constant $C$, which depends on $c_0,C_0$ and the dimension such that the following hold whenever $h$ is small enough. 
\begin{itemize}
\item[(i)] \,For every $t>h$ and $x \in \partial E^{h}(t)$ it holds $|d_{\partial E^{h}(t -h)}(x)| \leq C \sqrt{h}$ and the mean curvature and the Lagrange multiplier in the Euler-Lagrange equation \eqref{eg:Euler-Lag} satisfy
\[
\|\mathrm{H}_{E^{h}(t) }\|_{C^0} + |\lambda^{h}(t)| \leq \frac{C}{\sqrt{h}}. 
\] 
 
\item[(ii)] \,(Dissipation inequality)\,    For every   $T_2 \geq T_1 +h > 2h$ it holds
\[
\int_{T_1+h}^{T_2} \|\mathrm{H}_{E^{h}(t) } - \lambda^{h}(t)\|_{L^2}^2 \, dt \leq C \big( P(E^{h}(T_1)) -P( E^{h}(T_2))\big) .
\]

\item[(iii)] \, The Lagrange multipliers satisfy, for every $T_2>T_1>h$,  
\[
 \int_{T_1}^{T_2}  \lambda^{h}(t)^2 \, dt \leq C(T_2-T_1 +1). 
\]
As a consequence it holds  $ \int_{s}^{t}  |\lambda^{h}(\tau)| \, d\tau \leq C\sqrt{t-s}$ for all $0 <s <t<1$.  

\end{itemize}  
\end{proposition}

An important consequence of Proposition \ref{prop:apriori-est} (i) is that by dilatating with  factor  $\frac1{\sqrt{h}}$, the sets have bounded mean curvature. This means that we may use elliptic estimates when we are in a scale smaller than $\sqrt{h}$.

\subsection{Flatness decay}

As we explained in the introduction, our aim is to prove regularity  only in space at certain fixed times. We are going to prove directly $C^2$-regularity in space in order to use the regularity result for the Cauchy problem for $C^2$-initial sets from \cite{JN2}.
Therefore we need to prove the flatness decay directly  for second order Taylor approximation. We will prove many of the result in the general case $\R^{n+1}$ and therefore, even if Theorem \ref{mainthm} holds only in low dimensions, we will keep the notation for general dimension whenever we can. To this aim, given a  symmetric matrix $A \in S^{n}$ and $c \in \R$, we define the associated 
caloric polynomial of second order as
\begin{equation}
\label{def:caloric}
P(x,t) = P_{A,c}(x,t) := \frac12  Ax \cdot x  + b t + c, \qquad \text{where } \, b = \text{Tr}(A).  
\end{equation}
It is crucial that  in our definition the caloric polynomial has no spatial linear term. This simplifies the forthcoming analysis as the volume-preserving  mean curvature flow then linearizes to the heat equation with a forcing term. 

Because of the Lagrange multipliers, for a fixed $t_0 \in (0,\infty)$, it is natural to  define function  $\Lambda(\cdot):[0,\infty) \to \R$,   
\begin{equation}
\label{def:Lambda}
\Lambda(t) := \int_{t_0}^t \lambda^h(\tau+h) \, d \tau,  
\end{equation}
with the understanding that for $t<t_0$, $\int_{t_0}^t \lambda^h(\tau+h) \, d \tau=-\int_t^{t_0} \lambda^h(\tau+h) \, d \tau$. The crucial  technical properties are $\Lambda(t_0)= 0$ and $\Lambda(t_k) - \Lambda(t_{k-1}) = \lambda^h(t_k)h$ and we use these throughout the paper often  without further mention.  The issue with the Lagrange multipliers is that due to the a priori estimate from Proposition \ref{prop:apriori-est} (iii),   $\Lambda(\cdot)$ is merely $\frac12$-H\"older continuous, which is not enough to prove even $C^{1+\alpha}$  bounds, since by the standard parabolic scaling, this would imply $C^{\frac{1+\alpha}{2}}$-regularity in time. Therefore we need to cancel the movement created by the Lagrange multipliers. 


In order to measure the flatness in the scale of $C^{2+\alpha}$-regularity  in space, we first fix the center $x_0 \in \R^{n+1}$ and  a direction  $ \omega \in   \mathbb{S}^{n} $, and then  consider a caloric polynomial $P_{A,c}(\cdot, t) : \Pi_\omega \to \R$ defined in \eqref{def:caloric} where the linear space $\Pi_\omega $ is defined in \eqref{def:projection}. We  define  the subgraph of  $P_{A,c}$ shifted by $\Lambda(t)$, defined in  \eqref{def:Lambda}, as 
  \begin{equation}
\label{def:set-caloric}
\textbf{P}_t := \big{\{} x \in \R^{n+1} : (x-x_0) \cdot \omega  < P_{A,c}\big(\pi_\omega (x-x_0), t \big)  + \Lambda(t) \big{\}}. 
\end{equation}
We define the second order  Taylor approximation, or excess,  at time $t$ as  
\begin{equation} \label{def:excess}
\begin{split}
\mathcal{E}_{A,\omega,c}&\big(E^h(t); r, r_1 \big) \\
&:= \sup_{ x \in (E^h(t)   \Delta \textbf{P}_t)  \cap C_{r,r_1}(x_0, \omega)}  \big| (x-x_0) \cdot \omega  - P_{A,c}\big(\pi_\omega (x-x_0), t\big) -\Lambda(t)   \big| , 
\end{split}
\end{equation}
where $\textbf{P}_t$ is   defined in \eqref{def:set-caloric} and the cylinder  $C_{r,r_1}(x_0, \omega)$  in \eqref{def:cylinder-space}. 

Recall that $\Lambda(\cdot)$ is $\frac12$-H\"older continuous, and we always assume without further mention that  $r_1 \geq Cr$ where $C$ is large enough in order to guarantee $\sup_{t \in [t_0 -r^2,t_0]}|\Lambda(t)| \leq \frac{r_1}{2}$. This ensures that the set \eqref{def:set-caloric} does not intersect the top or the bottom of  the cylinder   $C_{r,r_1}(x_0, \omega)$.

The core of the paper is the flatness decay result, which we state shortly. We will of course assume that the excess \eqref{def:excess} is small for all $t \in (t_0-r^2, t_0]$. Usually one also needs to assume that the multiplicity in large scales is less than two, which one then obtains in every scale using monotonicity formula. As already mentioned, we overcome  the lack of  monotonicity formula in $\R^3$ by assuming the following bound on the Willmore energy
\begin{equation} \label{eq:density-point}
\inf_{0 < r <1} \frac{1}{r^2} \Big|\big{\{} t \in [t_0-r^2,t_0] : \frac14 \| \mathrm{H}_{E^{h}(t) }\|_{L^2}^2 \leq 8 \pi - \delta_0 \big{\}} \Big| \geq 1 - \eps_0. 
\end{equation}
The estimate \eqref{eq:density-point} means that in every scale $r$ near $t_0$, for most of  $t $ the Willmore energy of  $E^{h}(t)$ is less than $8 \pi$. We actually show in the next section that there is $T_0$ such that, up to a set with small measure,   all $t_0 \geq T_0$ satisfy the condition \eqref{eq:density-point} with $5 \pi$ instead of $ 8 \pi - \delta_0$. We choose to keep the  above form, because from there it is  clear that the assumption is closely related to the Li-Yau inequality and the result is stronger. We explain in the next section how \eqref{eq:density-point} is related to the multiplicity assumption. We remark that the notation $t_0$ in \eqref{eq:density-point} stands for a generic time, and is not related to the notation $t_k = kh$ for $k =1,2,\dots$ for the time steps in the minimizing movement scheme.

In the planar case, the a priori estimates given by Proposition \ref{prop:apriori-est} (ii) are much stronger and we replace \eqref{eq:density-point} by stronger condition
\begin{equation} \label{eq:density-point-pl}
\begin{split}
&P(E^h(t_0-1)) \leq (2\pi  \sqrt{2}-\delta_0)r \qquad \text{and} \\
&\inf_{0 < r <1} \frac{1}{r^2} \Big|\big{\{} t \in [t_0-r^2,t_0] : \| \kappa_{E^{h}(t) }  - \bar \kappa_{E^{h}(t)} \|_{L^2} \leq \eps_0 \big{\}} \Big| \geq 1 - \eps_0, 
\end{split}
\end{equation}
where $|E^h(t)|= |B_r|$. The condition \eqref{eq:density-point-pl} is strong since  \cite[Proposition 2.1]{JuMoPoSpa} implies that if a $C^2$-regular set $E$ with $|E| = |B_r|$ satisfies $P(E) \leq (2\pi  \sqrt{2}-\delta_0)r$  and  $\| \kappa_{E}  - \bar \kappa_{E} \|_{L^2} \leq \eps_0$ then it is $C^{1+\alpha}$-close to $B_r(x_0)$ for some $x_0 \in \R^2$  and 
\begin{equation} \label{eq:density-point-pl2}
\partial E = \{ (r +g(x))x  + x_0 : x \in \mathbb{S}^1 \} \qquad \text{with } \, \|g\|_{C^{1}( \mathbb{S}^1 )} \leq C \eps_0^\alpha.  
\end{equation}

Here is the statement of the flatness decay. We postpone the proof to Section 5.

\begin{theorem}
\label{thm:decay-flatness}
Let $\{ E^h(t)\}_{t \geq 0}$ be an approximative flat flow in $\R^{n+1}$ with $n \leq 2$, and fix $(x_0,t_0) \in \pa E^h(t_0)$, $t_0 \geq 1$, $C_1 \geq 1$ and small $\alpha,  \delta_0 >0$.  There are   $\sigma \in (0,1)$ and small $r_0, \eps_0>0$ such that assuming  \eqref{eq:density-point} if $n=2$ or    \eqref{eq:density-point-pl} if $n=1$,   
and 
\begin{equation} \label{eq:flatt-assumption}
\mathcal{E}_{A,\omega,c}\big(E^h(t_k); r,r_1 \big) \leq r^{2+\alpha} \qquad \text{for all } \, t_k = kh \in (t_0-r^2, t_0]
\end{equation} 
for $C_0 \sqrt{h} \leq  r \leq r_0 $ for some $C_0$ and  $(A,\omega,c) \in  S^n  \times \mathbb{S}^n \times  \R$ with  $|A|, |c| \leq C_1$,  then there are $(A', \omega',c') \in S^n \times  \mathbb{S}^n \times \R$ such that 
\[
\mathcal{E}_{A',\omega',c'}\big(E^h(t_k);  \sigma r, r_1 \big) \leq ( \sigma r)^{2+\alpha}   \qquad \text{for all } \, t_k =kh  \in (t_0-\sigma^2 r^2, t_0]
\]
and $|A-A'| \leq c_0 r^\alpha$, $|\omega-\omega'| \leq c_0 r^{1+\alpha}, |c-c'|\leq c_0 r^{2+\alpha}$. Here $c_0 $ is a number that does not depend on any  parameter. 
\end{theorem}

When we iterate the estimate given by  Theorem  \ref{thm:decay-flatness}, we roughly speaking obtain  $C^{2+\alpha}$-regularity for $E^h(t_0)$ in space up to scale of order $\sqrt{h}$. This is the key estimate and we obtain the $C^{2}$-regularity in the ball $B_{\sqrt{h}}(x_0) \times \{t_0\}$ 
using this together with  regularity  estimates for minimal surfaces. 

\subsection{The sub- and supergraphs} 

In the definition of the excess \eqref{def:excess}  it is more convenient to deal with functions than sets. To this aim we assume that $\omega \in \mathbb{S}^{n}$ is the unit vector  in  the excess  \eqref{def:excess} and we may choose the coordinates such that $\omega = e_{n+1}$. By translating the coordinates we may also assume that $x_0 =0$ and we usually denote the point in space as $(x',x_{n+1}) \in \R^{n+1}$  or simply  $(x,x_{n+1}) \in \R^{n+1}$ for $x \in \R^{n}$ if the meaning is clear from the context, and denote the cylinder as $C_{r,r_1} = \{(x,x_{n+1}) \in \R^{n+1} : |x|<r, \,\, |x_{n+1}|<r_1\}$.  Given the set $E^h(t)$ we define the  supergraph $u_+(\cdot, t):  B_r^{n} \to \R$ and the subgraph  $u_-(\cdot, t):  B_r^{n} \to \R$ as
\begin{equation}
\label{def:sub-super}
\begin{split}
&u_+(x, t) := \max \{ x_{n+1}  : (x,x_{n+1}) \in  \partial  E^h(t)\cap C_{r,r_1}\}  \qquad \text{and}\\ 
&u_-(x, t) := \min \{ x_{n+1}  : (x,x_{n+1}) \in  \partial  E^h(t)\cap C_{r,r_1}  \}. 
\end{split}
\end{equation}
Clearly  $u_-(\cdot, t)$ is lower semicontinuous and  $u_+(\cdot, t)$ is upper semicontinuous. 

Let us  assume that $t_0 = k_0 h$ for some $k_0 \in \N$. We rescale the coordinates as 
\begin{equation}
\label{rescale}
y = \frac{x}{r}, \quad \mathfrak{t} = \frac{t-t_0}{r^2}, \quad \mathfrak{h}  = \frac{h}{r^2} \quad \text{and} \quad \mathfrak{t}_k =  (k-k_0) \mathfrak{h}   = \frac{t_k- t_0}{r^2} 
\end{equation}
and define the functions associated with  the excess \eqref{def:excess} as  $v_r^-, v_r^+ : Q_1^-  \to \R$,
\begin{equation}
\label{def:v-r}
v_r^{\pm}(y, \mathfrak{t}_k) := \frac{u_{\pm}(r y, r^2 \mathfrak{t}_k +t_0)-  P(r y ,r^2 \mathfrak{t}_k +t_0 ) - \Lambda(r^2 \mathfrak{t}_k +t_0) }{r^{2+\alpha}}
\end{equation}
and $v_r^{\pm}(y, \mathfrak{t}) = v_r^{\pm}(y, \mathfrak{t}_k)$ for $\mathfrak{t} \in [\mathfrak{t}_k, \mathfrak{t}_{k+1})$, where  $\Lambda(\cdot)$ is  defined  in \eqref{def:Lambda}  and  $P(\cdot)$ is a caloric polynomial as in \eqref{def:caloric}.  
By  \eqref{def:excess},  the assumption \eqref{eq:flatt-assumption} can be written as 
\begin{equation}
\label{eq:restate-flatness}
\|v_r^+\|_{C^0(Q_1^-)} \leq 1 \quad \text{and} \quad \|v_r^-\|_{C^0(Q_1^-)} \leq 1.  
\end{equation}
We stress that the caloric polynomial is the same in the definitions of $v_r^+$ and $v_r^-$.  Throughout the paper we assume that the scaling in time is such that $ \fh$ is small, which means that $r\geq C_0 \sqrt{h}$,  and that  $\alpha \in (0,\tfrac18)$.

\section{Preliminary results}

In this section we collect relevant results  that we need in the proof of Theorem \ref{mainthm} which  are not related to PDEs. We begin by relating the assumption \eqref{eq:density-point} with the multiplicity, or to be more precise, with  the density of the boundary. This is given by the following quantitative version of the  Li-Yau inequality inspired by  \cite[Lemma 1.4]{Sim2}. 
\begin{proposition}
\label{prop:kuwert-schaztle}
Fix $\alpha, \delta \in (0,1)$ and let $x_0 \in \R^3$ and $\omega \in \mathbb{S}^2$. There is $\rho_0 = \rho_0(\alpha, \delta)>0$ such that for all $\rho \leq \rho_0$ and  for every  $C^2$-regular bounded set $E \subset \R^3$ which satisfies
\[
\frac14 \int_{\pa E}  \mathrm{H}_{E}^2 \, d \Ha^2  \leq 8 \pi - \delta \qquad \text{and} \qquad \pa E \cap B_{2\rho}(x_0) \subset \{ x \in \R^3 : |(x-x_0)\cdot \omega| \leq \rho^{1+\alpha}\}
\]
it holds 
\[
\Ha^2\big( \pa E \cap B_\rho(x_0) \big) \leq \big(2\pi -\frac{\delta}{6}\big)  \rho^2.  
\]
\end{proposition}

\begin{proof}
We first note that since our surface $\pa E$ is a boundary of a regular set and $\nu_E$ denotes the outer unit normal, we define  the mean curvature vector as $H_{E}\nu_E$ and  recall the well-known fact that 
\begin{equation}
\label{prop:willmore1}
-\Delta_{\pa E}  x =  H_E \nu_E \qquad \text{on } \, \pa E. 
\end{equation} 

Without loss of generality we assume $x_0 = 0$ and $\omega = e_3$.  For $0< r <R$ we recall the following consequence of the monotonicity formula from \cite[formula (1.2)]{Sim2} (see also \cite[Appendix (A.3)]{KS01}), which  in our notation reads as
\begin{equation}
\begin{split}
&\frac{\Ha^2\big( \pa E \cap B_r \big)}{r^2}   - \frac{1}{2r^2} \int_{\pa E \cap B_r} H_E (x \cdot \nu_E) \, d \Ha^2  \\ 
&\,\,\,\,\,\,\,\,\leq \frac{\Ha^2\big( \pa E \cap B_R \big)}{R^2} +  \frac{1}{16} \int_{\pa E\cap B_R}  \mathrm{H}_{E}^2 \, d \Ha^2- \frac{1}{2R^2} \int_{\pa E \cap B_R} H_E (x \cdot \nu_E) \, d \Ha^2.  
\end{split}
\end{equation} 
Letting $R \to \infty$, we obtain from the fact that $E$ is bounded  and from the bound on the Willmore energy that 
\begin{equation}
\label{prop:willmore2}
\frac{\Ha^2\big( \pa E \cap B_r \big)}{r^2}   - \frac{1}{2r^2} \int_{\pa E \cap B_r} H_E (x \cdot \nu_E) \, d \Ha^2 \leq   \frac{1}{16} \int_{\pa E}  \mathrm{H}_{E}^2 \, d \Ha^2 \leq 2\pi -\frac{\delta}{4}. 
\end{equation} 
Using \eqref{prop:willmore2} for $r = \rho$, it is thus enough to prove that 
\begin{equation}
\label{prop:willmore3}
\frac{1}{\rho^2} \int_{\pa E \cap B_\rho} |H_E (x \cdot \nu_E)| \, d \Ha^2 \leq \frac{\delta}{6} \qquad \text{for }\,\rho \leq \rho_0
\end{equation} 
in order to conclude the proof. 

Using Cauchy-Schwarz inequality we obtain from  \eqref{prop:willmore2}   with $r = 2 \rho$  that $\Ha^2\big( \pa E \cap B_{2\rho} \big) \leq C\rho^2$.  Let us denote $\nu_3 = \nu_E \cdot e_3$ and $x_3 = x \cdot e_3$ for short. Then \eqref{prop:willmore1} implies 
\[
-\Delta_{\pa E}  x_3 =  H_E \nu_3.
\]  
We choose a cut-off function $\zeta \in C_0^1(B_{2 \rho})$ such that $\zeta = 1$ in $B_\rho$, $0\leq \zeta \leq 1$ and $|\nabla \zeta |\leq 2\rho^{-1}$.  Multiply the above equation by $x_3 \zeta^2$ and integrate by parts to obtain
\[
\int_{\pa E} |\nabla_{\pa E} x_3|^2 \zeta^2 \, d\Ha^2 \leq \int_{\pa E} | H_E  x_3 \nu_3| \zeta^2 \, d \Ha^2 + 2\int_{\pa E}|x_3|  |\nabla_{\pa E} x_3| \zeta|\nabla_{\pa E}  \zeta|  \,d \Ha^2. 
\]
Using Young's inequality $2|x_3|  |\nabla_{\pa E} x_3| \zeta|\nabla_{\pa E}  \zeta| \leq  \frac12  |\nabla_{\pa E} x_3|^2 \zeta^2 + 2|\nabla_{\pa E}  \zeta|^2 x_3^2 $ we deduce 
\begin{equation}
\label{prop:willmore4}
\int_{\pa E} |\nabla_{\pa E} x_3|^2 \zeta^2 \, d\Ha^2 \leq 2 \int_{\pa E} | H_E  x_3 \nu_3| \zeta^2 \, d \Ha^2 + 4\int_{\pa E}|\nabla_{\pa E}  \zeta|^2 x_3^2  \, d\Ha^2. 
\end{equation} 
Recall that $\zeta = 0$ outside $B_{2\rho}$ and the assumption yields $|x_3| \leq \rho^{1+\alpha}$ on $B_{2\rho}\cap \pa E$. Therefore we may estimate using $\Ha^2\big( \pa E \cap B_{2\rho} \big) \leq C\rho^2$ 
\[
 \int_{\pa E} | H_E  x_3 \nu_3| \zeta^2 \, d \Ha^2 \leq \rho^{1+\alpha}  \|H_E\|_{L^2(\pa E)} \Ha^2(\pa E \cap B_{2\rho})^{\frac12} \leq C   \rho^{2+\alpha}
\]
and  then by $|\nabla \zeta |\leq 2\rho^{-1}$ we have $\int_{\pa E}|\nabla_{\pa E}  \zeta|^2 x_3^2  \, \Ha^2 \leq C \rho^{2+2\alpha}$. Note also that on $\pa E$ it holds  $|\nabla_{\pa E} x_3|^2 = 1 -\nu_3^2$. Therefore we have by \eqref{prop:willmore4}
\begin{equation}
\label{prop:willmore5}
 \int_{\pa E\cap B_\rho} (1-\nu_3^2) \, d \Ha^2 \leq C  \rho^{2+\alpha}. 
\end{equation} 

We obtain \eqref{prop:willmore3} from \eqref{prop:willmore5} as follows.  Write  $x' = x - x_3e_3$ and $\nu' = \nu_E - \nu_3e_3$ and estimate
\[
\int_{\pa E \cap B_\rho} |H_E (x \cdot \nu_E)| \, d \Ha^2 \leq  2 \|H_E\|_{L^2(\pa E)} \left( \int_{\pa E \cap B_\rho} |x'\cdot \nu'|^2 + |x_3\nu_3|^2 \, d \Ha^2 \right)^\frac12.  
\]
Since $|x_3| \leq \rho^{1+\alpha}$ on $B_{\rho}\cap \pa E$ we have  $\int_{\pa E \cap B_\rho} |x_3\nu_3|^2 \, d \Ha^2 \leq C  \rho^{4+2\alpha}$. On the other hand, since $|\nu'|^2 = 1 - \nu_3^2$ we have by \eqref{prop:willmore5} that 
\[
\int_{\pa E \cap B_\rho} |x'\cdot \nu'|^2 \, d \Ha^2 \leq  \rho^2 \int_{\pa E \cap B_\rho} |\nu'|^2 \, d \Ha^2 \leq C  \rho^{4+\alpha}.
\]
In conclusion, we have 
\[
\int_{\pa E \cap B_\rho} |H_E (x \cdot \nu_E)| \, d \Ha^2 \leq C  \rho^{2+\frac{\alpha}{2}},
\]
which implies \eqref{prop:willmore3} when $C\rho^\frac{\alpha}{2} \leq \frac{\delta}{6}$.   
\end{proof}

We may thus avoid the use of (parabolic) monotonicity formula  by assuming that we may choose a time $t_0$ such that \eqref{eq:density-point} holds 
for a small $\eps_0>0$, whose choice will be clear later.  If we assume \eqref{eq:density-point},  there is 
a set $I_{\eps_0} \subset [t_0 -1 ,t_0]$ such that for all $r\in (0,r_0)$
\begin{equation} \label{eq:density-point-20}
 | I_{\eps_0} \cap [t_0-r^2,t_0]| \geq (1 - \eps_0)r^2
\end{equation}
 and for all $t \in I_{\eps_0}$ it holds 
\begin{equation} \label{eq:density-point-2}
\frac14 \int_{\pa E^h(t)}  \mathrm{H}_{E^h(t)}^2 \, d \Ha^2  \leq 8 \pi - \delta_0.
\end{equation}
This means that we have the first assumption of Proposition \ref{prop:kuwert-schaztle}. The second follows from the assumption on the excess  \eqref{eq:flatt-assumption} for suitable radii, which will be clear later  in the proof of Lemma \ref{lem:key-lemma}. 

 In the planar case, if we assume \eqref{eq:density-point-pl} and $|E^h(t_0)| = |B_r|$, then by \eqref{eq:density-point-pl2} there is a set $I_{\eps_0} \subset [t_0 -1 ,t_0]$ which satisfies \eqref{eq:density-point-20}   and for all $t \in I_{\eps_0}$ it holds 
\begin{equation} \label{eq:density-point-pl3}
\partial E^h(t) = \{ (r +g(x,t))x  + x_t : x \in \mathbb{S}^1 \} \qquad \text{with } \, \|g(\cdot, t)\|_{C^{1}( \mathbb{S}^1 )} \leq C \eps_0^\alpha.  
\end{equation}

The problem is to find such times $t_0$ for which the assumption \eqref{eq:density-point} for $n=2$ or \eqref{eq:density-point-pl} for $n=1$ holds. 
In order to prove this, we use  the exponential convergence proven in \cite{JuMoOrSpa} in the case $\R^3$ and \cite{JuMoPoSpa} in $\R^2$. We state the result for the approximative flat flow, since this is in fact what the  proofs in \cite{JuMoOrSpa, JuMoPoSpa} provide. 

\begin{proposition}
\label{prop:JuMoOrSpa}
Let $\{E(t)\}_{t\geq 0}$ be a volume-preserving flat flow in $\R^{n+1}$ with $n \leq 2$ as in the statement of Theorem \ref{mainthm}  and assume  $\{E^{h_n}(t)\}_{t\geq 0}$ is the associated  appriximative flat flow converging to it. Then there is $x_0 \in \R^{n+1}$ such that 
\[
\sup_{x \in E^{h_n}(t) \Delta B_r(x_0)} \text{dist}(x, \partial B_r(x_0)) + \big|P(E^{h_n}(t)) - P(B_r) \big| \leq C e^{-c_1t}
\]   
for $C\geq 1$,  $c_1>0$ and $h_n$ small. Moreover, it holds 
\[
\int_{T}^\infty \|\mathrm{H}_{E^{h_n}(t) } - \lambda^{h_n}(t)\|_{L^2}^2 \, dt \leq C e^{-c_1T}. 
\]
\end{proposition}

\begin{proof}
We may assume that $|E(t)| = |B_1|$ and we denote $h = h_n$.  By the result in \cite{JN} there are point $x_1(t), \dots, x_N(t) \in \R^{n+1}$ such that for $F_t = \bigcup_{i = 1}^N B_{\rho}(x_i(t))$  with $|x_i(t) - x_j(t)| \geq 2 \rho$ and $|F_t|= |B_1|$ it holds 
\[
\lim_{t \to \infty} |E(t) \Delta F_t|= 0. 
\]
Since $P(F_t)= N^{\frac{1}{n+1}}P(B_1)$ we have by the lower semicontinuity of the perimeter that 
\[
\liminf_{t \to \infty}P(E(t)) \geq   N^{\frac{1}{n+1}}P(B_1).
\]
 The assumption in Theorem \ref{mainthm}  then implies that $N=1$.   We use the estimate (4.32) in the proof of \cite[Theorem 1.1]{JN}, which implies that for a fixed small $\delta>0$ there are $h_0>0$ and $t_0>0$ such that 
\[
\big| P(E^{h}(t_0)) - P(B_1) \big| \leq \delta
\]
for all $h \leq h_0$. Therefore we may use \cite[Theorem 1.2]{JuMoPoSpa} and  \cite[Theorem 1.2]{JuMoOrSpa}, which are stated for the limiting family $\{ E(t)\}_{t \geq 0}$ but from the proof it is clear that the result holds also for the approximative flat flow $\{ E^h(t)\}_{t \geq 0}$. This  implies the first estimate. The second estimate follows from the first and from the dissipation inequality in Proposition \ref{prop:apriori-est} (ii). 
\end{proof}

Let us focus on the $\R^3$ case for a moment. We  recall the following estimates for the Willmore energy which can be found in  \cite[Corollary 3.2]{JuMoOrSpa}. 
\begin{lemma}
\label{prop:JuMoOrSpa2} 
Let $\delta_0 >0$. There exists $q \in(0,1)$ and a constant $C$ such that for every $C^2$-regular set $E \subset \R^3$ with volume $|E| = |B_1|$ and $P(E) \leq 4 \pi \, \sqrt[3]{2} -\delta_0$ it holds 
\[
\big| \bar{\mathrm{H}}_{E} - 2 \big|  \leq C \|\mathrm{H}_{E}  -  \bar{\mathrm{H}}_{E}\|_{L^2}^q 
\]
and the Willmore energy satisfies
\[
\frac14 \int_{\pa E}  \mathrm{H}_{E}^2 \, d \Ha^2 \leq 4 \pi + C \|\mathrm{H}_{E}  - \bar{\mathrm{H}}_{E}\|_{L^2}^q . 
\]
\end{lemma}

We proceed by recalling  the trivial fact that for every $\lambda \in \R$ it holds
\[
 \|\mathrm{H}_{E}  - \bar{\mathrm{H}}_{E}\|_{L^2} \leq \|\mathrm{H}_{E} - \lambda \|_{L^2}. 
\]
Therefore we deduce from Proposition \ref{prop:JuMoOrSpa}   that for $T>1$ the set
\begin{equation} \label{def:gammaT}
\Gamma_T :=  \{ t \in [T,\infty) : \|\mathrm{H}_{E^{h_n}(t) } -  \lambda^{h_n}(t) \|_{L^2}^2 \geq e^{-\frac{c_1}{2}T} \}
\end{equation}
has measure $|\Gamma_T| \leq C e^{-\frac{c_1}{2}T}$ and by  Lemma \ref{prop:JuMoOrSpa2} it holds for $T$ large enough
\begin{equation} \label{def:gammaT-2}
\frac14 \int_{\pa E^{h_n}(t)}  \mathrm{H}_{E^{h_n}(t)}^2 \, d \Ha^2 \leq 4 \pi + C e^{-\frac{c_1 q}{4}T} < 5 \pi  \qquad \text{for every } \, t \in [T,\infty)\setminus \Gamma_T,
\end{equation}
when $|E^{h_n}(t)| = |B_1|$.  Similarly we may estimate the Lagrange multipliers, again by Lemma \ref{prop:JuMoOrSpa2}, for every $t \in [T,\infty)\setminus \Gamma_T$  as
\begin{equation} \label{def:gammaT-3}
\begin{split}
| \lambda^{h_n}(t)  - 2 | &\leq | \bar{\mathrm{H}}_{E^{h_n}(t)}  -2 | +  | \bar{\mathrm{H}}_{E^{h_n}(t)} - \lambda^{h_n}(t) |  \\
&\leq | \bar{\mathrm{H}}_{E^{h_n}(t)}  -2 | + C \|\mathrm{H}_{E^{h_n}(t)} - \lambda^{h_n}(t)\|_{L^2} \leq  C e^{-\frac{c_1 q}{4}T}.  
\end{split}
\end{equation}

In order to guarantee that we find  times for which \eqref{eq:density-point} (or \eqref{eq:density-point-pl} in the planar case)  holds,  we  use the previous estimates and the following measure theoretical lemma, perhaps well-known, which gives a quantitative  size  of the set of points with positive density.  
\begin{lemma}
\label{lem:lebesgue}
Assume that a measurable set $\Gamma \subset \R$ satisfies $0 <|\Gamma| <1 $. Define the  set $\Sigma $ such that
\[
\Sigma = \Big{\{} x \in \R : \sup_{r \in (0,1)} \frac{\big| [x-r, x] \cap \Gamma \big|}{r} \geq \sqrt{|\Gamma|} \Big{\}}.
\]
Then it holds 
\[
|\Sigma| \leq  C \sqrt{|\Gamma|}
\]
for a  constant $C\geq 1$. 
\end{lemma}

\begin{proof}
Denote $\sigma = |\Gamma| >0$. Denote also  $I_r(x) = [x-r, x+r]$ and the centered maximal function by
\[
M \chi_\Gamma(x) = \sup_{r >0} \frac{1}{2r} \int_{I_r(x)}  \chi_\Gamma(y) \, dy.  
\]
By the weak type estimate (see e.g. \cite[Section 7, p. 60]{CC}) it holds 
\begin{equation} \label{eq:weak-type}
\big| \{x \in \R :  M \chi_\Gamma(x) > \lambda \}\big| \leq \frac{C}{\lambda} \| \chi_\Gamma\|_{L^1} =  \frac{C}{\lambda}  \sigma
\end{equation}
for all $\lambda>0$ and for a constant $C$. We use \eqref{eq:weak-type} with $\lambda = \frac{\sqrt{\sigma}}{4}$ and have 
\begin{equation} \label{eq:weak-type-1}
\big| \{x \in \R :  M \chi_\Gamma(x) >\frac{\sqrt{\sigma}}{4} \}\big| \leq  4C \sqrt{\sigma}. 
\end{equation}

Let us show that 
\begin{equation} \label{lem:measure-lemma-1}
\Sigma  \subset  \{x \in \R :  M \chi_\Gamma(x) >\frac{\sqrt{\sigma}}{4} \}.  
\end{equation}
Fix $x \in \Sigma$. By the definition of $\Sigma$ there is $r \in (0,1)$  such that $r^{-1}\big| [x-r, x] \cap \Gamma \big| > \tfrac12 \sqrt{\sigma}$. Then it holds 
\[
\frac{\sqrt{\sigma}}{2} < \frac{\big| [x-r, x] \cap \Gamma \big|}{r} \leq  \frac{1}{r}  \int_{I_r(x)}  \chi_\Gamma(y) \, dy \leq 2    M \chi_\Gamma(x).
\]
Hence, we have \eqref{lem:measure-lemma-1}. The claim follows from \eqref{eq:weak-type-1} and \eqref{lem:measure-lemma-1}.
\end{proof}

For $T>1$ we  define $\Gamma_T$ as in \eqref{def:gammaT} and define $\Sigma_T$ as
\begin{equation}
\label{def:sigmaT}
\Sigma_T :=  \Big{\{} t \geq T+1 :  \sup_{r \in (0,1)} \frac{\big| [t-r^2, t] \cap \Gamma_T \big|}{r^2} \geq \sqrt{|\Gamma_T|} \Big{\}},
\end{equation}
if $\Gamma_T$ is non-empty. If $\Gamma_T$ is an empty set, then we may choose $\Sigma_T = \emptyset$.  As we observed above, the set $\Gamma_T$ has measure $|\Gamma_T| \leq C e^{-\frac{c_1}{2}T}$ and therefore by Lemma \ref{lem:lebesgue} the set $\Sigma_T$ has measure 
\begin{equation}
\label{eq:measure-sigmaT}
|\Sigma_T| \leq C e^{-\frac{c_1}{4}T}.
\end{equation}

Let $\eps_0 >0$ be as in Theorem \ref{thm:decay-flatness} and assume still that  $n=2$. We claim that there is  $T_0$ such that  for all $t_0 \in [T_0+1,\infty) \setminus \Sigma_{T_0}$ it holds
\begin{equation}
\label{eq:density-sigmaT}
\inf_{r \in (0,1)} \frac{1}{r^2} \Big|\{ t \in [t_0-r^2,t_0] : \frac14 \| \mathrm{H}_{E^{h_n}(t) }\|_{L^2}^2 \leq 5 \pi  \} \Big| \geq 1 - \eps_0.
\end{equation}
We simplify the notation by $h_n =h$. Assume that $t_0 \geq T_0+1$ is such that  \eqref{eq:density-sigmaT} fails. Then there is $r \in (0,1)$  such that 
\[
\big|\{ t \in [t_0-r^2,t_0] : \frac14 \| \mathrm{H}_{E^{h}(t) }\|_{L^2}^2 > 5 \pi  \} \big| \geq  \eps_0 \, r^2 . 
\]
If $t \geq T_0$ is such that $ \frac14 \| \mathrm{H}_{E^{h}(t) }\|_{L^2}^2 > 5 \pi $, then by \eqref{def:gammaT-2} we conclude that $t \in \Gamma_{T_0}$ when $T_0$ is large. Therefore the above inequality yields 
\[
\big| [t_0-r^2,t_0] \cap \Gamma_{T_0} \big| \geq  \eps_0\, r^2. 
\]
Since $|\Gamma_T| \leq C e^{-\frac{c_1}{2}T}$ for every $T$, it holds $|\Gamma_{T_0}| \leq \eps_0^2$ when $T_0$ is large. We thus have 
\[
\big| [t_0-r^2,t_0] \cap \Gamma_{T_0} \big| \geq  \sqrt{|\Gamma_{T_0}| } \,  r^2
\]
for some $r \in (0,1)$, which in turn implies $t_0 \in \Sigma_{T_0}$ by definition \eqref{def:sigmaT}. Hence, we have \eqref{eq:density-sigmaT}. 

In the planar case, if $|E^{h}(t)| = |B_1|$ and $P(E^{h}(t)) \leq 2 \pi \, \sqrt{2} -\delta_0$ the definitions \eqref{def:gammaT} and \eqref{def:sigmaT}  imply  immediately that  the condition \eqref{eq:density-point-pl}
holds for  every $t \in [T_0,\infty)\setminus \Sigma_{T_0}  $, when $T_0$ is large. Moreover, for such $t$  we have by \cite[Proposition 2.1]{JuMoPoSpa} that 
\[
| P(E^{h}(t)) - 2 \pi| \leq C \|\kappa_{E^{h}(t)}-  \bar \kappa_{E^{h}(t)}\|_{L^2}. 
\]
Therefore we may estimate the Lagrange multipliers by the Gauss-Bonnet theorem, i.e.,  $\int \kappa_{E^{h}(t)} \, d \Ha^1 = 2 \pi$ as
\begin{equation} \label{eq:gammaT-pl}
|\lambda^{h}(t) -1| \leq   C  e^{-\frac{c_1}{4}t}  \qquad \text{for all }\, t \in [T_0,\infty)\setminus \Sigma_{T_0}.  
\end{equation}

We conclude this section by recalling the following estimate for the approximative flat flow from \cite{JN2}. We formulate the result in a slightly different way as it is stated in \cite{JN2}. 
\begin{proposition}
\label{prop:JN2}
Let  $\{E^{h}(t)\}_{t\geq 0}$ be  an approximative flat flow in $\R^{n+1}$ starting from a bounded set of finite perimeter. If the set $E^{h}(t_0)$ satisfies uniform ball condition with radius 
$r_0 \in (0,1)$, then there is $\delta>0$, depending on $r_0, |E^{h}(t_0)|$ and the dimension, such that all sets $E^{h}(t)$ with $t \in [t_0, t_0 + \delta]$ satisfy  uniform ball condition with radius 
$r_0/2$.  Moreover, the flow is smoothing which means that   for every $k \in \N$ there is $C_k$ such that 
\[
\sup_{t \in [t_0+ \frac{\delta}{2},  t_0 + \delta] } \|\Delta_{\partial E^{h}(t)}^k \mathrm{H}_{E^{h}(t) }\|_{L^2} \leq C_k. 
\] 
\end{proposition}

\section{Weak Harnack inequality}

As we already mentioned,  we will prove Theorem \ref{thm:decay-flatness} using methods from regularity theory for  fully nonlinear PDEs. This idea goes back to Savin \cite{Sa} in the elliptic case, while the parabolic case is due to Wang \cite{Wa}. 
In our case  the time is discrete and  we need to revisit the proof of the weak Harnack inequality by Wang. In this section we prove the crucial ABP estimate, which is the technical core of the proof.

We need to perturb the functions $v_r^\pm$ and define  $w_r^\pm: Q_1^- \to \R$ as
\begin{equation}
\label{def:w-r}
\begin{split}
&w_r^-(y, \ft_k) :=  v_r^-(y, \ft_k)-  r^{-2\alpha} \int_{r^2 \ft_k}^0 |\lambda^h(\tau +t_0+h)|\, d \tau \quad \text{and}\\ 
&w_r^+(y, \ft_k) :=  v_r^+(y, \ft_k) +  r^{-2\alpha} \int_{r^2 \ft_k}^0  |\lambda^h(\tau +t_0+h)|\, d \tau,
\end{split}
\end{equation}
and $w_r^\pm(y, \mathfrak{t}) := w_r^\pm(y, \mathfrak{t}_k)$ for $\mathfrak{t} \in [\mathfrak{t}_k, \mathfrak{t}_{k+1})$. The crucial nonlinear estimate 
is the weak Harnack inequality for $w_r^\pm$ defined in \eqref{def:w-r}. Since the argument is symmetric we only prove it for $w_r^-$.


 It turns out that $w_r^-$ is almost  a viscosity supersolution of the heat equation when $r$ is small, and we prove this in Lemma \ref{lem-visco1} below. We need to be careful when we use the terminology from the  viscosity theory, in particular, we need to define what it means to `touch from below'. 
\begin{definition}
\label{def:touch-from-below}
Let $w_r^-: Q_1^- \to \R$ be as in \eqref{def:w-r}.  We say that $\varphi \in C^2(Q_1^-)$ \emph{touches $w_r^-$ from below at $(y,\mathfrak{t}_k)$} if 
\[
\begin{split}
& w_r^-(\tilde y,\mathfrak{t}_j)  \geq  \varphi(\tilde y,\mathfrak{t}_j) \quad \text{for all }\, \tilde y \in B_1,  \, j \leq k-1, \quad \text{and}\\ 
& w_r^-(y,\mathfrak{t}_k)  - \varphi(y,\mathfrak{t}_k) = \min_{\tilde y \in B_1}\big( w_r^-(\tilde y,\mathfrak{t}_k) - \varphi(\tilde y,\mathfrak{t}_k) \big)\leq 0. 
\end{split}
\]
\end{definition}
We stress that the definition for touching from below is stated only  for discrete times, which changes the nature of  the proof of the   ABP-estimate. Definition \ref{def:touch-from-below} is flexible in the sense that we have to define the test function $\varphi$ only at discrete times  $\mathfrak{t}_k$. 

Touching from  below the   function $u_-$ from  \eqref{def:sub-super} is defined similarly. If  $\varphi  \in C^2$ touches $u_-$ from below at a point $(x, t_k)$ then by the regularity theory for perimeter minimizers  the boundary $E^h(t_k)$ is regular near $(x,u_-(x,t_k)) \in \partial E^h(t_k)$ in every dimension  \cite[Lemma 3]{DM}. In particular, we may write the mean curvature as, writing $u = u_-$ for short, 
\begin{equation}
\label{eq:mean-curv}
 \mathrm{H}_{E^h(t_k)}\big(x,u(x, t_k) \big)  = - \frac{1}{\sqrt{1 + |\nabla u(x, t_k)|^2} } \trace\left( \Big( I - \frac{\nabla u(x, t_k) \otimes \nabla u(x, t_k)}{1 + |\nabla u(x, t_k)|^2} \Big) \nabla^2 u(x, t_k)  \right). 
\end{equation}


We will repeatedly use the following technical lemma. 
\begin{lemma}
\label{lem-technic}
Let $w_r^-$ be as in \eqref{def:w-r}.
Assume $\varphi\in C^2(Q_1^-)$ is such that $r^\alpha \|\varphi\|_{C_x^{2}} \leq 1$  and it  touches  $w_r^-$ from below at $(y, \ft_k) \in Q_{\frac78}^-$. Then there is a point $\hat y$ with 
$|\hat y- y|\leq \hat C \sqrt{ \mathfrak{h}}$ and a non-negative number $\eta_r \leq C r^2$ such that 
\[
 \Delta \varphi(y, \ft_k)  \leq  Cr^{2-\alpha}  +  (1-  \eta_r)  \frac{\varphi(\hat y, \ft_k) -   \varphi(\hat y, \ft_{k-1}) }{ \mathfrak{h}}.  
\]
\end{lemma}

\begin{proof}
Let us simplify the notation as $w_r = w_r^-$,  $v_r = v_r^-$ and $u = u_-$. Moreover,  by adding a constant to $\varphi $ if needed, we may assume that $w_r(\tilde y, \ft_j) \geq \varphi(\tilde y, \ft_j)$ for all $\tilde y \in B_{1}$, $ \ft_j \leq \ft_k$  and $w_r(y,\ft_k) = \varphi(y,\ft_k)$.  It follows from the definitions of $v_r$ \eqref{def:v-r}, $w_r$ \eqref{def:w-r} and from the rescaling of the coordinates  \eqref{rescale}  that for $t = t_k = kh$
\[
u(x,t) = P(x,t) + \Lambda(t) + r^{2-\alpha}  \int_{t}^{t_0} |\lambda^h(\tau +h)|\, d \tau   + r^{2+\alpha} w_r\big(\frac{x}{r}, \frac{t- t_0}{r^2}\big),
\]  
where $P(x,t) = c + bt + \frac12  A x \cdot x $  with $ b = \trace(A)$. Therefore the function $\psi : Q_r^-(0,t_0) \to \R$,
\begin{equation}
\label{def:visco-psi}
\psi(x,t) :=  P(x,t) + \Lambda(t) + r^{2-\alpha}  \int_{t}^{t_0} |\lambda^h(\tau +h)|\, d \tau   + r^{2+\alpha} \varphi\big(\frac{x}{r}, \frac{t- t_0}{r^2}\big),
\end{equation}
touches $u$ from below at $(x, t_k) = (r y, r^2\ft_k +t_0) $, i.e., $ u(\tilde x, t_j) \geq \psi(\tilde x,t_j) $ for all $\tilde x \in B_r$, $j \leq k-1$ and $ u(x, t_k) = \psi(x,t_k)$. Since   $\nabla P(x,t) = Ax$ and $x = ry$  we have  
\begin{equation}
\label{eq:gradient-max}
\begin{split}
&\nabla u(x,  t_k)  =\nabla \psi(x,  t_k) =   r^{1+\alpha} \nabla \varphi(y, \ft_k) + r A y  \quad \text{and} \\
&\nabla^2 u(x, t_k)\geq \nabla^2 \psi(x,  t_k)  = r^{\alpha} \nabla^2 \varphi(y, \ft_k) +A. 
\end{split}
\end{equation}
Moreover, since  $r^\alpha \|\varphi\|_{C_x^{2}} \leq 1$ we have  
\begin{equation}
\label{eq:gradient-bound}
\|\nabla \psi(\cdot ,  t_k)\|_{C^0} \leq r^{1+\alpha}\|\nabla \varphi\|_{C^0}  + r |A|  \leq (1+C_0) r.  
\end{equation}
Therefore we conclude using \eqref{eq:mean-curv},  \eqref{eq:gradient-max}, \eqref{eq:gradient-bound}  and  $r^\alpha \|\varphi\|_{C_x^{2}} \leq 1$   that 
\begin{equation}
\begin{split}
\label{eq:EL-RHS}
- \mathrm{H}_{E^h(t_k)}&\big(x,u(x,t_k) \big)\\
&\geq  \frac{1}{\sqrt{1 + | \nabla \psi(x,t_k)|^2} } \trace\left( \Big( I - \frac{\nabla \psi(x,  t_k) \otimes \nabla \psi(x,  t_k)}{1 + | \nabla \psi(x,  t_k)|^2} \Big) \nabla^2 \psi(x,  t_k) \right)\\
&\geq  \trace A  + r^\alpha \Delta \varphi(y, \ft_k) - Cr^2 ( 1 + r^\alpha \|\varphi\|_{C_x^{2}})\\
& \geq      \trace A  + r^\alpha   \Delta \varphi(y, \ft_k) -  Cr^2, 
\end{split}
\end{equation}
for a constant $C$.  

Let us next estimate the LHS of the Euler-Lagrange equation \eqref{eg:Euler-Lag}. Here the difficulty is the implicit definition of  the distance function $d_{E^h(t_{k-1})}$. In order to deal with this, we claim that there exists a point $\hat x$, with $|\hat x - x|\leq C\sqrt{h}$, and a small number $\eta_r$,  with $0 \leq \eta_r \leq C r^2$, such that 
\begin{equation}
\label{eq:distance}
d_{E^h(t_{k-1})}\big(x,u(x,t_k) \big) \leq (1-  \eta_r)\big( \psi(\hat x, t_k) - \psi(\hat x, t_{k-1}) \big). 
\end{equation}
 We divide the argument  for \eqref{eq:distance}  in  two cases. 

Assume first that  $d_{E^h(t_{k-1})}\big(x,u(x,t_k) \big)\geq 0$. In this case we may estimate the geometric distance $d_{E^h(t_{k-1})}$ from above by the vertical distance between the points  $(x,u(x,t_k)) $ and $(x,u(x,t_{k-1}))$ and have 
\[
d_{E^h(t_{k-1})}\big(x,u(x,t_k) \big) \leq u(x,t_k) - u(x,t_{k-1}). 
\]
Since $\psi$ touches $u$ from below at $(x,t_k)$, we have $ u(x,t_k) =  \psi(x,t_k)$ and $ u(x,t_{k-1}) \geq \psi(x, t_{k-1})$. Therefore we have \eqref{eq:distance} for $\hat x = x$ and $\eta_r = 0$. 

Let us then assume that  $d_{E^h(t_{k-1})}\big(x,u(x,t_k) \big) < 0$.  By Proposition \ref{prop:apriori-est} it holds 
\[
|d_{E^h(t_{k-1})}\big(x,u(x,t_k) \big)|\leq C \sqrt{h}.
\]
  Therefore  we may choose a point $\hat x$ and a number $\sigma$ with $(\hat x, \sigma) \in \partial E^h(t_{k-1})\cap B_{C \sqrt{h}}\big(x,u(x,t_k) \big)$ such that 
\begin{equation}
\label{eq:el-LHS1}
d_{E^h(t_{k-1})}\big(x,u(x,t_k)\big) = - \sqrt{|x-\hat x|^2 + |u(x,t_k) -\sigma|^2}. 
\end{equation}
We may assume that  $\psi(\hat x, t_k) < \psi(\hat x, t_{k-1})$ since otherwise \eqref{eq:distance} is trivially true. Then it follows from the fact that $u(\cdot, t_{k-1})$ is the subgraph of $E^h(t_{k-1})$ and from $u(\hat x, t_{k-1})\geq \psi(\hat x, t_{k-1})$ that 
$\sigma \geq  \psi(\hat x, t_{k-1})$.  Therefore by $ u(x,t_k) =  \psi(x,t_k)$ and by \eqref{eq:gradient-bound} it holds 
\begin{equation}
\label{eq:el-LHS2}
\begin{split}
0 < \psi(\hat x, t_{k-1}) - \psi(\hat x, t_k) &\leq\big( \sigma -u(x,t_k)\big)  +  \big( \psi(x,t_k) - \psi(\hat x, t_k)  \big)\\ 
&\leq |u(x,t_k) - \sigma | + Cr|x - \hat x|. 
\end{split}
\end{equation}
Denote $\bar a = |u(x,t_k) - \sigma |$, $\bar b = |x - \hat x|$ and $\eps = Cr$. By Young's inequality 
\[
(\bar a +\eps  \bar b)^2 = \bar a^2 +2 \eps \bar a \bar b +\eps^2  \bar b^2 \leq (1+\eps^2)  \bar a^2 + (1+\eps^2)  \bar b^2 \leq (1+\eps^2)^2(\bar a^2+ \bar b^2). 
\]
Therefore by the above and by  \eqref{eq:el-LHS1} we have 
\[
\begin{split}
 |u(x,t_k) - \sigma | + Cr|x - \hat x| &\leq (1+Cr^2) \sqrt{|u(x,t_k) -\sigma|^2 + |x-\hat x|^2} \\
&= - (1+Cr^2) d_{E^h(t_{k-1})}\big(x,u(x,t_k)\big). 
\end{split}
\]
This together with \eqref{eq:el-LHS2}    implies \eqref{eq:distance}.

We proceed by using \eqref{eq:distance}, the definition of $\psi$ in \eqref{def:visco-psi}, the facts that $P(\hat x,t_k) - P(\hat x,t_{k-1}) = b h$, $\Lambda(t_k) - \Lambda(t_{k-1}) = \lambda^h(t_k) h$ and have for $r$ small enough that 
\[
\begin{split}
d_{E^h(t_{k-1})}&\big(x,u(x,t_k) \big) \leq (1-  \eta_r)\big( \psi(\hat x, t_k) - \psi(\hat x, t_{k-1}) \big) \\
&=  (1-  \eta_r)\Big(   b h + \lambda^h(t_k) h - r^{2-\alpha}  |\lambda^h(t_k)| h  +  r^{2+\alpha}\big( \varphi(\hat x/r, \ft_k) -   \varphi(\hat x/r, \ft_{k-1})   \big) \Big)\\
&\leq  b h + \lambda^h(t_k) h  + Cr^{2}h +   (1-  \eta_r)  r^{2+\alpha}\big( \varphi(\hat x/r, \ft_k) -   \varphi(\hat x/r, \ft_{k-1})   \big) , 
\end{split}
\]
where in the last inequality we used $0\leq \eta_r \leq Cr^2$. By setting $\hat y = \frac{\hat x}{r}$, recalling  that $\mathfrak{h} = \frac{h}{r^2}$ and $b =   \trace A$, and by using the Euler-Lagrange equation \eqref{eg:Euler-Lag}  we then conclude using  the above inequality 
\begin{equation}
\label{eq:el-LHS4}
- \mathrm{H}_{E^h(t_k)}\big(x,u(x,t_k) \big) \leq \trace A   + Cr^{2} +   (1-  \eta_r)  r^{\alpha} \frac{\varphi(\hat y, \ft_k) -   \varphi(\hat y, \ft_{k-1}) }{\fh}.  
\end{equation}
The claim follows from \eqref{eq:EL-RHS}, \eqref{eq:el-LHS4} and by recalling that  
 $|\hat x - x|\leq C\sqrt{h}$ and $y = \frac{x}{r}$, we have  $|\hat y - y |\leq  C\sqrt{\mathfrak{h}}$.
\end{proof}

The first consequence of Lemma \ref{lem-technic} is  that the function $w_r^-$ defined in  \eqref{def:w-r} is almost a supersolution of the heat equation when $r$ is small. 
\begin{lemma}
\label{lem-visco1}
Let $w_r^-$ be as  in  \eqref{def:w-r}. Assume $\varphi\in C^2(Q_1^-)$ is such that $r^\alpha \|\varphi\|_{C_x^{2}} \leq 1$  and it  touches  $w_r^-$ from below at $(y, \ft_k) \in Q_{\frac78}^-$. Then it holds
\[
\partial_t \varphi(y, \ft_k) \geq \Delta \varphi(y, \ft_k) - Cr^{2-\alpha} -  \fh \|\partial_{tt}^2 \varphi \|_{C^0} - C \sqrt{\fh} \|\partial_{t} \nabla \varphi \|_{C^0}.
\]

\end{lemma}

\begin{proof}
We begin by estimating the  last term  on the RHS in the statement of Lemma \ref{lem-technic} by  recalling that $\ft_k - \ft_{k-1} = \fh$  
\[
\frac{\varphi(\hat y, \ft_k) -   \varphi(\hat y, \ft_{k-1})}{\fh}   \leq \pa_t \varphi(\hat y, \ft_k)+ \|\pa_{tt}^2 \varphi\|_{C^0}\fh
\]
and then  using  $|\hat y - y |\leq \hat  C\sqrt{\mathfrak{h}}$
\[
\pa_t \varphi(\hat y, \ft_k) \leq \pa_t \varphi(y, \ft_k) + C  \|\pa_{t}\nabla \varphi\|_{C^0}\sqrt{\mathfrak{h}}. 
\]
The claim follows from the two inequalities above and from  Lemma \ref{lem-technic}.
\end{proof}

\subsection{The ABP-estimate}

We fix $a >0 $. For $\xi \in \R^n$ and $\tau \in \R$, we consider the  parabola centered at $(\xi, \tau) \in \R^{n+1}$ with a slope $a >0$ 
\begin{equation}\label{def:parabola}
p_{\xi, \tau;a}(x,t):=a(t - \tau)-\frac{a}2|x- \xi|^2.
\end{equation}
We often denote $p_{\xi, \tau} = p_{\xi, \tau;a}$ when the slope is clear from the context. 

\begin{remark}
\label{rem:para-algebra}
If we have two parabolas $p_{\xi_1, \tau_1;a_1}$ and $p_{\xi_2, \tau_2;a_2} $  with centers $(\xi_1, \tau_1), (\xi_2, \tau_2)  \in Q_\rho^-(x_0,t_0)$,  then 
 it holds  
\[
p_{\xi_1, \tau_1;a_1} + p_{\xi_2, \tau_2;a_2} = p_{\xi, \tau;a}
\]
for $a= a_1 + a_2$,     $\xi = \frac{a_1}{a} \xi_1 +\frac{a_2}{a} \xi_2  \in B_\rho(x_0)$ and 
\[
\tau = \frac{a_1}{a} \tau_1 + \frac{a_2}{a} \tau_2 +  \frac{a_1}{2a} |\xi_1|^2  +  \frac{a_2}{2a} |\xi_2|^2  -    \frac{|\xi|^2}{2}    >t_0 -\rho^2.
\]
 In addition, if  $p_{\xi, \tau;a}(x,t) \geq 0$ for some $(x,t) \in Q_\rho^-(x_0,t_0)$, then  $\tau \leq t_0$, and therefore   $(\xi,\tau) \in Q_\rho^-(x_0,t_0)$. 
\end{remark}

Let $G\subset\R^{n+1}$ be a compact set and fix a slope $a >0$. We define the set of \textit{contact points} $\A(a;G)\subset Q_1^-$ as follows,
\begin{equation} \label{def:A_a(G)}
\A(a;G) : = \{ (x,t) \in Q_1^- : \exists (\xi, \tau) \in G \,\, \text{s.t. } \, p_{\xi, \tau} \, \text{touches} \,  w_r^- \, \text{ from below  at } \, (x,\ft_{k}),  \ft_k \leq t < \ft_{k+1}\}.
\end{equation}
 For the  meaning for `touching from below'  see Definition \ref{def:touch-from-below}.

\begin{lemma}
    \label{lem:ABP}
  Let $w_r^-: Q_1 \to \R$ be as in \eqref{def:w-r} and $r^\alpha \leq a \leq   r^{-\alpha}$. Assume that $G \subset Q_1^-$ and $\A(a;G)  \subset Q_{\frac78}^-$. Then there exists a dimensional constant $C$ such that
   \[
    |G|\leq C |\A(a;G) |. 
    \]
\end{lemma}

\begin{proof}
Denote $u = u_-$ and $w_r = w_r^-$ for short.  For $k\in\mathbb{Z}$ we denote by $G_k$ the set of points $(\xi, \tau)\in G$ such that
\[
k=\max\{m\in\mathbb{Z} \,:\, w_r(y,\mathfrak{t}_j)>p_{\xi, \tau}(y,\mathfrak{t}_j)\text{ for all $y \in B_1 $ and $ j \leq m-1$}\}.
\]
In other words, for $(\xi, \tau)\in G_k$ the associated touching point is of the form $(x,\ft_k)$, i.e.,  
\[
 \A(a;G_k)  = \A(a;G)  \cap \big(\R^n \times [\ft_k, \ft_{k+1}) \big) = : \A_{a, k}.
\]
  Notice that by the definition of $p_{\xi,\tau}$, the position  in  space  of the touching point $(x,\ft_k)$  is independent of $\tau$ whenever $(\xi,\tau)\in G_k$.   Therefore if we denote  the projection to space as $\mathtt{Pr}: \R^{n+1} \to \R^n$, $\mathtt{Pr}(x,t) = x$, then it holds 
\[
 |\A_{a,k}| = \fh \,  \Ha^n\big( \mathtt{Pr}( \A_{a,k})\big). 
\]
Since  $G=\cup_k G_k$ and the sets $ \A_{a,k}$ are disjoint, it is enough to prove that for every $k$ it holds 
\begin{equation}
\label{eq:ABP-enough-to-prove}
|G_k| \leq C  \fh \Ha^n\big( \mathtt{Pr}( \A_{a,k}) \big). 
\end{equation}

Let us fix $(\xi , \tau) \in G_k$ and denote the associated touching point by $(x,\ft_k) \in \A_{a,k}$. Since the parabola $p_{\xi, \tau}$ defined in \eqref{def:parabola} touches $w_r$ from below (see Definition \ref{def:touch-from-below}) then it holds 
\begin{equation}
\label{eq:ABP-touch1}
\begin{split}
&\nabla w_r(x,\ft_k)= \nabla p_{\xi,\tau}(x,\ft_k)=-a(x-\xi) \quad \text{and} \\
&\nabla^2 w_r(x,\ft_k) \geq   \nabla^2 p_{\xi,\tau}(x,\ft_k)=-a I .
\end{split}
\end{equation}
From the first equality in \eqref{eq:ABP-touch1} we deduce that $\xi = x + a^{-1} \nabla w_r(x, \ft_k)$. Hence, we may define  map $\Psi: \mathtt{Pr}( \A_{a,k})  \to  \mathtt{Pr}(G_k)$  as
\begin{equation}
\label{eq:ABP-touch2}
\Psi(x) =  x + a^{-1} \nabla w_r(x, \ft_k). 
\end{equation}

We obtain from the second line in \eqref{eq:ABP-touch1} that $\nabla_x \Psi(x,\ft_k) = I +a^{-1}  \nabla^2 w_r(x, \ft_k)\geq 0$. It follows from the definition of the function $w_r$ in \eqref{def:w-r}  that at $x_r = rx$ and $t_k = r^2 \ft_k+t_0$ it holds 
\[
\begin{split}
&\nabla u(x_r,  t_k)  =    r^{1+\alpha} \nabla w_r(x, \ft_k) + r A x  \quad \text{and} \\
&\nabla^2 u(x_r, t_k)   = r^{\alpha} \nabla^2 w_r(x, \ft_k) +A. 
\end{split}
\]
In particular, $|\nabla u(x_r,  t_k)| \leq Cr $. Let us denote  the Pucci operator   $ \mathcal{P}_{\frac12,2}^-$ defined for symmetric matrices as 
\[
\mathcal{P}_{\frac12,2}^-(X) =  \inf_{\frac12I \leq  M \leq 2 I}  \trace \big( M X  \big),
\]
where the infimum is taken over symmetric matrices $M$ with eigenvalues bounded between $1/2$ and $2$ (see \cite{CC} for more detailed introduction to Pucci opertors). When $r$ is small,  we deduce from \eqref{eq:mean-curv} and from the above observations that 
\[
\begin{split}
- \mathrm{H}_{E^h(t_k)}\big(x_r,u(x_r,t_k) \big) &=  \frac{1}{\sqrt{1 + |\nabla u(x_r, t_k)|^2} } \trace\left( \Big( I - \frac{\nabla u(x_r, t_k) \otimes \nabla u(x_r, t_k)}{1 + |\nabla u(x_r, t_k)|^2} \Big) \nabla^2 u(x_r, t_k)  \right) \\
&\geq  r^\alpha  \mathcal{P}_{\frac12,2}^- \big(\nabla^2 w_r(x, \ft_k)  \big) +    \trace A - Cr^2.
\end{split}
\] 
We use the estimate \eqref{eq:el-LHS4} from the proof of Lemma \ref{lem-technic}  and have 
\[
\begin{split}
- \mathrm{H}_{E^h(t_k)}\big(x_r,u(x_r,t_k) \big) &\leq \trace A   + Cr^{2} +   r^{\alpha}\frac{p_{\xi, \tau}(x,\ft_k) - p_{\xi, \tau}(x,\ft_{k-1})}{\fh} \\
&\leq \trace A   + Cr^{2} + r^\alpha a.  
\end{split}
\]
Combining the previous two inequalities yields
\[
\mathcal{P}_{\frac12,2}^- \big(\nabla^2 w_r(x, \ft_k)  \big) \leq  a +  Cr^{2-\alpha} \leq (1 +Cr^{2-2\alpha})  a, 
\]
when $ r^\alpha \leq a$. Therefore since $\nabla^2 w_r(x,\ft_k) \geq -a I $ we deduce that 
\[
\nabla^2 w_r(x,\ft_k)  \leq C a I,
\]
when $\alpha < \frac14$ and  $r$ is small enough. Recalling the definition of  $\Psi:  \mathtt{Pr}( \A_{a,k})  \to  \mathtt{Pr}(G_k)$ from \eqref{eq:ABP-touch2}, we conclude by the area formula  and by the above that 
\begin{equation}
\label{eq:ABP-touch3}
\begin{split}
\Ha^n\big(\mathtt{Pr}(G_k)\big) = \Ha^n\big(\Psi(\mathtt{Pr}( \A_{a,k}) )\big) &\leq \int_{\mathtt{Pr}( \A_{a,k})}\det(\nabla_x \Psi(x,\ft_k) )\, dx \\
&= \int_{\mathtt{Pr}( \A_{a,k})}\det\big( I + a^{-1} \nabla^2 w_r(x,\ft_k)\big)\, dx \\
&\leq C \Ha^n\big(\mathtt{Pr}( \A_{a,k})\big).
\end{split}
\end{equation}

We need yet to relate the measure of the set $G_k \subset \R^{n+1}$ with the measure of its projection $\mathtt{Pr}(G_k) \subset \R^n$. In order to avoid confusion, we denote carefully the associated measures and claim that it holds 
\begin{equation}
\label{eq:ABP-time1}
\begin{split}
\Ha^{n+1}\big( G_k\big) \leq C \fh\,  \Ha^{n}\big( \mathtt{Pr}(G_k)\big).
\end{split}
\end{equation}
The claim \eqref{eq:ABP-enough-to-prove} then follows from \eqref{eq:ABP-touch3} and \eqref{eq:ABP-time1}.

 In order to prove \eqref{eq:ABP-time1} we claim that if  $\tau_1,\tau_2 \in \R$ are such that  $(\xi,\tau_1), (\xi, \tau_2) \in G_k$,  then 
\begin{equation}
\label{eq:ABP-time2}
|\tau_2 - \tau_1| \leq C_1 \fh. 
\end{equation}
In order to prove \eqref{eq:ABP-time2} we may assume $\tau_2 > \tau_1$. We define a test function $\varphi$ such that 
\[
\varphi(y,\ft_j) = \begin{cases} p_{\xi, \tau_1}(y,\ft_j), \,\, \text{for } \, j \leq k-1 \,\,  \text{ for all }\,  y \in\R^n,\\
p_{\xi, \tau_2}(y,\ft_k), \,\, \text{for } \, j = k \,\,  \text{ for all }\,  y \in \R^n.
\end{cases}
\]
By the definition of $G_k$,  $p_{\xi, \tau_1}$ and $p_{\xi, \tau_2}$ touch $w_r$ from below at some point  $(x,\ft_k)$, with $x \in B_{\frac78}$.  We conclude that also $\varphi$ touch $w_r$ from below at  $(x,\ft_k)$. Therefore Lemma \ref{lem-technic} yields 
\[
 \Delta \varphi(x, \ft_k)  \leq  Cr^{2-\alpha}  +  (1-  \eta_r)  \frac{\varphi(\hat x, \ft_k) -   \varphi(\hat x, \ft_{k-1}) }{ \mathfrak{h}}  
\]
for some point $\hat x$ with $|\hat x - x| \leq \hat  C \sqrt{\fh}$. It follows from the definition of  $p_{\xi, \tau}$ in  \eqref{def:parabola}  that
\[
\Delta \varphi(x, \ft_k)  = - n a \quad \text{and} \quad \varphi(\hat x, \ft_k) -   \varphi(\hat x, \ft_{k-1}) = p_{\xi, \tau_2}(\hat x,\ft_{k}) - p_{\xi, \tau_1}(\hat x,\ft_{k-1}) = a \fh -  a (\tau_2- \tau_1). 
\]
These, together with the previous inequality, yield for $r^\alpha \leq a$
\[
\tau_2- \tau_1 \leq C \fh 
\]
and \eqref{eq:ABP-time2} follows. 

We may then  proceed  to prove \eqref{eq:ABP-time1}. For every $\xi \in \text{Pr}(G_k)$ we denote $\tau_\xi := \inf\{ \tau \in \R : \,  (\xi,\tau) \in G_k\}$. Setting 
\[
\tilde G_k = \mathtt{Pr}(G_k) \times [\tau_\xi, \tau_\xi +C_1 \fh] 
\]
we conclude from  \eqref{eq:ABP-time2} that $G_k \subset \tilde G_k$. Therefore
\[
\begin{split}
\Ha^{n+1}(G_k) \leq \Ha^{n+1}(\tilde G_k) &= \int \chi_{\tilde G_k}(\xi,\tau) d\tau d\xi   = \int_{\R^n}\int_{\tau_\xi}^{\tau_\xi + C_1 \fh} \chi_{\tilde G_k}(\xi,\tau)d\tau d\xi\\
& = \int_{\R^n}\int_{0}^{C_1\fh}\chi_{\tilde G_k}(\xi,s + \tau_\xi)ds d\xi   = \int_{0}^{C_1\fh}\int_{\R^n} \chi_{\tilde G_k}(\xi,s + \tau_\xi)d\xi ds \\
& = \int_{0}^{C_1\fh} |\{\xi \in \R^n: (\xi, s+\tau_\xi) \in \tilde G_k\}|ds \\
&=C_1\fh  \mathcal H^n (\mathtt{Pr} (G_k)). 
\end{split}
\]
Hence, we have \eqref{eq:ABP-time1} and the claim of the lemma follows.
\end{proof}

\subsection{Basic measure estimate}

Since the equation we are dealing with degenerates as the gradient increases, we use the  idea of Savin \cite{Sa} and prove the decay of the contact set $\mathcal{A}(a, Q_1^1)$  as the slope $a$ increases, and not the decay of the level sets. The reason is that at a touching point we have a precise control on the gradient. The proof is based on the following fundamental lemmas, which we call basic measure estimates, which state that the set of touching points propagates  in time and space. It turns out that the equation has parabolic behavior in space-time cylinders $Q_\rho$ in large scales  $\rho \geq \sqrt{\fh}$. In smaller space-time cylinders, the behavior is elliptic.

\begin{lemma}[Basic measure estimate in large scales]
\label{lem:basic-macro}
 Let $w_r^-: Q_1^- \to \R$ be as in \eqref{def:w-r}, fix a cylinder $Q^- \subset Q_{\frac78}^-$, assume $w_r^-$ is  non-negative in $Q^-$,  $r \leq r_0$ and   $r^\alpha \leq a  \leq r^{-\alpha}$.   There exist constants $C_1 \geq 1$  and $\mu_1>0$ such that for all  $\rho \geq C_1 \sqrt{\fh}$ the following holds: 
If  $Q_{4\rho}^-(y,\ft)  \subset Q^-$ and $(y_1,\ft) \in  \mathcal{A}\big(a; Q^-\big)$ for some $y_1 \in \bar B_\rho(y)$, then for every $y_2 \in \bar B_{\rho/2}(y)$ it holds   
\[
\big|  \mathcal{A}\big(C_1 a; Q^- \big) \cap \big(  Q_{\rho/16}^-(y_2, \ft - \frac{3}{4}\rho^2)  \big) \big| \geq \mu_1 \rho^{n+2}. 
\]
In particular, the constants are independent of $a, r$ and $h$, when  $h$ is of course assumed to be small enough. 
\end{lemma}

\begin{proof}
Denote $w_r = w_r^-$ and choose 
\begin{equation}\label{eq:lem-basic1-sigma}
\sigma := \frac{1}{2\cdot 64^2 + 8n}.
\end{equation}
It turns out to be  technically more convenient to  prove a slightly stronger result. Namely we claim that if there is a point $(y_1,t_1)  \in  \mathcal{A}\big(a; Q^-\big)$, with $y_1 \in \bar B_\rho(y)$ and $-\sigma \rho^2+ \ft  \leq t_1 \leq \ft$, then for every  $y_2 \in \bar B_{\rho/2}(y)$  and for  space-time cylinder 
$\textbf{Q}_\rho  :=B_{\frac{\rho}{16}}\times  (-\sigma\rho^2,0]$ it holds 
\begin{equation}\label{eq:lem-basic1-claim}
\big|  \mathcal{A}\big(\hat C a; Q^- \big) \cap \big( \textbf{Q}_\rho +(y_2,  \ft - \frac{\rho^2}{4n}) \big)  \big| \geq \hat \mu  \rho^{n+2}, 
\end{equation}
for  constants $\hat C \geq 1$ and $\hat \mu >0$. 

We note that we  may repeat this argument, this time  for $y = y_2$, since by  \eqref{eq:lem-basic1-claim} there is a point $\tilde y_1 \in B_{\frac{\rho}{16}}(y_2)$ and  $\tilde t_1 \in [\ft -\sigma \rho^2- \frac{\rho^2}{4n},\ft -\frac{\rho^2}{4n})$ such that $(\tilde y_1, \tilde t_1) \in   \mathcal{A}\big(\hat C a; Q^- \big)$.  The claim of the lemma then follows by repeating \eqref{eq:lem-basic1-claim} for  $y= y_2$  and replacing $\ft$  by $ \ft - \frac{j}{4n}\rho^2$ with $ j =1, \dots, 3n$. Of course, we have to increase the value of the constant $\hat C$ in  the slope $\hat C a$.

We also point out that by the above assumption there is a parabola $p_{\xi, \tau; a}= p_{\xi, \tau}$ which touches $w_r$ from below at $(y_1,t_1)$.  By adding a constant we may assume that $p_{\xi, \tau}(x,t) \leq w_r(x,t)$ for $t \leq t_1$ and $p_{\xi, \tau}(y_1,t_1)= w_r(y_1,t_1)$. Fix $y_2 \in \bar B_{\rho/2}(y)$. Without loss of generality we may assume that $y_2 = 0$ and $\ft= 0$.  We prove the  claim \eqref{eq:lem-basic1-claim}  in two steps. 

\textbf{Step 1:} \, We show that there exists a point $z_1 \in B_{\frac{\rho}{32}}$ such that 
\begin{equation}
\label{eq:macro-1}
w_r\big(z_1,-\frac{\rho^2}{4n}\big) \leq p_{\xi, \tau}\big(z_1, -\frac{\rho^2}{4n}\big) + C'a \rho^2. 
\end{equation}

 Let $0 < \delta < \sigma  $ be a small  constant, whose choice will be clear later, where $\sigma$ is defined in \eqref{eq:lem-basic1-sigma}. Denote the heat kernel  by  $\Phi(x,t) = t^{-\frac{n}{2}} e^{-\frac{|x|^2}{4t}}$ and define  function $\psi : \R^n \times [-\frac{1}{4n},\infty)  \to \R$
\[
\psi(x,t) := \Phi\big(x,t+ \frac{1}{4n}+\delta\big)    - e^{-8n}   t - \frac{e^{-8n} }{(\frac{1}{4n} +\delta)^{\frac{n}{2}}}.  
\]
Clearly $\partial_t \psi = \Delta \psi  - e^{-8n}$.  We claim  that 
\begin{equation}
\label{eq:macro-2}
\psi(x,t) <0 \qquad \text{for all } \,  (x,t) \in \partial B_4 \times [-\frac{1}{4n},0] \quad \text{and} \quad    (x,t) \in  \Big(B_4\setminus \bar B_{\frac{1}{32}}\Big) \times \{-\frac{1}{4n} \}. 
\end{equation}

By a direct calculation one may check that the function $t \mapsto \big(t + \frac{1}{4n}+\delta\big)^{-\frac{n}{2}} e^{-\frac{4}{t + \frac{1}{4n}+\delta}}  $ is increasing in $ [-\frac{1}{4n},0]$ when $0 <\delta \leq \frac{1}{4n}$. Therefore for $|x|= 4$ and $ - \frac{1}{4n} \leq  t \leq 0$ by choosing $\delta$ small 
\[
\begin{split}
\psi(x,t) &= \frac{1}{\big(t + \frac{1}{4n}+\delta\big)^{\frac{n}{2}} } e^{ -\frac{4}{t + \frac{1}{4n}+\delta} } - e^{-8n }   t - \frac{ e^{-8n}}{(\frac{1}{4n} +\delta)^{\frac{n}{2}}} \\
&\leq  \frac{1}{\big(\frac{1}{4n}+\delta\big)^{\frac{n}{2}}} e^{-\frac{4}{\frac{1}{4n}+\delta}} + \frac{e^{-8n}}{4n}   - \frac{ e^{-8n} }{(\frac{1}{4n} +\delta)^{\frac{n}{2}}} \leq\frac{e^{-8n}}{4n}  -  \frac{ e^{-8n} }{2(\frac{1}{4n} +\delta)^{\frac{n}{2}}} <0.
\end{split}
\] 
For $|x| \geq \frac{1}{32}$ we have,  again by choosing $\delta $ small enough,  
\[
\psi(x,-\frac{1}{4n}) \leq \delta^{- \frac{n}{2} } e^{ -\frac{1}{64^2 \cdot \delta} }  + \frac{e^{-8n}}{4n}    - \frac{ e^{-8n}}{(\frac{1}{4n} +\delta)^{\frac{n}{2}}} <0.
\] 
Hence, we have \eqref{eq:macro-2}.  On the other hand, at every point $(x,t )\in \bar B_2 \times [-\sigma, 0]$ it holds 
\begin{equation}
\label{eq:macro-3}
\psi(x,t) \geq \frac{1}{\big(\frac{1}{4n} +\delta\big)^{\frac{n}{2}} } e^{ -\frac{1}{\frac{1}{4n} -  \sigma +\delta} }    - \frac{ e^{-8n}}{(\frac{1}{4n} +\delta)^{\frac{n}{2}}}   \geq \frac{1}{\big(\frac{1}{4n}+\delta\big)^{\frac{n}{2}} }\left(  e^{ -\frac{1}{\frac{1}{4n}-\sigma} }  -e^{-8n} \right)  >0, 
\end{equation}
where the last inequality follows from $\sigma < \frac{1}{8n}$.

We define function 
\[
\varphi(x,t) : = p_{\xi, \tau}(x,t) + (n +2)e^{8n}a \rho^2 \, \psi\big(\frac{x}{\rho},\frac{t}{\rho^2}\big).
\]
To  prove \eqref{eq:macro-1}, we claim that there is $z_1 \in B_{\frac{\rho}{32}}$ such that 
\begin{equation}
\label{eq:macro-4}
w_r\big(z_1,-\frac{\rho^2}{4n}\big) \leq \varphi\big(z_1, -\frac{\rho^2}{4n}\big). 
\end{equation}
Clearly, \eqref{eq:macro-4} implies \eqref{eq:macro-1}.
We argue by contradiction and assume \eqref{eq:macro-4} is not true. 

 Recall that $  p_{\xi, \tau} $ touches $w_r$ from below at $(y_1,t_1)$ for $- \sigma \rho^2  \leq t_1 \leq 0$ and $\sigma < \frac{1}{4n}$. In particular, $p_{\xi, \tau}(x,t) \leq w_r(x,t)$ for $t  \leq t_1 $. 
Therefore we have by \eqref{eq:macro-2} and by the contradiction assumption, that $\varphi < w_r$ on the parabolic boundary of the cylinder 
$\tilde{\textbf{Q}}_\rho := B_{4\rho} \times (-\frac{\rho^2}{4n}, t_1]$. On the other hand, it holds   $p_{\xi, \tau}(y_1,t_1) = w_r(y_1,t_1)$. Note that since we assume $0 = y_2 \in \bar B_{\rho/2}(y)$ then $|y|\leq \rho/2$, while $y_1 \in \bar B_\rho(y)$ yields $|y_1| \leq |y_1-y| + |y| < 2 \rho$. Therefore $(y_1, t_1) \in B_{2\rho} \times [-\sigma \rho^2, 0]$ and we observe by    $p_{\xi, \tau}(y_1,t_1) = w_r(y_1,t_1)$ and \eqref{eq:macro-3} that $w_r(y_1,t_1) <  \varphi(y_1,t_1)$. Therefore 
\[
\min_{\tilde{\textbf{Q}}_\rho} (w_r - \varphi)  <0
\]
and the minimum  is attained at a point  $(\tilde y,\ft_k)$ which is not on the parabolic boundary of $\tilde{\textbf{Q}}_\rho$. By adding a constant to $\varphi$ if necessary and extending it to $t< -\frac{1}{4n}$ in a suitable way, we have that  $\varphi$ touches $w_r$ from below at $(\tilde y,\ft_k)$.  We apply Lemma \ref{lem-visco1} with $r$ small enough and have  
\[
\partial_t \varphi(\tilde y, \ft_k) \geq \Delta \varphi(\tilde y, \ft_k) - Cr^{2-\alpha} -  \fh \|\partial_{tt}^2 \varphi \|_{C^0} - C \sqrt{\fh} \|\partial_{t} \nabla \varphi \|_{C^0}.
\]
Using  $\partial_t \psi - \Delta \psi = - e^{-8n}$ and $\partial_t p_{\xi, \tau} - \Delta p_{\xi, \tau} =  (n+1) a$ we deduce 
\begin{equation}
\label{eq:macro-5}
-a =(\partial_t \varphi -  \Delta \varphi)(\tilde y, \ft_k)  \geq   - Cr^{2-\alpha} - \frac{C_\delta}{\rho^2} a \fh - C  \frac{C_\delta}{\rho} a \sqrt{\fh}.
\end{equation}
We choose $C_1$ large enough so that for $\rho \geq C_1  \sqrt{\fh}$ and $r^\alpha \leq a $ with $r \leq r_0$ small enough  it holds 
\[ 
 Cr^{2-\alpha} + \frac{C_\delta}{\rho^2} a \fh + C  \frac{C_\delta}{\rho} a \sqrt{\fh} \leq \frac{a}{2}. 
\]
This contradicts \eqref{eq:macro-5} and thus we have \eqref{eq:macro-1}. 

\smallskip

\textbf{Step 2:} \,   We prove the inequality \eqref{eq:lem-basic1-claim}. Recall that we assume  $(y_2, \ft)= (0,0)$ and we need to prove 
\begin{equation}
\label{eq:macro-6}
\big|  \mathcal{A}\big(\hat C a; Q^-\big) \cap   \big( \textbf{Q}_\rho  +  (0, -\frac{\rho^2}{4n}) \big) \big| \geq \hat \mu \rho^{n+2} ,
\end{equation}
where $\textbf{Q}_\rho  =B_{\frac{\rho}{16}}\times  (-\sigma\rho^2,0]$ and $\sigma$ is from \eqref{eq:lem-basic1-sigma}. We  define the set of parameters 
\[
\hat G := \{ (z,s)\in \R^{n+1} : |z-z_1|\leq \frac{\sqrt{\sigma}}{2} \rho,  \quad \frac{\sigma}{2}\rho^2  \leq s \leq  \sigma \rho^2 \}, 
\]
where   $z_1 \in B_{\frac{\rho}{32}}$ is  from \eqref{eq:macro-1}.   For $(z,s)\in \hat G$ we consider the parabola 
\begin{equation}
\label{eq:macro-7}
q_{z,s}(x,t) :=  p_{\xi, \tau}(x,t)  + \tilde C a \big(t + s + \frac{\rho^2}{4n} \big) - \frac{\tilde C a}{2}|x- z|^2
\end{equation}
and claim that it touches $w_r$ from below at some point in the cylinder $\textbf{Q}_\rho +(0, -\frac{\rho^2}{4n})$.  

Since $  p_{\xi, \tau} $ touches $w_r$ from below at $(y_1,t_1)$ for $- \sigma \rho^2  \leq t_1 \leq 0$ and $\sigma < \frac{1}{4n}$, then  $p_{\xi, \tau}(x,t) \leq w_r(x,t)$ for $t \leq - \frac{\rho^2}{4n} $. For $t \leq  -\frac{\rho^2}{4n} - \sigma \rho^2$ it holds 
\[
q_{z,s}(x,t)  \leq p_{\xi, \tau}(x,t) <  w_r(x,t).  
\]
If $|x| \geq \frac{\rho}{16}$ then, since   $|z| \leq |z_1|+ |z-z_1| \leq \frac{\rho}{32} + \frac{\rho}{64}$,  we have $|x-z|\geq \frac{\rho}{64}$. Therefore  for $|x| \geq \frac{\rho}{16}$ and $t \leq -  \frac{\rho^2}{4n}$ we have by the choice of $\sigma$
\[
q_{z,s}(x,t) \leq    p_{\xi, \tau}(x,t)  + \tilde C a \sigma \rho^2  - \frac{\tilde C a}{2} \frac{\rho^2}{64^2} < p_{\xi, \tau}(x,t) \leq w_r(x,t). 
\]
On the other hand,  when  the constant $C'$ is fixed,  by choosing $\tilde C$ large enough we have   
\[
\begin{split}
q_{z,s}\big(z_1,-\frac{\rho^2}{4n}\big) \geq   p_{\xi, \tau}\big(z_1,-\frac{\rho^2}{4n}\big)  +   \frac{\tilde C \sigma}{2} a \rho^2 -  \frac{\tilde C }{2} \frac{\sigma}{4} a \rho^2  > p_{\xi, \tau}\big(z_1,-\frac{\rho^2}{4n}\big)  + C' a \rho^2. 
\end{split}
\] 
Therefore we obtain from \eqref{eq:macro-1} that $q_{z,s}\big(z_1,-\frac{\rho^2}{4n}\big)  >  w_r\big(z_1,-\frac{\rho^2}{4n}\big)$.   We conclude that for all  $(z,s)\in \hat G$ the parabola \eqref{eq:macro-7} touches $w_r$ from below  at some point in the cylinder $\textbf{Q}_\rho + (0, -\frac{\rho^2}{4n}) $. 

Remark \ref{rem:para-algebra} implies that we may write  the parabola in \eqref{eq:macro-7} as $q_{z,s}(x,t)  = p_{\xi_z, \tau_{z,s}; \hat C a}$ for  $\hat C = 1 + \tilde C$ and $(\xi_z,  \tau_{z,s}) \in Q^-$.  Denote $G = \{ (\xi_z, \tau_{z,s}) : z \in \hat G \}$. Since $|\hat G| \geq c \rho^{n+2}$, it follows from the formulas in Remark  \ref{rem:para-algebra} that $|G|\geq c_1\rho^{n+2}$. 
Hence, we have by Lemma \ref{lem:ABP} that 
\[
\big|  \mathcal{A}(\hat C a; Q^-) \cap \big(  \textbf{Q}_\rho + (0, -\frac{\rho^2}{4n})  \big) \big| \geq \hat \mu \rho^{n+2},
\]
which is exactly the inequality  \eqref{eq:macro-6}.
\end{proof}

\begin{lemma}[Basic measure estimate in small scales]
\label{lem:basic-micro}
 Let $w_r^-: Q_1 \to \R$ be as  in \eqref{def:w-r}, fix a cylinder $Q^- \subset Q_{\frac78}^-$, assume $w_r^-$ is  non-negative in $Q^-$, $r \leq r_0$ and  $r^\alpha \leq a  \leq r^{-\alpha}$.    There exist  constants $C_2, C_3 \geq 1$  and $\mu_2>0$ such that the following holds for $Q_{C_3\sqrt{ \fh}}^-(y,\ft) \subset Q^-$. 
\begin{itemize}
\item[(a)] If   $(y_1,\ft)   \in  \mathcal{A}(a; Q^-)$ for some $y_1 \in \bar B_{4\rho}(y)$ with   $0< \rho \leq \sqrt{\fh}$ , then  it holds 
\[
\Ha^n \big( \mathcal{A}(C_2 a; Q^-) \cap \big( B_{\rho}(y) \times \{\ft\}  \big) \big)  \geq \mu_2 \rho^n. 
\]
\item[(b)]  If    $(y_1,\ft + j \fh)   \in  \mathcal{A}(a; Q^-)$ with $0 \leq j \leq C_3$ for some $y_1 \in \bar B_{(j+2)\sqrt{\fh}}(y)$, then  it holds 
\[
\Ha^n \big( \mathcal{A}(C_3 a; Q^-) \cap \big( B_{\sqrt{\fh}/2}(y) \times \{\ft\}  \big) \big)  \geq \mu_2 \fh^{\frac{n}{2}}. 
\]
\end{itemize}

\end{lemma}

\begin{proof}  
\textbf{Claim (a):} \,  Denote $w_r = w_r^-$.  Without loss of generality we may assume that $y = 0$ and  $\ft = \ft_k = k \fh$ for some $k \in \mathbb{Z}$. Let us first prove the claim (a) for $\rho \leq 4^{-n-2} \sqrt{\fh}$.

By the assumption there is a parabola $p_{\xi, \tau; a}= p_{\xi, \tau}$ which touches $w_r$ from below at $(y_1,\ft_k)$ with $y_1 \in \bar B_{4\rho}$. By adding a constant we may assume that $p_{\xi, \tau}(\cdot,t) \leq w_r(\cdot ,t)$ for $t \leq t_1$ and $p_{\xi, \tau}(y_1,\ft_k)= w_r(y_1,\ft_k)$.    First, we claim that there is a point $z_1 \in B_{\frac{\rho}{2}}$  such that 
\begin{equation}
\label{eq:micro-0}
w_r(z_1, \ft_k) \leq p_{\xi, \tau}(z_1, \ft_k) + 8^{n+3}a \rho^2. 
\end{equation}
To this aim, consider function $g \in C^2(\R^n)$,  $g(x) = |x|^{-n-1} -1$ for $|x|\geq \tfrac18$, positive in $B_1$  and $\|g\|_{C^0} \leq 8^{n+1}$. Then it holds for $x \in \bar  B_1 \setminus B_{\frac18}$
\begin{equation}
\label{eq:micro-2a}
\Delta g(x) = 3(n+1)|x|^{-n-3} \geq 3(n+1). 
\end{equation}
  We define  function $\psi : \R^n \to \R$,
\[
\psi(x) := p_{\xi, \tau}(x,\ft_k) + a (4\rho)^2 g\big( \frac{x}{4\rho} \big)
\]
and claim that 
\begin{equation}
\label{eq:micro-3a}
\min_{\bar B_{5\rho}} (w_r(\cdot,\ft_k) - \psi) = \min_{\bar B_{\rho/2}} (w_r(\cdot,\ft_k) - \psi). 
\end{equation}
Since  $p_{\xi , \tau}$ touches $w_r$ from below at $(y_1, \ft_k)$ with  $y_1 \in \bar B_{4\rho}$ and $g \geq 0$ in $B_1$, it holds  $w_r(y_1, \ft_k) =  p_{\xi , \tau}(y_1,\ft_k) \leq \psi(y_1,\ft_k)$. Therefore   the minimum on the LHS in  \eqref{eq:micro-3a} is non-positive. In addition, once we have proven  \eqref{eq:micro-3a}, then \eqref{eq:micro-0} follows by choosing $z_1$ as the point where the minimum in \eqref{eq:micro-3a} is attained.

To prove \eqref{eq:micro-3a} we argue  by contradiction and assume the minimum is attained at  $\tilde y \in \bar B_{5 \rho} \setminus  \bar{B}_{\rho/2}$.  By construction  for all $|x|>4 \rho$ it holds  $\psi(x) < p_{\xi, \tau}(x,\ft_k) \leq w_r(x, \ft_k)$ and  therefore 
$\tilde y \in \bar B_{4 \rho}$.  We define a test function $\varphi$ as
\[
\varphi(x,\ft_j) = \begin{cases} p_{\xi,\tau}(x,\ft_j), \quad  \text{for }  j < k, \\
p_{\xi, \tau}(x,\ft_k) + a (4\rho)^2 g\big( \frac{x}{4\rho} \big), \quad  \text{for }  j = k.  
\end{cases}
\]
Then, according to Definition \ref{def:touch-from-below}, $\varphi$ touches $w_r$ from below at $(\tilde y,\ft_{k})$ with $\tilde y \in \bar B_{4 \rho}\setminus  \bar{B}_{\rho/2}$  and Lemma \ref{lem-technic} yields
\begin{equation}
\label{eq:micro-4a}
 \Delta \varphi(\tilde y,\ft_{k}) \leq  Cr^{2-\alpha}  +  (1-  \eta_r)  \frac{\varphi(\hat y, \ft_{k}) -   \varphi(\hat y, \ft_{k-1}) }{ \mathfrak{h}}  
\end{equation}
for $|\hat y| \leq \hat C \sqrt{\fh}$. We have by \eqref{eq:micro-2a} that 
\begin{equation}
\label{eq:micro-5a}
 \Delta \varphi(\tilde y,\ft_{k}) =  \Delta p_{\xi, \tau}(\tilde y,\ft_k) +  a (\Delta g)\big( \frac{\tilde y}{4 \rho} \big) \geq -a n + 3(n+1) a > 2(n+1) a.
\end{equation}
On the other hand, since $\|g\|_{C^0} \leq 8^{n+1}$ and $\rho \leq 4^{-n-2} \sqrt{\fh}$ we have 
\[
\begin{split}
\varphi(\hat y, \ft_{k}) -   \varphi(\hat y, \ft_{k-1}) = p_{\xi, \tau}(\hat y,\ft_k) - p_{\xi, \tau}(\hat y,\ft_{k-1})  + a (4\rho)^2 g\big( \frac{\hat y}{4 \rho} \big)  < a \fh +  a  \fh = 2 a \fh 
\end{split}
\]
Therefore by the above,  \eqref{eq:micro-4a} and  \eqref{eq:micro-5a}  we have  for $r^\alpha \leq a$ and $r$  small
\[
 2(n+1)  a  < Cr^{2-\alpha}  + 2a   \leq  3a,
\] 
which is a contradiction. Hence we have \eqref{eq:micro-3a}, which in turn implies  \eqref{eq:micro-0}.

We proceed by considering parabolas of type
\[
q_{z}(x,t) = p_{\xi, \tau}(x,t) + C' a\big(t -\ft_{k}+ \frac{\rho^2}{32} \big) - \frac{C' a}{2} |x - z|^2
\] 
for $z \in \bar B_{\frac{\rho}{5}}(z_1)$ and a large constant $C'$. Here $z_1 \in B_{\frac{\rho}{2}}$ is from \eqref{eq:micro-0}.  Since $\rho \leq \sqrt{\fh}$, it follows that for every $j \leq k-1$   we have $q_{z}(\cdot,\ft_j) < p_{\xi, \tau}(\cdot,\ft_j) \leq w_r(\cdot, \ft_j)$.  Similarly for all $x$ with $|x-z| > \frac{\rho}{4}$ it holds $q_{z}(x,\ft_k) <  p_{\xi, \tau}(x,\ft_k)\leq w_r(x, \ft_k)$. On the ther hand, using $|z - z_1| \leq \frac{\rho}{5}$ and choosing $C'$ large enough yield
\[
\begin{split}
q_{z}(z_1,\ft_k) &= p_{\xi, \tau}(z_1,\ft_k) +   C' a \frac{\rho^2}{32}    - \frac{C' a}{2} |z_1 - z|^2 \\
&\geq   p_{\xi, \tau}(z_1,\ft_k)  + C'  a \rho^2  \left( \frac{1}{32} - \frac{1}{50}  \right) >  p_{\xi, \tau}(z_1,\ft_k)  +   8^{n+3} a \rho^2. 
\end{split}
\]
Therefore we have by   \eqref{eq:micro-0} that $q_{z}(z_1,\ft_k) > w_r(z_1,\ft_k)$. We conclude that the parabola $q_{z}$ touches $w_r$ from below at some point $(x,\ft_k)$ and 
\[
|x| \leq | x-z| + |z-z_1| + |z_1| \leq   \frac{\rho}{4} + \frac{\rho}{5}  + \frac{\rho}{2} < \rho.  
\]

By  Remark \ref{rem:para-algebra} for every  $z \in \bar B_{\frac{\rho}{5}}(z_1)$  we may  write   $q_{z}(x,t) = p_{\xi_z, \tau_z ; C_2 a}$ for  $C_2 = 1 + C'$ and $\xi_z = \frac{1}{1+C'} \xi + \frac{C'}{1+ C'}z$.  Denote $G = \{ \xi_z : z \in \bar B_{\frac{\rho}{5}}(z_1)\} $ and note that $|G|\geq c\rho^n$. We conclude that for every $\xi_z \in G$, the parabola $ p_{\xi_z, \tau_z ; C_2 a}$ touches $w_r$ from below at a point $x \in B_{\rho}$. Hence, we have by Lemma \ref{lem:ABP}, or to be more precise, by the estimate \eqref{eq:ABP-touch3} in the proof of Lemma \ref{lem:ABP}, that 
\[
C\Ha^n \big( \mathcal{A}(C_2 a; Q^-) \cap \big( B_\rho \times \{\ft_k\}  \big) \big)  \geq \Ha^n(G) \geq c \rho^n. 
\]
This proves the claim (a) in the case $\rho \leq 4^{-n-2} \sqrt{\fh}$.

If $\rho > 4^{-n-2} \sqrt{\fh}$ and  $|y_1| \geq 2 \cdot 4^{-n-2} \sqrt{\fh}$, then we use  the previous  estimate repeatedly  as follows. We denote $\hat \rho =  4^{-n-2} \sqrt{\fh}$ and choose $y_2 = y_1 - 2  \hat \rho  \frac{y_1}{|y_1|}$. Then $y_1 \in B_{4 \hat \rho}(y_2)$ and we may use the above estimate to deduce 
\[
\Ha^n \big( \mathcal{A}(C_2 a; Q^-) \cap \big( B_{\hat \rho}(y_2) \times \{\ft_k\}  \big) \big)  \geq c \rho^n. 
\]
We repeat the argument by choosing $y_3 = y_1 - 3  \hat \rho  \frac{y_1}{|y_1|}$. Then $y_2 \in  \in B_{4 \hat \rho}(y_3)$ and by the above $y_2 \in \mathcal{A}(C_2 a; Q^-)$. Therefore the estimate implies 
\[
\Ha^n \big( \mathcal{A}(C_2^2 a; Q^-) \cap \big( B_{\hat \rho}(y_3) \times \{\ft_k\}  \big) \big)  \geq c \rho^n. 
\]
Therefore, we repeat this argument  for points $y_j = y_1 -  j 4^{-n-2} \sqrt{\fh} \frac{y_1}{|y_1|}  $ with $j =2,3,  \dots, N$ until we reach $y_N \in B_{\rho/2}$ and the claim (a) follows also in this case.  

\medskip

\textbf{Claim (b):} \, We again assume $y = 0, \ft = \ft_k$,  and prove the claim first  for $j = 1$.  Note that  $\ft + \fh = \ft_{k+1}$. By the assumption there is a parabola $p_{\xi,\tau}$ with an opening $a$ and center $(\xi, \tau) \in Q^-$, which touches $w_r$ from below at $(y_1,\ft_{k+1})$ with $|y_1| \leq 3\sqrt{\fh}$. In particular, $p_{\xi,\tau}(\cdot, \ft_i) \leq w_r(\cdot, \ft_i)$ for all $i \leq k$. Let us denote 
\[
q(x) = C' a \fh -  a |x- y_1|^2
\]
for a large constant $C'$, whose choice will be clear later. We claim that there is   a point $z_2$ with $|z_2 - y_1|< \sqrt{C'  \fh}$ such that 
\begin{equation}
\label{eq:micro-1}
q(z_2) + p_{\xi,\tau}(z_2,\ft_k) > w_r(z_2,\ft_k). 
\end{equation}
We argue by contradiction. Since $q(x) \leq 0$ for all $|x-y_1| \geq  \sqrt{C' \fh}$ and $p_{\xi,\tau}(\cdot, \ft_k) \leq  w_r(\cdot, \ft_k)$, then the contradiction assumption implies  $q(x) + p_{\xi,\tau}(x,\ft_k) \leq  w_r(x, \ft_k)$ for all $x$. We define a test function $\varphi$ as
\[
\varphi(x,\ft_j) = \begin{cases} p_{\xi,\tau}(x,\ft_j), \,\, \text{for }  j \neq k, \\
q(x) + p_{\xi,\tau}(x,\ft_k),  \,\, \text{for }  j = k.  
\end{cases}
\]
Then $\varphi$ touches $w_r$ from below at $(y_1,\ft_{k+1})$ and Lemma \ref{lem-technic} implies 
\begin{equation}
\label{eq:micro-2}
 \Delta \varphi(y_1,\ft_{k+1}) \leq  Cr^{2-\alpha}  +  (1-  \eta_r)  \frac{\varphi(\hat y, \ft_{k+1}) -   \varphi(\hat y, \ft_{k}) }{ \mathfrak{h}}  
\end{equation}
for some point $\hat y$ with $|y_1 - \hat y | \leq \hat C \sqrt{\fh}$. The last condition yields 
\[
\begin{split}
\varphi(\hat y, \ft_{k+1}) -   \varphi(\hat y, \ft_{k}) &= p_{\xi,\tau}(\hat y,\ft_{k+1}) - p_{\xi,\tau}(\hat y,\ft_k)  -  q(\hat y) \\
&= a \fh - C' a \fh +  a |\hat y- y_1|^2\\
&\leq  a \fh - C' a \fh +  a \hat C^2 \fh \leq - \frac{C'}{2}   a \fh, 
\end{split}
\]
when $C'$ is chosen large enough.  Since $\Delta \varphi(x, \ft_{k+1}) = \Delta p_{\xi,\tau}(x,\ft_{k+1}) = -n a$ we obtain from \eqref{eq:micro-2}, by recalling that $\eta_r \leq Cr^2$, 
\[
-n a \leq Cr^{2-\alpha}  - (1 - Cr^2)  \frac{C'}{2}  a   \leq  - \frac{C'}{4}  a, 
\] 
when $r^\alpha \leq a$ and $r$ is small enough. This is clearly a contradiction when $C'$ is large enough and thus we have \eqref{eq:micro-1}.

We define  parabola 
\[
\tilde q(x,t) = C' a (t - \ft_{k-1} ) - \frac{C' a}{2} |x- z_2|^2 +  p_{\xi,\tau}(x,t) .
\]
It clearly holds $\tilde q(x,t) \leq  p_{\xi,\tau}(x,t)$ for all 
$t \leq \ft_{k-1} $ and $(x,\ft_{k})$ with $|x-z_2| \geq \sqrt{2  \fh}$. Since $p_{\xi,\tau}(\cdot, \ft_i) \leq  w_r(\cdot, \ft_i)$ for all $i \leq k$, then $\tilde q$ is below $w_r$ in these points. On the other hand, by \eqref{eq:micro-1}  we have 
that 
\[
\tilde q(z_2,\ft_{k}) =  C' a \fh +  p_{\xi,\tau}(z_2,\ft_{k}) \geq q(z_2) +  p_{\xi,\tau}(z_2,\ft_{k}) > w_r(z_2,\ft_{k}). 
\]
Therefore   $\tilde q$ touches $w_r$ from below at some point $(y_3,\ft_{k})$ with $|y_3-z_2|\leq \sqrt{2  \fh}$. By Remark \ref{rem:para-algebra} we may write 
$\tilde q(x,t) = p_{\xi',\tau';a'}$ for some $(\xi',\tau') \in Q^-$ and $a' = (C' +1)a$ and therefore $(y_3, \ft_k) \in \A\big( (C' +1)a; Q^-\big)$. Since $|z_2 - y_1 |\leq\sqrt{C' \fh} $ and $|y_1| \leq 3\sqrt{\fh}$  we have 
$|y_3|\leq C' \sqrt{\fh}$ for $C'$ large enough. We may then use the claim (a)  repeatedly with $\rho = \sqrt{\fh}/2$ for points $\tilde y_l = y_3 -l \rho  \,  \frac{y_3}{|y_3|}$ with $l=2,3, \dots, N$ until we reach $\tilde y_N \in \bar B_{\sqrt{\fh}}$  and obtain the claim (b). 

Finally the Claim (b)  for general  $j = 2,3,\dots, $ follows by using  the above estimate repeatedly  for  points $\big(\lambda_l y_1 + (1- \lambda_l)y, \ft +(j-l)\fh\big)$ (in place of $(y,\ft)$), with  $l =1,  2, \dots, j$ and  $\lambda_l =  \max\{0, (1- \frac{(l+2)\sqrt{\fh}}{|y_1-y|}) \}$ as in the Claim (a).   
\end{proof}

\subsection{Weak Harnack inequality}

Using the basic measure estimates from  Lemma \ref{lem:basic-macro} and Lemma \ref{lem:basic-micro} together with  a  covering argument we obtain weak Harnack inequality.  The covering argument is similar  to \cite{Wa}, but since the time is discrete, we have to be careful when our covering cylinders are small. For this reason we give the argument in full details. 

\begin{proposition}
\label{prop:weak-harnack}
Let $\delta_0>0$. There exist  constants    $r_0, m_0 > 0 $ and $C_0\geq 1$  such that  if  $w_r^-: Q_1^- \to \R$, defined  in \eqref{def:w-r}   with $C_0 \sqrt{\fh} \leq r \leq r_0$,  is non-negative in $Q_{16\rho_0}^-(y,\ft)  \subset Q_{\frac78}^-$, with $\rho_0 \geq r^{\alpha/2}$ and $\inf_{Q_{\rho_0}^-(y,\ft)} w_r  \leq m_0$, then 
\[
\big|\big{\{} (\tilde y,\tilde \ft) \in   Q_{\rho_0}^-(y,\ft-8\rho_0^2)  : \,   w_r^-(\tilde y,\tilde \ft) \leq C_0 \inf_{Q_{\rho_0}^-(y,\ft)} w_r^-   \big{\}} \big| \geq  \big( 1-\frac{\delta_0}{300} \big) |Q_{\rho_0}^-|. 
\] 
\end{proposition}

\begin{proof}
Denote $w_r = w_r^-$. Let $(y',t') \in \bar Q_{\rho_0}^-(y,\ft)$ be such that  $w_r(y',t') = \inf_{Q_{\rho_0}^-(y,\ft)} w_r  = m$. By considering $w_r +\eps$ we may assume $m>0$.  We begin by showing that there exists $(y_1,t_1)\in Q_{\frac{5}{4}\rho_0}^- (y,\ft)$ such that $(y_1,t_1) \in \A(a_0; Q_{16\rho_0}^-(y,\ft))$ for $a_0 = \frac{32m}{\rho_0^2}$, where the contact set $\A(\cdot;\cdot)$ is defined in \eqref{def:A_a(G)}. Indeed, consider parabola 
\[
p(x,t) = m + \frac{32m}{\rho_0^2} (t-t') -  \frac{16m}{\rho_0^2} |x-y'|^2.
\]
Then $p(\cdot, t) <0$ for all $t \leq t' - \frac{\rho_0^2}{4^2}$, and  $p(x, t) <0$ for all $|x-y'| \geq \frac{\rho_0}{4}$ and $t <t'$. Since  $p(y',t') = m =  w_r(y',t')$, we deduce that the parabola $p$ touches $w_r$ from below at some point  $(y_1,t_1)\in  Q_{\frac{\rho_0}{4}}^-(y',t')$.

We define domain  
\[
\mathcal{B}^- := \{ (x,t) \in Q_{3\rho_0}^- (y_1,t_1) : t-t_1 < -|x-y_1|^2  \}.  
\]
Then it holds $Q_{\rho_0}^-(y,\ft-8\rho_0^2)  \subset \mathcal{B}^-$.  We choose $L := \max\{ C_1, C_2, C_3\}$, where the constants are from Lemma \ref{lem:basic-macro} and Lemma \ref{lem:basic-micro}, and consider contact   sets
\[
A_k = \{(x,t) \in \mathcal{B}^- \cap \A\big(L^k a_0; Q_{16\rho_0}^-(y,\ft) \big) \}. 
\]
We may assume that $t_1 = \ft_l = l \fh$ for some $k \in \mathbb{Z}$. We also assume $\rho_0^2 =  l'  \fh$ for  some $l' \in \mathbb{N}$ to simplify the argument. The case for general $\rho_0$ follows from the same argument with minor adjustements.  
Our goal is to show that there exists $\eta \in (0,1)$ such that 
\begin{equation} \label{eq:weak-harnack-0}
|\mathcal{B}^- \setminus A_k | \leq \eta^k |\mathcal{B}^-|
\end{equation}
for all $k$ which satisfy $32 m_0 L^k \leq 1$. 

Let us first make a few remarks on \eqref{eq:weak-harnack-0}. From the definition of contact point it follows that if $(x,t) \in A_k$ and $t \in  [\ft_i, \ft_{i+1})$, then  $(x,s) \in A_k$  for all $s \in [\ft_i, t)$. Moreover, since we assume that $\rho_0 \geq r^{\alpha/2}$ and $a_0 = \frac{32m}{\rho_0^2}$,  the slope $a= L^k a_0$ satisfies $a\leq r^{-\alpha}$ if $32 m_0 L^k \leq 1$. We will assume that  $m \leq m_0$ is small and $k$ is not too large so that this   is true throughout the proof. Finally, a parabola  $p_{\xi, \tau;a}$ defined in \eqref{def:parabola} with $(\xi,\tau) \in Q_{16\rho_0}^-(y,\ft)$ and $0 < a \leq  L^k a_0 = L^k \frac{32m}{\rho_0^2} $ satisfies $|p_{\xi, \tau;a}(x,t)| \leq C_0 m$ for $(x,t) \in Q_{16\rho_0}^-(y,\ft)$. 
Therefore if $(x,t)\in   \A\big(L^k a_0; Q_{16\rho_0}^-(y,\ft)\big)$, then $w_r(x,t) \leq C_0 m$. Hence, the inequality \eqref{eq:weak-harnack-0}  implies the claim by applying it with large enough $k$.

We prove  \eqref{eq:weak-harnack-0} using the basic measure estimates from  Lemma \ref{lem:basic-macro} and Lemma \ref{lem:basic-micro} together with  a covering argument. By definition of contact point, $A_k$ is relatively closed and thus $\mathcal{B}^-  \setminus A_k$ is (relatively) open.   We cover the set $\mathcal{B}^-  \setminus A_k $ with a family of  cylinders such that  for every point $(x,\ft_k) \in \mathcal{B}^-  \setminus A_k$ we choose a cylider $\textbf{Q}_{(x,\ft_k)}$ as follows. Denote an upward opening parabola of width $r$ as
\[
\mathcal{B}_{R} = \{ (y,s) \in \R^{n+1} :  |y|^2 <s  < R^2 \}. 
\]
We let $\rho>0$ be the infimum of the values $R$ for which the set $ \big( \mathcal{B}_R+(x,\ft_k) \big) \cap A_k$  is non-empty. The definition of $\mathcal{B}^- $  and the fact that $(y_1,t_1) \in \A\big(a_0; Q_{16\rho_0}^-(y,\ft) \big)$ guarantee that such $\rho$ exists and $\rho \leq 3\rho_0$.  If $\rho \geq \sqrt{\fh}$, we define $\textbf{Q}_{(x,\ft_k)} = Q_\rho^+ (x,\ft_k)$. On the other hand, if  $\rho < \sqrt{\fh}$ we 
define $\textbf{Q}_{(x,\ft_k)} = B_\rho(x) \times [\ft_k, \ft_k + \fh)$. Then the union of $\textbf{Q}_{(x,\ft_k)}$ covers $\mathcal{B}^- \setminus A_k $.

We claim that there is a universal constant $\mu_3 >0$ such that for every such cylinder $\textbf{Q}_{(x,t)}$ it holds 
\begin{equation} \label{eq:weak-harnack-1}
\big|  \big( A_{k+1} \cap \textbf{Q}_{(x,\ft_k)} \big) \setminus  A_k  \big| \geq \mu_3 | \textbf{Q}_{(x,\ft_k)}|
\end{equation}
We divide the proof in three cases. 

Let us first consider the case when the width of the cylinder  $\textbf{Q}_{(x,\ft_k)}$ satisfies  $\rho \geq C_1 \sqrt{\fh} $ where $C_1$  is from Lemma  \ref{lem:basic-macro}.  Denote $R = \sqrt{t_1-\ft_k} < 3 \rho_0$, and note that by the definition of $\mathcal{B}^-$ and $\rho$ it holds $|x-y_1|<R$ and $\rho \leq R$. We claim that  we may  choose  a point $y_2 \in B_{\rho/2}(x)$ such that  
\begin{equation} \label{eq:weak-harnack-2}
Q_{\frac{\rho}{16}}^-(y_2, \ft_k+ \frac{\rho^2}{4})  \subset  \mathcal{B}^- \cap \big(  \mathcal{B}_\rho + (x,\ft_k) \big). 
\end{equation}

First, we may  choose $y_2 \in B_\rho(x)$ such that 
\begin{equation} \label{eq:weak-harnack-3}
B_{\frac{\rho}{16}}(y_2) \subset B_{\frac{3\rho}{8}}(x) \cap B_{R - \frac{\rho}{4}}(y_1).  
\end{equation}
Indeed, if $|x-y_1| \geq \frac{5 \rho}{16}$ we choose $y_2 = x- \frac{5\rho}{16} \frac{x-y_1}{|x-y_1|} $, while if   $|x-y_1| < \frac{5 \rho}{16}$ we choose $y_2 = y_1$. It is straightforward to see from $|x-y_1|<R$ that with this choice we have  \eqref{eq:weak-harnack-3}. 

In order to show  \eqref{eq:weak-harnack-2}, we fix $(\tilde y,\tilde t) \in Q_{\frac{\rho}{16}}^-(y_2, \ft_k+ \frac{\rho^2}{4})$. This means that 
\begin{equation} \label{eq:weak-harnack-4}
 \frac{\rho^2}{4} - \frac{\rho^2}{16^2} < \tilde t  - \ft_k \leq  \frac{\rho^2}{4}  \quad \text{and} \quad \tilde y \in  B_{\frac{\rho}{16}}(y_2). 
\end{equation}
In particular, it follows from \eqref{eq:weak-harnack-3} and  \eqref{eq:weak-harnack-4} that 
\[
|\tilde y-x|^2  < \left(\frac{3\rho}{8}\right)^2  <  \frac{\rho^2}{4} - \frac{\rho^2}{16^2} <  \tilde t  - \ft_k <\rho^2. 
\]
By  the definition of $\mathcal{B}_r $, this implies $(\tilde y, \tilde t) \in   \mathcal{B}_\rho + (x,\ft_k)$.  We need yet to show that $(\tilde y, \tilde t) \in \mathcal{B}^-$, which is equivalent to $|\tilde y-y_1|^2 < t_1 -\tilde t$. From \eqref{eq:weak-harnack-3} it follows that $|\tilde y-y_1|^2  < (R -\rho/4)^2$. By \eqref{eq:weak-harnack-4} and by recalling  $R = \sqrt{t_1-\ft_k}$ we have
\[
t_1 -\tilde t = (t_1 - \ft_k) + (\ft_k  -\tilde t) = R^2  + (\ft_k  -\tilde t)  \geq R^2 - \frac{\rho^2}{4}.
\]
From $\rho \leq R$  it follows   $ (R -\rho/4)^2 < R^2 - \rho^2/4$. Therefore  we have $|\tilde y-y_1|^2 < t_1 -\tilde t$ and \eqref{eq:weak-harnack-2} follows. 

By the choice of $\rho$ it follows that there is a point $\tilde y_1 \in \bar B_{\rho}(x)$ such that $(\tilde y_1, \ft_k+ \rho^2 ) \in A_k$.  We may thus apply   Lemma  \ref{lem:basic-macro} for $(y,\ft) =  (x,\ft_k+ \rho^2)$ and $a = L^k a_0$, as long as it holds $L^k a_0 \leq 1$, to deduce that 
\[
\big|\A\big(L^{k+1}a_0; Q_{16\rho_0}^-(y,\ft)\big)  \cap \big(Q_{\frac{\rho}{16}}^-(y_2, \ft_k  + \frac{\rho^2}{4}) \big) \big| \geq \mu_1\rho^{n+2}.
\]
The inequality \eqref{eq:weak-harnack-1} then follows from above, together with \eqref{eq:weak-harnack-2} and  $\rho^{n+2} \geq c |\textbf{Q}_{(x,\ft_k)}|$.

Let us then consider the case $  \rho < \sqrt{\fh}$.   The construction implies that $\textbf{Q}_{(x,\ft_k)} = B_\rho(x) \times [\ft_k, \ft_k + \fh)$ with  $B_\rho(x) \times \{ \ft_k\} \cap A_k = \emptyset$ and there is a point $(\tilde y_1,\ft_k) \in A_k$ with $\tilde y_1 \in \partial B_\rho(x)$. Denote again $R = \sqrt{t_1-\ft_k} \geq \rho$. Since $x  \in  B_{R}(y_1)$, we may choose a point $y_2 \in \bar B_{\rho/2}(x)$ such that $B_{\rho/2}(y_2) \subset B_{\rho}(x) \cap B_{R}(y_1)$. Indeed, if $|x-y_1|\geq \frac{\rho}{2}$ we choose $y_2 = x - \frac{\rho}{2}\frac{x-y_1}{|x-y_1|}$, while if $|x-y_1| < \frac{\rho}{2}$ we choose $y_2 = y_1$.   Then it holds $B_{\rho/2}(y_2)  \times \{ \ft_k\} \subset \mathcal{B}^- \setminus A_k$.   We may use the first part of Lemma \ref{lem:basic-micro} for $y =y_2, y_1 = \tilde y_1$ and $\frac{\rho}{2}$ in place of $\rho$ and deduce 
\[
\Ha^n \big( \mathcal{A}\big(L^{k+1}a_0; Q_{16\rho_0}^-(y,\ft)\big) \cap \big( B_{\rho/2}(y_2) \times \{\ft_k\}  \big) \big)  \geq \mu_2 \rho^n. 
\]
Using $B_{\rho/2}(y_2)  \times \{ \ft_k\} \subset \mathcal{B}^- \setminus A_k$ the above  implies 
\[
\big|  \mathcal{A}\big(L^{k+1}a_0; Q_{16\rho_0}^-(y,\ft)\big) \cap \textbf{Q}_{(x,\ft_k)}\cap \mathcal{B}^- \setminus A_k  \big| \geq \mu \rho^n \fh \geq \mu |\textbf{Q}_{(x,\ft_k)}|
\]
and we have  \eqref{eq:weak-harnack-1} in the case  $  \rho < \sqrt{\fh}$.

We are left with the case $\sqrt{\fh} \leq   \rho \leq C_1 \sqrt{\fh}$. As above,  we may choose a point $y_2\in \bar B_{\fh/2}(x)$ such that $B_{\sqrt{\fh}/2}(y_2) \subset B_{\fh}(x) \cap B_{r}(y_1)$, where $r = \sqrt{t_1-\ft_k} \geq \rho$. Then by construction it holds $B_{\sqrt{\fh}/2}(y_2)  \times \{ \ft_k\} \subset \mathcal{B}^- \setminus A_k$, and there is a point $(\tilde y_1,\ft_{k+j}) \in A_k$ with $j \leq C_1+1$ and  $|\tilde y_1-x|^2 \leq (j+1) \fh$. We have  by the choice of $y_2$ that   $|\tilde y_1 -y_2| \leq |\tilde y_1 - x| + |x- y_2| \leq   \sqrt{(j+1)\fh}  + \sqrt{\fh}< (j+2)\sqrt{\fh}$. We have by the second  part of the Lemma \ref{lem:basic-micro}, applied to  $y =y_2, y_1 = \tilde y_1$, that 
\[
\Ha^n \big( \mathcal{A}\big(L^{k+1}a_0; Q_{16\rho_0}^-(y,\ft)\big) \cap \big( B_{\sqrt{\fh}/2}(y_2) \times \{\ft_k\}  \big) \big)  \geq \mu_3 \fh^{\frac{n}{2}}. 
\]
Since  $B_{\frac{\rho}{2}}(y_2) \times  \{\ft_k\} \subset \mathcal{B}^- \setminus A_k$, this implies 
\[
\big|\mathcal{A}\big(L^{k+1}a_0; Q_{16\rho_0}^-(y,\ft))\cap \textbf{Q}_{(x,\ft_k)}\cap \mathcal{B}^- \setminus A_k \big| \geq \mu \fh^{\frac{n}{2}+1}  \geq \mu |\textbf{Q}_{(x,\ft_k)}|
\]
and we have  \eqref{eq:weak-harnack-1} also in the third case    $\sqrt{\fh} \leq   \rho \leq C_1 \sqrt{\fh}$.

We may then apply a variant of the Vitali covering theorem, see e.g. \cite[Proposition 5.2]{AJ},  and choose a countable family of cylinders $\{ \textbf{Q}_i \}_{i\in I}$  from $\textbf{Q}_{(x,\ft_k)}$, which are pairwise disjoint and the larger cylinders   $ \{5\textbf{Q}_i\}_{i\in I}$ cover  $\mathcal{B}^-  \setminus A_k $. Therefore we have by \eqref{eq:weak-harnack-1} 
\[
\begin{split}
\big|\mathcal{B}^- \setminus A_k\big| &\leq \sum_{i \in I} |5\textbf{Q}_i| \leq  C  \sum_{i \in I} |\textbf{Q}_i|\\
&\leq \frac{C}{\mu_3}   \sum_{i \in I}\big|  \big( A_{k+1} \cap \textbf{Q}_{i} \big) \setminus  A_k  \big|  \leq  \frac{C}{\mu_3} \big| A_{k+1} \setminus  A_k  \big|. 
\end{split}
\]
Therefore 
\[
|\mathcal{B}^-  \setminus A_k | = |\mathcal{B}^-  \setminus A_{k+1}| + | A_{k+1} \setminus A_k| \geq |\mathcal{B}^-  \setminus A_{k+1}| + \frac{\mu_3}{C}|\mathcal{B}^-  \setminus A_k |, 
\]
which in turns implies 
\[
|\mathcal{B}^-  \setminus A_{k+1} | \leq \eta |\mathcal{B}^-  \setminus A_{k} | 
\]
for $\eta = 1 - \frac{\mu_3}{C} \in (0,1)$. By iterating this we have $|\mathcal{B}^-  \setminus A_k | \leq \eta^k |\mathcal{B}^- |$ and the estimate \eqref{eq:weak-harnack-0} follows. This concludes the proof. 
\end{proof}

\section{Flatness decay}

\subsection{Oscillation decay}

As in \cite{Ga, Sa, Wa}  the crucial nonlinear estimate that we need is the oscillation estimate which roughly states that the boundary functions $v_r^-$ and $ v_r^+$ defined in \eqref{def:v-r} are close to each other and H\"older continuous. As in  \cite{DGS, Ga}  this does not follow from a standard PDE argument, since  the boundary could have multiple layers where $v_r^- < v_r^+$. This is the case, for instance, when the boundary contains a thin catenoid, and we need the a priori estimates for the Willmore energy  and Proposition \ref{prop:kuwert-schaztle} to rule this out.  We formulate the key estimate that follows from the geometric assumptions in terms of the functions  $v_r^-$ and $v_r^+$. We stress that here we need to assume the dimension to be at most three. In the following lemma we specify carefully the dimension of the ball $B_r^n(x) \subset \R^n$.
\begin{lemma}
\label{lem:key-lemma}
Assume $n \leq 2$ and fix $\delta_0>0$. Let   $v_r^-\leq  v_r^+$ be as in  \eqref{def:v-r}, with $r \leq r_0$, and assume $\|v_r^{\pm}\|_{C^0(Q_1^-)} \leq 1$. Assume further that \eqref{eq:density-point} holds when $n=2$ or \eqref{eq:density-point-pl} holds when $n=1$   at $t_0\geq 1$ and  $ \eps_0 \leq \delta_0^3$. Fix $Q_\rho^-(\hat y,\hat \ft) \subset Q_{1/2}^-$ of size $\rho \geq 12 \eps_0^{\frac13}$ and $\rho \geq r^{\alpha/2}$. When $r_0$ is small enough,  it holds
\[
\big|\{ (y,\ft) \in Q_\rho^-(\hat y, \hat \ft) : v_r^-(y,\ft) < v_r^+(y,\ft) \} \big| \leq \Big(1- \frac{\delta_0}{120}\Big)|Q_\rho^-|.
\]

\end{lemma}
\begin{proof}
For simplicity  we assume $(\hat y,\hat \ft) = (0,0)$, since the argument is the same in the general case. 
Recall that for $n=2$ the assumption  \eqref{eq:density-point}  implies \eqref{eq:density-point-20} and \eqref{eq:density-point-2}, while for $n=1$ the assumption \eqref{eq:density-point-pl}  implies  \eqref{eq:density-point-20} and \eqref{eq:density-point-pl3}.   This means that there is  a set $ I_{\eps_0} \subset [t_0-1, t_0]$  such that 
\begin{equation} \label{eq:key-lemma-1}
 | I_{\eps_0} \cap [t_0-r^2,t_0]| \geq (1 - \eps_0)r^2
\end{equation}
and for all  $t \in I_{\eps_0}$ it holds  \eqref{eq:density-point-2} when $n=2$ and   \eqref{eq:density-point-pl3} when $n=1$.

Recall that $t = r^2 \ft +t_0$, let $\rho \in (0,1]$ be as in the statement and define $\mathcal{I}_{\eps_0} = \{ \ft \in [-\rho^2, 0]: t = r^2 \ft +t_0  \in I_{\eps_0}\}$. Since $I_{\eps_0}$ satisfies \eqref{eq:key-lemma-1} it holds 
\[
\big|[-\rho^2,0] \setminus \mathcal{I}_{\eps_0}\big| \leq \frac{1}{r^2} \big|[t_0-r^2,t_0] \setminus  I_{\eps_0}\big| \leq   \eps_0. 
\]
We may thus estimate 
\[
\big|\{ (y,\ft) \in Q_\rho^-:  \ft \notin \mathcal{I}_{\eps_0}   \} \big| \leq \eps_0 |B_\rho^n|. 
\]
The assumptions  $ \eps_0^{\frac13} \leq \delta_0$ and  $\eps_0^{\frac13} \leq \frac{\rho}{12}$ then imply 
\begin{equation} \label{eq:key-lemma-12}
\big|\{ (y,\ft) \in Q_\rho^-:  \ft \notin \mathcal{I}_{\eps_0}   \} \big| \leq \frac{\delta_0}{120}  |Q_\rho^-| .
\end{equation}

Next we claim that if $ \ft  \in\mathcal{I}_{\eps_0} $, then 
\begin{equation} \label{eq:key-lemma-2}
\big|\{ y \in B_\rho^n : v_r^-(y,\ft) < v_r^+(y,\ft)\} \big| \leq \Big(1- \frac{\delta_0}{60}\Big)|B_\rho^n|.
\end{equation}
In order to prove \eqref{eq:key-lemma-2} we state it in the original coordinates  \eqref{rescale}. Fix $ \ft  \in\mathcal{I}_{\eps_0}$ and denote $t =  r^2 \ft +t_0 \in I_{\eps_0}$. We write \eqref{eq:key-lemma-2}  using the functions $u_\pm(\cdot, t) : Q_r^- \to \R$ from \eqref{def:sub-super} as  
\begin{equation} \label{eq:key-lemma-3}
\big|\{ x \in B_{r\rho}^n : u_-(x,t) < u_+(x,t)\} \big| \leq \Big(1- \frac{\delta_0}{60}\Big)|B_{r \rho}^n|.
\end{equation}
We first prove \eqref{eq:key-lemma-3} for $n=2$.

The definition of $v_r^\pm$ in   \eqref{def:v-r},  the  assumption  $\|v_r^{\pm}\|_{C^0(Q_1^-)} \leq 1$ and  $ \rho \geq r^{\alpha/2}$ imply 
\begin{equation} \label{eq:key-lemma-4}
\| u_\pm(\cdot ,t) - \Lambda(t) \|_{C^0(B_{ r}^2)} \leq Cr^2   \leq C (r \rho)^{2-\alpha}. 
\end{equation}
By the assumption it holds $B_{2\rho r}^2 \subset B_{ r}^2$. Recall that we assume  $\alpha \leq \frac18$ which implies  $1+\alpha < 2 -\alpha$. This means that by \eqref{eq:key-lemma-4} we have the second assumption in Proposition \ref{prop:kuwert-schaztle} for the ball $B_{\rho r}^3(z)$ with $z = \Lambda(t) e_3$ when $r$ is small. Therefore since $t \in   I_{\eps_0}$, we have by  \eqref{eq:density-point-2}, by \eqref{eq:key-lemma-4}  and by Proposition \ref{prop:kuwert-schaztle} that 
\begin{equation} \label{eq:key-lemma-44}
\Ha^2\big( \pa E^{h}(t) \cap B_{r \rho}^3(z) \big)  \leq \big(2 \pi - \frac{\delta_0}{6}\big) (r \rho)^2 
\end{equation}
when $r$ is small. 

Let us denote $\Sigma = \{ x \in B_{r\rho}^2 : u_-(x,t) < u_+(x,t)\} \subset \R^2$ and define $\Gamma^\pm \subset \R^3$ as
\[
\Gamma^+ := \big{\{} \big(x,u_+(x,t) - \Lambda(t)\big) :  x \in \Sigma \big{\}} \quad \text{and} \quad \Gamma^- := \big{\{} \big(x,u_-(x,t) - \Lambda(t)\big) :  x \in \Sigma \big{\}}. 
\]
Then $\Gamma^+$ and $ \Gamma^-$ belong to the boundary of the set $E^{h}(t) - \Lambda(t) e_3$.  Moreover,   $\Gamma^+ \cap \Gamma^- = \emptyset $ and for $z = \Lambda(t) e_3$ 
\[
\Ha^2\big( \pa E^{h}(t) \cap B_{r \rho}^3(z) \big) \geq  \Ha^2\big(\Gamma^+ \cap   B_{r \rho}^3 \big) +   \Ha^2\big(\Gamma^- \cap   B_{r \rho}^3 \big) .
\]
Note that  trivially it holds $\Ha^2\big(\Gamma^\pm \big) \geq \Ha^2(\Sigma )$. However, when we project the sets $\Gamma^\pm \cap   B_{r \rho}^3$ we might miss the part of $\Sigma$ which is close to the boundary of the disk 
$B_{r\rho}^2$. However, using   \eqref{eq:key-lemma-4}, which states that $\Gamma^\pm$ are close to the $(x_1,x_2)$-plane,  we may conclude that this part is small. To be more precise, by an 
elementary geometric argument  we deduce from  \eqref{eq:key-lemma-4}  that 
\[
  \Ha^2\big(\Gamma^\pm \cap   B_{r \rho}^3 \big) \geq   \Ha^2\big(\Sigma \big) - C(r \rho)^{3-2\alpha} \geq \Ha^2\big(\Sigma \big) - \frac{\delta_0}{48}(r \rho)^{2}
\]
when $r$ is small. Using  \eqref{eq:key-lemma-44} and the two above inequalities yield 
\[
\begin{split}
\big(2 \pi - \frac{\delta_0}{6}\big) (r \rho)^2 &\geq \Ha^2\big( \pa E^{h}(t) \cap B_{r \rho}^3(z) \big) \\
&\geq  \Ha^2\big(\Gamma^+ \cap   B_{r \rho}^3 \big) +   \Ha^2\big(\Gamma^- \cap   B_{r \rho}^3 \big) \\
&\geq 2 \Ha^2\big(\Sigma \big) -   \frac{\delta_0}{24}(r \rho)^{2} . 
\end{split}
\] 
Therefore we have 
\[
2 \Ha^2(\Sigma ) \leq \big(2\pi - \frac{\delta_0}{6}\big) (r \rho)^2  +  \frac{\delta_0}{24}(r \rho)^{2} =  2 \pi\big(1 - \frac{\delta_0}{16\pi}\big)(r \rho)^{2}, 
\]
which implies \eqref{eq:key-lemma-3} in the case $n=2$.

In the planar  case $n=1$ we obtain  \eqref{eq:key-lemma-3}  trivially for all $t \in I_{\eps_0}$, because  \eqref{eq:density-point-pl} implies  \eqref{eq:density-point-pl3}. Since the boundary  $\partial E^h(t)$ is  a graph of a spherical function with small gradient, and since the excess is small, this 
immediately implies that the set $\{ x \in B_{r\rho}^1 : u_-(x,t) < u_+(x,t)\} $ is empty and we  trivially have  \eqref{eq:key-lemma-3}. 

Finally, combining \eqref{eq:key-lemma-12} and \eqref{eq:key-lemma-2} yields
\[
\begin{split}
\big|\{ (y,\ft) \in Q_\rho^- : &v_r^-(y,\ft) < v_r^+(y,\ft) \} \big| \\
 &\leq \big|\{ (y,\ft) \in Q_\rho^-:  \ft \notin \mathcal{I}_{\eps_0}   \} \big|  + \big|\{ (y,\ft)  \in B_\rho^2 \times \mathcal{I}_{\eps_0} : v_r^-(y,\ft) < v_r^+(y,\ft) \} \big|\\
&\leq \frac{\delta_0}{120}  |Q_\rho^-|  +\rho^2 \Big(1- \frac{\delta_0}{60}\Big)|B_\rho^n| =  \Big(1- \frac{\delta_0}{120}\Big)|Q_\rho^-| .\qedhere
\end{split}
\]
\end{proof}

We may proceed to  the oscillation decay. We define the oscillation in a cube $Q_\rho^-(y,\ft)$ as
\begin{equation} \label{def:osc-v-r}
\text{osc} \big(v_r^\pm;  Q_\rho^-(y,\ft)\big) =  \sup_{Q_\rho^-(y,\ft)} v_r^+ - \inf_{Q_\rho^-(y,\ft)}v_r^-, 
\end{equation}
where $v_r^\pm$ are defined in  \eqref{def:v-r}.  

If we would we able to prove  $\text{osc} \big(v_r^\pm;  Q_\rho^-(y,\ft)\big) \leq C \rho^\gamma$ in every cube, then we would be able to conclude  that $v_r^-(y,\ft) = v_r^+(y,\ft)$, which means that the boundaries of the evolving sets $E^h(t)$ are locally given by graph of functions $u_-(\cdot, t)$ defined in \eqref{def:sub-super}. Moreover this estimate would   imply  that $v_r^-$ is H\"older continuous. However,  we may only prove    $\text{osc} \big(v_r^\pm;  Q_\rho^-(y,\ft)\big) \leq C \rho^\gamma$  for $\rho \geq \eps_1$, where $\eps_1$ a small number depending on $r, \delta_0$ and $\eps_0$ in the statement of  Theorem \ref{thm:decay-flatness}. Such an estimate is enough for the flatness decay.

\begin{proposition}
\label{peop:osc-decay}
Assume $n \leq 2$ and fix $\delta_0>0$. Let   $v_r^-\leq  v_r^+$ be as in  \eqref{def:v-r}, with $C_0 \sqrt{\fh} \leq r \leq r_0$, and assume $\|v_r^{\pm}\|_{C^0(Q_1^-)} \leq 1$. 
Assume further that \eqref{eq:density-point} holds when $n=2$ or \eqref{eq:density-point-pl} holds when $n=1$   at $t_0\geq 1$ and  $ \eps_0 \leq \delta_0^3$. There are  constants $C \geq 1$ and $\gamma \in (0,1)$ such that  for every cylinder $Q_\rho^-(\hat y,\hat \ft) \subset Q_{\frac78}^-$ it holds 
\[
\text{osc} \big(v_r^\pm;  Q_\rho^-(\hat y,\hat \ft)\big) \leq C \rho^\gamma
\]
for all $\rho \geq \max \{ 12 \eps_0^{\frac13}, r^{\alpha/2}\}$ when $r_0, \eps_0$ are small enough. In particular, $\|v_r^+ - v_r^-\|_{C^0(Q_{\frac34}^-)} \leq C(\eps_0 + r_0)^{\gamma'}$ for $\gamma' = \frac{\alpha}{2} \gamma$. 
\end{proposition}

\begin{proof}
We may assume that $(\hat y,\hat \ft) = (0,0)$. Let $w_r^-$ and $w_r^+$ be as  in \eqref{def:w-r}.  By definition  it holds  $w_r^\pm(x, t) :=  v_r^\pm(x, t)  \pm  g(t)$ with $g(t)  = r^{-2\alpha} \int_{r^2t }^0 |\lambda^h(\tau+t_0+h)|\, d \tau$. By Proposition \ref{prop:apriori-est} it  holds  
$\int_{t_0-1}^{t_0+1} |\lambda^h(\tau)|^2, d \tau \leq C$, and therefore $g$ is H\"older continuous and 
\begin{equation} \label{eq:osc-decay-11}
\sup_{t \in (-\rho^2, 0]} |g(t)|  \leq r^{-2\alpha} \left( \int_{t_0-1}^{t_0+1} |\lambda^h(\tau)|^2\, d \tau  \right)^{\frac12} \left( (r \rho)^2 \right)^{\frac12}\leq C r^{1-2\alpha} \rho \leq \sqrt{r} \rho,
\end{equation}
when $\alpha \leq \frac18$ and $r$ is small enough. Let us denote 
\[
\text{osc} \big(w_r^\pm;  Q_{\rho}^-\big): =  \sup_{Q_\rho^-} w_r^+ - \inf_{Q_\rho^-}w_r^-.
\]
Since $w_r^\pm(x, t) :=  v_r^\pm(x, t)  \pm  g(t)$, by \eqref{eq:osc-decay-11} it is  enough to prove that $\text{osc} (w_r^\pm;  Q_\rho^- ) \leq C \rho^\gamma$ for  $\rho \geq \max \{ 12 \eps_0^{\frac13}, r^{\alpha/2}\}$. 

 We claim that there is $\eta \in (0,1)$ such that if $Q_{16 \rho}^- \subset  Q_{\frac78}^-$ then 
\begin{equation} \label{eq:osc-decay-2}
\text{osc} \big(w_r^\pm;  Q_{\rho}^-\big)   \leq \eta \,  \text{osc} \big(w_r^\pm;  Q_{16\rho}^-\big)  +  8 \rho
\end{equation}
for every $\rho \geq \max \{ 12 \eps_0^{\frac13}, r^{\alpha/2}\}$. The claim  then  follows from \eqref{eq:osc-decay-2} by standard iteration. 

To this aim we denote $M = \sup_{Q_{16\rho}^-} w_r^+ $ and by shifting the functions we may assume $\inf_{Q_{16\rho}^-} w_r^- = 0$. If $M \leq  8 \rho$ then \eqref{eq:osc-decay-2} is trivially true. Let us then assume $M > 8 \rho $.  We need to prove that either
\begin{equation} \label{eq:osc-decay-3}
\sup_{Q_{\rho}^-} w_r^+ \leq \eta M \qquad \text{or} \qquad \inf_{Q_{\rho}^-} w_r^- \geq (1-\eta) M. 
\end{equation}

We argue by contradiciton and assume that neither of the inequalities in \eqref{eq:osc-decay-3} is true. Since $\inf_{Q_{16\rho}^-} w_r^- = 0$, then $w_r^-$ is nonnegative in $Q_{16\rho}^-$ and the contradiction assumption implies 
$ \inf_{Q_{\rho}^-} w_r^- \leq (1-\eta) M$. Moreover by $\|v_r^{+}\|_{C^0(Q_1^-)} \leq 1$ and by \eqref{eq:osc-decay-11} it holds $M \leq 1+ \sqrt{r} \rho$. Therefore when $\eta$ is small enough,  it holds  $ \inf_{Q_{\rho}^-} w_r^-  \leq (1-\eta)M < m_0$, where $m_0$ is from Proposition \ref{prop:weak-harnack}. We may apply Proposition \ref{prop:weak-harnack} to deduce that there is $C_0$ such that 
\begin{equation} \label{eq:osc-decay-4}
\big|\big{\{} (y,\ft) \in   Q_{\rho}^-(0,-8\rho^2)  : \,   w_r^-(y,\ft) \leq C_0 (1-\eta) M   \big{\}} \big| \geq  \big( 1-\frac{\delta_0}{300} \big) |Q_{\rho}^-|. 
\end{equation}
On the other hand, if $\sup_{Q_{\rho}^-} w_r^+ \geq \eta M$, then we may apply the same estimate for function $M - w_r^+$ and conclude 
\begin{equation} \label{eq:osc-decay-5}
\big|\big{\{} (y,\ft) \in   Q_{\rho}^-(0,-8\rho^2)  : \,  M- w_r^+(y,\ft) \leq C_0 (1-\eta) M   \big{\}} \big| \geq  \big( 1-\frac{\delta_0}{300} \big) |Q_{\rho}^-|. 
\end{equation}

When  $\eta$ is close  to one, it holds $C_0 (1-\eta) \leq \frac18$. Then $w_r^-(y,\ft) \leq C_0 (1-\eta) M \leq \frac18 M$. Recall that   $w_r^-(y,\ft) :=  v_r^-(y,\ft)  - g(t)$  and by \eqref{eq:osc-decay-11} $\sup_{t \in (-\rho^2, 0]} |g(t)| \leq \rho$. Since we assume $M > 8 \rho$, then   $w_r^-(y,\ft)  \leq \frac18 M$ implies $v_r^-(y,\ft) < \frac14 M$, and by \eqref{eq:osc-decay-4} we have
\[
\big|\big{\{} (y,\ft) \in   Q_{\rho}^-(0,-8\rho^2)  : \,   v_r^-(y,\ft) < \frac14 M   \big{\}} \big| \geq  \big( 1-\frac{\delta_0}{300} \big) |Q_{\rho}^-|.
\]
Similarly  \eqref{eq:osc-decay-5} implies 
\[
\big|\big{\{} (y,\ft) \in   Q_{\rho}^-(0,-8\rho^2)  : \,   v_r^+(y,\ft) > \frac34 M    \big{\}} \big| \geq  \big( 1-\frac{\delta_0}{300} \big) |Q_{\rho}^-|. 
\]
In particular, these give
\[
\big|\big{\{} (y,\ft) \in   Q_{\rho}^-(0,-8\rho^2)  : \, v_r^-(y,\ft)  <  v_r^+(y,\ft)  \big{\}} \big| \geq  \big( 1-\frac{\delta_0}{150} \big) |Q_{\rho}^-|. 
\]
But this contradicts Lemma \ref{lem:key-lemma} and the claim \eqref{eq:osc-decay-3} follows. 
\end{proof}

\subsection{Flatness decay}

We proceed to the proof of Theorem \ref{thm:decay-flatness}. 

\begin{proof}[\textbf{Proof of Theorem \ref{thm:decay-flatness}}]

Let $A, \omega$ and $c$ be the parameters  in the excess \eqref{def:excess}, and we may choose the coordinates such that $\omega = e_{n+1}$ and $x_0= 0$. We associate  the space  $\Pi_{e_{n+1}}$ defined in \eqref{def:projection} with $\R^n$.  
Let $u_-$ and $u_+$ be the sub- and supergraphs defined in \eqref{def:sub-super}. 

\textbf{Step 1:} \, We claim first that there is $\sigma \in (0,1)$ such that 
\begin{equation} \label{eq:flatness-decay-1}
\|u_{\pm}(x',t_k)  - \Lambda(t_k) - \tilde P(x',t_k) -  p \cdot x' \|_{C^0(B_{\sigma r})} \leq \sigma^{3-\alpha} r^{2+\alpha} \quad \text{for all }\, t_k \in (t_0 - \sigma^2 r^2, t_0]
\end{equation}
for a caloric polynomial $\tilde P(\cdot, t) : \R^n \to \R$,  $\tilde P(x',t) = \frac12 \tilde A x' \cdot  x'  +\tilde b t +\tilde c$, with $|\tilde A - A| \leq c_0 r^\alpha,  |\tilde c - c| \leq c_0r^{2+\alpha}$ and a vector $ p \in \R^n$ with   $| p| \leq c_0r^{1+\alpha}$.  We remark that the claim does not follow from  \eqref{eq:flatness-decay-1} because of the additional  linear term $ p \cdot x' $. 

 We define the functions $v_r^+$ and $v_r^-$ as in \eqref{def:v-r}. The inequality \eqref{eq:flatness-decay-1} follows once we show that 
\begin{equation} \label{eq:flatness-decay-2}
\|v_r^{\pm}(y,\ft)   - \hat P(y,\ft) -\hat  p \cdot y \|_{C^0(Q_{\sigma }^-)} \leq \sigma^{3-\alpha}  
\end{equation}
for a caloric polynomial $\hat  P(x,t) = \frac12   \hat A x \cdot  x  +\hat b t +\hat c$, with $|\hat A|, |\hat p|, |\hat c| \leq c_0$. The constant $c_0$ is a number that does not depend on any parameter and its value will be clear from the proof.

Proposition \ref{peop:osc-decay} implies $\|v_r^+ - v_r^-\|_{C^0(Q_{\frac34}^-)} \leq C(\eps_0 + r_0)^{\gamma'}$. Therefore if we denote $\eps_1 =  \max \{ 12 \eps_0^{\frac13}, r_0^{\alpha/4}\}$, we have 
\begin{equation} \label{eq:flatness-decay-31}
v_r^+ \leq  v_r^- + \eps_1^{\gamma'} \qquad \text{in } \quad Q_{\frac34}^-. 
\end{equation}
Let $\varphi$ be the solution of 
\[
\begin{cases}
\pa_t \varphi = \Delta \varphi \,  \quad \text{in } \,  Q_{\frac34}^-, \\
\varphi  = v_r^-  \,  \quad \text{on } \, \pa_p Q_{\frac34}^-,
\end{cases}
\]
where $\pa_p Q_{\frac34}^-$ denotes the parabolic boundary. Proposition \ref{peop:osc-decay} yields $\text{osc} \big(v_r^-;  Q_\rho^-(\hat y,\hat \ft)\big) \leq C \rho^\gamma$ for $\rho \geq \eps_1$ and therefore by standard regularity theory  for heat equation we have, by arguing as in \cite[Proposition 5.2]{Wa}, that 
\begin{equation} \label{eq:flatness-decay-32}
v_r^- - C \eps_1^{\gamma}  \leq \varphi \leq  v_r^- + C \eps_1^{\gamma}   \quad \text{in }  \, \bar Q_{\frac34}^-\setminus Q_{\frac34 - \eps_1}^-. 
\end{equation}
and
\begin{equation} \label{eq:flatness-decay-5}
 \eps_1^2|\nabla^2 \varphi| +   \eps_1^4 |\pa_{tt}^2 \varphi| +  \eps_1^3 |\pa_{t} \nabla \varphi|  \leq C  \quad \text{in }  \, \bar Q_{\frac34 - \eps_1}^-, 
\end{equation}
by increasing $C$ if necessary. 

Denote $w_r^\pm$ the functions defined in \eqref{def:w-r}, and recall that  $w_r^\pm(x, t) :=  v_r^\pm(x, t)  \pm  g(t)$, where by \eqref{eq:osc-decay-11} it holds $\sup_{t \in (-1,0]} |g(t)| \leq \sqrt{r}$. 
Therefore we have by \eqref{eq:flatness-decay-31} and \eqref{eq:flatness-decay-32} that (recall that $\gamma' < \gamma$)
\begin{equation} \label{eq:flatness-decay-4}
w_r^+ - 2C \eps_1^{\gamma'}  \leq \varphi \leq  w_r^- + 2C \eps_1^{\gamma'}   \quad \text{in }  \, \bar Q_{\frac34}^-\setminus Q_{\frac34 - \eps_1}^-. 
\end{equation}

We claim that 
\begin{equation} \label{eq:flatness-decay-6}
w_r^+ - 4C \eps_1^{\gamma'}  \leq \varphi \leq  w_r^- + 4C \eps_1^{\gamma'}   \quad \text{in }  \,  Q_{\frac34 - \eps_1}^-.
\end{equation}
We only prove the latter inequality, as the first one follows from the same argument. We define $\tilde \varphi(x,t) = \varphi(x,t) - \eps_1 t$, notice that $\pa_t \tilde  \varphi - \Delta \tilde \varphi = -\eps_1$,  and claim that 
\[
\tilde \varphi(y,\ft_k) \leq  w_r^-(y,\ft_k) + 3C \eps_1^{\gamma'}   \quad \text{for all  }  \,  (y,\ft_k) \in  Q_{\frac34 - \eps_1}^-.
\]
By \eqref{eq:flatness-decay-4} this holds  in $Q_{\frac34}^-\setminus Q_{\frac34 - \eps_1}^-$. In order to prove the above inequality in $Q_{\frac34 - \eps_1}^-$, we claim that the maximum of the function $(y,\ft_k) \mapsto \tilde \varphi(y,\ft_k) - w_r^-(y, \ft_k)$  is attained in $\bar Q_{\frac34}^-\setminus Q_{\frac34 - \eps_1}^-$. Indeed, if this were not the case, then the function $\tilde \varphi +c'$, for some constant $c'$, would touch $w_r^-$ from below at some point $(y,\ft_k) \in Q_{\frac34 - \eps_1}^-$. The estimate \eqref{eq:flatness-decay-5} and $\eps_1 \geq r_0^{\alpha/4} \geq r^{\alpha/4} $ imply  $r^\alpha \|\tilde \varphi\|_{C_x^2} \leq 1$. Therefore Lemma \ref{lem-visco1}, applied to a slightly smaller cylinder, and the estimates \eqref{eq:flatness-decay-5} imply 
\[
\begin{split}
-\eps_1 = \partial_t \tilde \varphi(y, \ft_k) - \Delta \tilde \varphi(y, \ft_k) &\geq - Cr^{2-\alpha} -  \fh \|\partial_{tt}^2 \tilde \varphi \|_{C^0} - C \sqrt{\fh} \|\partial_{t} \nabla \tilde  \varphi \|_{C^0}\\
&\geq -Cr^{2-\alpha} - C \eps_1^{-4}   \fh  - C  \eps_1^{-3} \sqrt{  \fh }\\
&\geq  -Cr^{2-\alpha} - \frac{\eps_1}{2}
\end{split}
\]
when $\fh= \frac{h}{r^2} \leq \frac{1}{C_0^2}$ is small enough.  This yields $2Cr^{2-\alpha}  \geq \eps_1$, which is a contradiction when $r$ is small enough, since $\eps_1\geq r^{\alpha/4}$. This completes the proof of \eqref{eq:flatness-decay-6}. 

We use the fact that $\varphi$ is a solution of  heat equation and deduce that  there is a caloric polynomial  $\hat  P(x,t) = \frac12 \hat A x \cdot x  +\hat b t  +\hat  p\cdot x +\hat c$, with $|\hat A|, |\hat p|, |\hat c| \leq c_0$, such that 
\[
\|\varphi - \hat  P\|_{C^0(Q_{\sigma}^-)} \leq C \sigma^3 \qquad \text{for some small  } \, \sigma \in (0,1).
\] 
Using this with \eqref{eq:flatness-decay-6}    implies the estimate \eqref{eq:flatness-decay-2} when  $\eps_1$ is small enough.

\textbf{Step 2:} \, In order to prove the claim, we need to find new coordinates where the linear term $p\cdot x'$ in \eqref{eq:flatness-decay-1} vanishes. This follows from elementary linear algebra, but we give the argument for the reader's convenience.  We do this in $\R^3$, as the planar case follows from the same argument.  In this step we denote by $x \in \R^3$ a point in 3D and  claim that we may change the coordinate basis of $\R^3$ from $\{e_1,e_2, e_3\}$ to $\{\hat e_1, \hat e_2, \hat e_3\}$ such that in the new coordinates it holds for $x= \hat  x_1 \hat e_1 + \hat  x_2 \hat e_2 +\hat  x_3 \hat e_3 \in \pa E^h(t) \cap  C_{\sigma r, r_1}$ and all $t \in (t_0-\sigma^2 r^2,t_0]$  that
\begin{equation} \label{eq:flatness-decay-7}
\| \hat x_3 -   \Lambda(t) - \tilde{P} \big((\hat x_1, \hat x_2),t\big)\| \leq \sigma^{2+\alpha} r^{2+\alpha}, 
\end{equation}    
where $\tilde P$ is the polynomial in \eqref{eq:flatness-decay-1}.  The inequality \eqref{eq:flatness-decay-7}  then concludes the proof. 

If the vector $p \in \R^2$ in \eqref{eq:flatness-decay-1}  is zero then there is nothing to prove. Otherwise we choose the first two basis vectors  for $\R^3$ as $e_1 = \frac{p}{|p|}$ and $e_2$ orthogonal to it.   We define the matrix 
\[
T = \frac{1}{\sqrt{1+ |p|^2}} \begin{pmatrix}
1 & 0 & -|p| \\
0 & \sqrt{1+ |p|^2} & 0 \\
|p| & 0 & 1
\end{pmatrix}
\]
and choose the new basis vectors as $\hat e_i = Te_i$ for $i = 1,2,3$. 

In order to prove \eqref{eq:flatness-decay-7} we fix $t \in (t_0-\sigma^2 r^2,t_0]$ and $x \in \pa E^h(t) \cap  C_{\sigma r, r_1}$. We write the point $x$ in the old coordinates as $x = \tilde x_1  e_1 +\tilde x_2  e_2 + \tilde x_3 e_3$ and in the new coordinates as $x= \hat  x_1 \hat e_1 + \hat  x_2 \hat e_2 +\hat  x_3 \hat e_3$, and   denote $\Lambda = \Lambda(t)$ and $\tilde P = \tilde P\big((\hat x_1, \hat x_2),t\big)$ for short. We choose    $x_P = \hat x_1 \hat e_1 + \hat x_2 \hat e_2 + (\Lambda +\tilde P) \hat e_3$ and observe that the inequality  \eqref{eq:flatness-decay-7} is equivalent to   $|x- x_P| \leq  \sigma^{2+\alpha} r^{2+\alpha}$.

 We may  relate the old coordinates  $x =  \tilde x_1  e_1 +\tilde x_2  e_2 + \tilde x_3 e_3$   with the new ones by 
\begin{equation} \label{eq:flatness-decay-8}
Tx = T  \begin{pmatrix} \hat x_1\\ \hat x_2 \\ \hat x_3 \end{pmatrix} = \frac{1}{\sqrt{1+ |p|^2}} \begin{pmatrix} \hat x_1 - |p|\hat x_3\\ \sqrt{1+ |p|^2} \hat x_2 \\ \hat x_3 + |p|\hat x_1 \end{pmatrix}  =  \begin{pmatrix} \tilde x_1 \\ \tilde x_2 \\ \tilde x_3  \end{pmatrix}.
\end{equation}
Then $\hat x_2 = \tilde x_2$. We use \eqref{eq:flatness-decay-8} with  $|p| \leq C r^{1+\alpha}$, $|\Lambda(t)| \leq Cr$ and $|\hat x_i| \leq C r$ for $i =1, 3$, which follow from \eqref{eq:flatness-decay-1},    to deduce that $|\hat x_i-\tilde x_i| \leq Cr^{2+\alpha}$  for $i = 1,3$, and 
\[
|\hat x_1 - (\tilde x_1 +p|\tilde x_3|)| \leq C r^{3+2\alpha} \quad \text{and} \quad |\hat x_3 - (\tilde x_3 - p|\tilde x_1|)| \leq C r^{3+2\alpha} .
\]
In particular, these imply $|\tilde P - \tilde P\big((\tilde x_1,\tilde x_2), t\big)| \leq Cr^{3+\alpha}$. Trivially it also holds  $|\tilde P| \leq Cr^2$.  Therefore  we may estimate  
\begin{equation} \label{eq:flatness-decay-9}
Tx_P =   T  \begin{pmatrix} \hat  x_1\\ \hat x_2 \\  \Lambda +\tilde P \end{pmatrix} = \frac{1}{\sqrt{1+ |p|^2}} \begin{pmatrix} \hat x_1 - |p| (\Lambda + \tilde P)\\ \sqrt{1+ |p|^2} \hat x_2 \\  \Lambda +\tilde P + |p|\hat x_1 \end{pmatrix} = \begin{pmatrix} \tilde x_1 + |p|(\tilde x_3 - \Lambda - \tilde P)\\ \tilde x_2 \\  \Lambda +\tilde P + |p|\tilde x_1 \end{pmatrix} +R,
\end{equation}
for a remainder which satisfies $|R| \leq Cr^{2+2\alpha}$. By the choice of the vector $e_1$ it holds $|p|\tilde x_1  = p \cdot  x'$ for $x' = \tilde x_1 e_1 + \tilde x_2 e_2$. Since $x = x' +\tilde x_3e_3  \in \pa E^h(t) \cap  C_{\sigma r, r_1} $ we have by  \eqref{eq:flatness-decay-1} that 
\[
\big|\tilde x_3 -  \Lambda - \tilde P\big((\tilde x_1, \tilde x_2), t\big)-  |p|\tilde x_1 \big| \leq \sigma^{3-\alpha} r^{2+\alpha}. 
\]
Note that this together with $|\tilde P - \tilde P\big((\tilde x_1,\tilde x_2), t\big)| \leq Cr^{3+\alpha}$  implies also that 
\[
|p|\big|\tilde x_3 - \Lambda - \tilde P \big| \leq C r^{1+\alpha} \big( \big|\tilde x_3 -  \Lambda - \tilde P  -   |p||\tilde x_1| \big| + |p||\tilde x_1|\big) \leq   Cr^{3+2\alpha}. 
\]
By combining \eqref{eq:flatness-decay-8} and \eqref{eq:flatness-decay-9}, and using the two previous inequalities,  we deduce that 
\[
|x - x_P| = |Tx - Tx_P| \leq \sigma^{3-\alpha} r^{2+\alpha} + Cr^{2+2\alpha} \leq \sigma^{2+\alpha} r^{2+\alpha}
\]
when $r \leq r_0$ is small enough.  This concludes the proof. 
\end{proof}

\section{Final regularity estimate and the proof of the main theorem}

\subsection{Final regularity estimate}

In this section we conclude the proof of the regularity result by proving that the set $E^h(t_0)$ is $C^2$-regular near the point $x_0$. This will follow from  Theorem \ref{thm:decay-flatness} together with results from minimal surfaces. 

 We first recall the standard Schauder estimate. 
\begin{proposition}
\label{prop:schauder}
Assume $u :B_2 \to \R$ is a bounded classical solution of  the equation
\[
\text{Tr}\left(A(x)\nabla^2 u \right) = f \qquad \text{in } \, B_2, 
\]
where $A(\cdot)$ is uniformly elliptic and $\|A(\cdot)\|_{C^{\beta}} \leq C$ for $\beta \in (0,1)$.  If $f$ is bounded, then for every $\gamma \in (0,1)$ it holds 
\begin{equation}\label{prop:schauder-1}
\|u\|_{C^{1+\gamma}(B_1)} \leq C \big( \|u\|_{C^0(B_2)} +\|f\|_{C^0(B_2)}   \big).
\end{equation}
If $f$ is $\beta$-H\"older continous  then it holds
\begin{equation}\label{prop:schauder-2}
\|u\|_{C^{2+\beta}(B_1)} \leq C \big( \|u\|_{C^0(B_2)} +\|f\|_{C^\beta(B_2)}   \big).
\end{equation}
\end{proposition}

We will need the standard interpolation inequality (see e.g. \cite{Tri}), which reads as follows for $0 < \beta < \gamma <1$:
\begin{equation}\label{eq:interpolation}
\|f\|_{C^{1+\beta}(B_1)} \leq C\|f\|_{C^{1+\gamma}(B_1)}^\theta\|f\|_{C^0(B_1)}^{1-\theta}, \quad \text{for }\, \theta = \frac{1+\beta}{1+\gamma}. 
\end{equation}

We also need the following simple lemma, which is a weak version of  the so called Danskin's theorem. 
\begin{lemma}
\label{lem:danskin}
Let $f(\cdot, \cdot): \R^{n}\times \R^n \to \R$ be a $C^1$-function, let $K \subset \R^n$ be a compact set and define
$F(x) = \min_{y\in K} f(x,y)$. At the points of differentiability of $F$ it holds 
\[
|\nabla F(x)| \leq |\nabla_x f(x,y_x)|, 
\]
for any $y_x \in K$ with  $f(x,y_x)= F(x)$. 
\end{lemma}

\begin{proof}
Note that $F$ is Lipschitz continuous and thus differentiable almost everywhere. Fix $x$ and $y_x \in K$ as in the statement. If $|\nabla F(x)| = 0$ the inequality is trivially true. Otherwise choose $\omega = \frac{\nabla F(x)}{|\nabla F(x)|}$. It holds for $\tau>0$
\[
\frac{F(x + \tau \omega)- F(x)}{\tau} \leq \frac{f(x+\tau \omega,y_x) - f(x,y_x)}{\tau}. 
\]
Letting $\tau \to 0$ yields $|\nabla F(x)| \leq \nabla_x f(x,y_x)\cdot \omega \leq | \nabla_x f(x,y_x)|$. 
\end{proof}

Here is the statement of the proposition. 

\begin{proposition}
\label{lem:final-reg}
Let $\{ E^h(t)\}_{t \geq 0}$ be an approximative flat flow in $\R^{n+1}$ with $n \leq 2$, fix $(x_0,t_0) \in \pa E^h(t_0)$, $t_0 \geq 1$, $C_2 \geq 1$ and  small $\alpha,  \delta_0 >0$.  There are   $r_0, \eps_0>0$ and $C_0 \geq 1$ such that assuming  \eqref{eq:density-point} if $n=2$ or    \eqref{eq:density-point-pl} if $n=1$  and \eqref{eq:flatt-assumption} for $C_0 \sqrt{h} \leq r \leq r_0$ for some $(A,\omega,c) \in  S^2  \times \mathbb{S}^2 \times  \R$ with  $|A|, |c|\leq C_2$, then  for all $t \in (t_0-\hat \rho^2,t_0]$, where  $\hat \rho = \hat C \sqrt{h}$ for some $\hat C\geq1$,  there is a function $u(\cdot, t) : B_{\hat \rho}^n \to \R$ with $\|u(\cdot, t)\|_{C^{2}(B_{\hat \rho}^n)} \leq C$ and an isometry $L : \R^{n+1} \to \R^{n+1}$ such that 
\[
L\big(\pa E^h(t)\big) \cap  C_{\hat \rho, r_1} =    \big{\{}  (x',u(x',t)) \in \R^{n+1} : x' \in B_{\hat \rho}^n\big{\}} \cap C_{\hat \rho, r_1}. 
\]
In particular, the second fundamental form  of $\pa E^h(t_0)$ at $x_0$ is bounded, i.e., $|B_{E^h(t_0)}(x_0)| \leq C$.  
\end{proposition}

\begin{proof}
We may assume that $x_0= 0$ and denote  $C_r = C_{r, r_1} =  B_r^n \times (-r_1,r_1)$ and $ \rho = C_0\sqrt{h}$  for short.  By iterating the estimate from Theorem \ref{thm:decay-flatness}, we conclude that there is a caloric polynomial $P(y,t) = \frac12  Ay\cdot y   +bt$, $b = \text{Tr}(A)$,  with  $|A| \leq C$  and coordinate basis $e_i$ such that for the set
\[
\textbf{P}_t =  \big{\{}  (x', x_{n+1}) \in \R^{n+1} : x_{n+1}   < P(x',t) +\Lambda(t) \big{\}}, 
\]
where $\Lambda(\cdot)$ is defined in \eqref{def:Lambda},   it holds 
\begin{equation}
\label{eq:final-reg-1}
\sup_{ x\in  (E^h(t) \Delta \textbf{P}_t) \cap C_{\rho }}   |x_{n+1} - P(x',t-t_0) - \Lambda(t)| \leq C h^{1+\tfrac{\alpha}{2}} \quad \text{for all } \, t \in (t_0 -\rho^2, t_0]. 
\end{equation}
Here we have used the fact that $(0,t_0) \in \pa E^h(t_0)$ and $\Lambda(t_0) = 0$ in order to ignore the constant $c$ in the caloric polynomial. 
Since $|P(x',t)| \leq C h$ for all $|x'| \leq \rho$ and  $t \in (t_0 -\rho^2, t_0]$, \eqref{eq:final-reg-1} implies 
\begin{equation}
\label{eq:final-reg-2}
\pa E^h(t) \cap C_\rho \subset \{ (x',x_{n+1}) \in B_\rho^n \times\R : |x_{n+1} - \Lambda(t)| \leq Ch \}.  
\end{equation}
This information is crucial since it implies  that the boundary $\pa E^h(t) \cap C_\rho $ is trapped in a narrow strip at height  $\Lambda(t)$ of  width  $Ch$. This makes it easy to estimate the geometric distance between the sets $E^h(t-h)$ and $E^h(t)$.

\medskip

\textbf{Step 1:} \, Let us first show that the mean curvature $\mathrm{H}_{E^h(t)}$ is bounded in  $ C_{\rho/2}= B_{\rho/2}^n \times (-r_1,r_1)$. To this aim we recall the Euler-Lagrange equation \eqref{eg:Euler-Lag}, which  reads as 
\begin{equation}
\label{eq:final-reg-3}
\frac{d_{E^h(t-h)} }{h} = - \mathrm{H}_{E^h(t)} + \lambda^h(t)  \qquad \text{on }\, \partial E^h(t). 
\end{equation}
Let us fix $t = kh$. We use \eqref{eq:final-reg-2} and the fact that $\Lambda(t) - \Lambda(t-h) = \lambda^h(t)h$ to deduce that for every $x \in \pa E^h(t) \cap C_\rho$ and $y \in \pa E^h(t-h) \cap C_\rho$ we have
\[
|x_{n+1} - y_{n+1} - \lambda^h(t) h| \leq |x_{n+1} - \Lambda(t)| + |y_{n+1} - \Lambda(t-h)| \leq C h.
\]
From here it follows from an elementary argument that  for every  $x \in \pa E^h(t) \cap C_{\rho/2}$ it holds
\begin{equation}
\label{eq:final-reg-30}
| d_{E^h(t-h)}(x) - \lambda^h(t) h| \leq Ch. 
\end{equation}
Then from \eqref{eq:final-reg-3} we deduce 
\begin{equation}
\label{eq:final-reg-31}
|\mathrm{H}_{E^h(t)}(x)|\leq C   \qquad \text{for  }\, x \in  \pa E^h(t) \cap  C_{\rho/2}.
\end{equation}

We may use Allard regularity theory \cite[Theorem 5.2]{Sim} in order to deduce that the boundary  $\pa E^h(t)$ is $C^{1+\alpha}$ regular in $C_{\rho'}$ for $\rho' = C' \sqrt{h}$ and for  a small  $\alpha >0$. We only need to verify that the multiplicity is close to one, which we do next by using the minimality of $E^h(t)$. We choose $c \in \R$ such that the half-space  $H_c = \{ x \in \R^{n+1} :  x_{n+1} <c\}$ satisfies $| H_c \cap C_{\rho/2}| = |E^h(t) \cap C_{\rho/2} |$ and define $F = \big( E^h(t) \setminus C_{\rho/2}\big) \cup \big( H_c \cap C_{\rho/2} \big)$, i.e.,  the set $F$ agrees with $ E^h(t)$ outside the cylinder $C_{\rho/2}$ and is flat inside  $C_{\rho/2}$.   Note that by \eqref{eq:final-reg-2} it holds $|c -  \Lambda(t)| \leq Ch$.  Then $|F| = |E^h(t)|$ and  the minimality of $E^h(t)$ implies  
\begin{equation}
\label{eq:final-reg-32}
P(E^h(t)) \leq P(F) + \frac{1}{h} \left( \int_{F\setminus E^h(t)} d_{E^h(t-h)}\, dx - \int_{E^h(t)\setminus F} d_{E^h(t-h)}\, dx  \right) . 
\end{equation}
From \eqref{eq:final-reg-2}, \eqref{eq:final-reg-30} and from $|c -  \Lambda(t)| \leq Ch$ we deduce that it holds  $|d_{E^h(t-h)} - \lambda^h(t) h| \leq Ch$ in  $F \Delta E^h(t)$ . Since $|F| = |E^h(t)|$, we have $|F\setminus E^h(t)| = |E^h(t)\setminus F|$. Therefore 
\[
\begin{split}
&\int_{F\setminus E^h(t)} d_{E^h(t-h)}\, dx - \int_{E^h(t)\setminus F} d_{E^h(t-h)}\, dx  \\
&\,\,\,\,\,\,= \int_{F\setminus E^h(t)} (d_{E^h(t-h)} - \lambda^h(t) h) \, dx - \int_{E^h(t)\setminus F} (d_{E^h(t-h)} - \lambda^h(t) h)  \, dx  \leq Ch |F \Delta E^h(t)|.
\end{split}
\]
Combining this with  \eqref{eq:final-reg-32}  yields
\[
P(E^h(t)) \leq P(F)  + C  |F \Delta E^h(t)|.
\]
By construction it holds $|E^h(t) \Delta F| \leq C h\rho^n $ and $P(F)  \leq P(E^h(t); \R^{n+1} \setminus \bar C_{\rho/2}) + |B_{1}^n|\left(\frac{\rho}{2}\right)^n + Ch \rho^{n-1}$, where the last term is due to $\Ha^{n}\big( (H_c \Delta E^h(t))\cap \pa  C_{\rho/2} \big) \leq Ch \rho^{n-1}$, which in turn follows from \eqref{eq:final-reg-2} and $|c -  \Lambda(t)| \leq Ch$. Therefore 
\[
P(E^h(t); C_{\rho/2}) \leq |B_{1}^n|\left(\frac{\rho}{2}\right)^n + Ch \rho^{n-1}. 
\]
Since $\rho = C_0 \sqrt{h}$, this means that the multiplicity is close to one when $h$ is small, and we deduce that there are $\rho' = C' \sqrt{h}$  and function $u(\cdot, t) $ with $\|u(\cdot, t) \|_{C^{1+\alpha}(B_{2\rho'})} \leq C$ such that  
\[
 \pa E^h(t) \cap C_{2\rho'} = \{ \big(x', u(x',t)\big) : x' \in B_{2\rho'}^n \}.  
\]

In the planar case $n=1$, the above estimate and the curvature bound \eqref{eq:final-reg-31} imply the claim. In the case $n=2$ we need further estimates, as the bound on mean curvature does not immediately imply 
bound on the second fundamental form. From now on we assume $n=2$.  Note that we may  write the estimate \eqref{eq:final-reg-1} as
 \begin{equation}
\label{eq:final-reg-4}
 |u(x',t) - P(x',t-t_0) -\Lambda(t)| \leq C h^{1+\tfrac{\alpha}{2}} \qquad \text{for all } \,  x' \in B_{2\rho'}^2. 
\end{equation}

\medskip

\textbf{Step 2:} \, Let us  show that for all $t \in (t_0 - \rho'^2, t_0]$  the function $u(\cdot, t)$ in \eqref{eq:final-reg-4} satisfies 
 \begin{equation}
\label{eq:final-reg-41}
 |\nabla u(x',t) - \nabla P(x',t-t_0)| \leq C h^{\tfrac12+\tfrac{\alpha}{8}} \qquad \text{for all } \,  x' \in B_{\rho'}^2. 
\end{equation}

To this aim we define
\[
v_{\rho'}(x') := \frac{u(\rho' x', t) - P(\rho' x',t-t_0)  -\Lambda(t)}{\rho'}. 
\]
Since $u(\cdot, t)$ is uniformly $C^{1+\alpha}$-regular we have by \eqref{eq:final-reg-31} that $v_{\rho'}$  is a solution of a uniformly elliptic equation
\[
\text{Tr}\left(A(x)\nabla^2 v_{\rho'} \right) = \rho' f \qquad \text{in } \, B_2 
\] 
for $\|f\|_{C^0(B_{2})} \leq C$. Applying \eqref{prop:schauder-1} from Proposition \ref{prop:schauder} with $\gamma = 1-\alpha$ we have 
\[
\|v_{\rho'}\|_{C^{2-\alpha}(B_{1})} \leq C (\|v_{\rho'}\|_{C^0(B_2)} + \rho' \|f\|_{C^0(B_2)} ). 
\]
We have   $\rho' = C' \sqrt{h} \geq  \sqrt{h}$ and therefore \eqref{eq:final-reg-4} yields  $\|v_{\rho'}\|_{C^0(B_2)} \leq C h^{\tfrac12+\tfrac{\alpha}{2}}$.  Thus we have $\|v_{\rho'}\|_{C^{2-\alpha}(B_{2})}  \leq C \sqrt{h}$. 
We use the interpolation inequality \eqref{eq:interpolation} with $\beta = \alpha$ and $\gamma = 1-\alpha$ 
\[
\|v_{\rho'}\|_{C^{1+\alpha}(B_1)} \leq C\|v_{\rho'}\|_{C^{2-\alpha}(B_1)}^\theta\|v_{\rho'}\|_{C^0(B_1)}^{1-\theta} \leq C h^{\frac{\theta}{2}}  \, (h^{\tfrac12 + \tfrac{\alpha}{2}})^{1-\theta} = C h^{\tfrac12 + (1-\theta)\tfrac{\alpha}{2}},
\]
for  $\theta = \frac{1+\alpha}{2-\alpha}$. When $\alpha$ is small it holds $1-\theta \geq \frac{1}{4}$. Therefore we obtain 
\[
\|v_{\rho'}\|_{C^{1}(B_1)} \leq \|v_{\rho'}\|_{C^{1+\alpha}(B_1)}   \leq Ch^{\tfrac12+\tfrac{\alpha}{8}} . 
\]
The claim \eqref{eq:final-reg-41} follows from $\nabla v_{\rho'}(x') = \nabla u(\rho' x',t) - \nabla P(\rho 'x',t-t_0)$.

\medskip

\textbf{Step 3:} \, We show that the mean curvature of $E^h(t)$ is H\"older continuous. We define $d: B_{\rho'} \to \R$ as
 \begin{equation}
\label{eq:final-reg-60}
d(x') = d_{E^h(t-h)}\big(x',u(x',t)\big) 
\end{equation}
and recall that if $d(x') \geq 0$ then $d(x') = \min_{y' \in \R^2} f(x',y') $ for 
 \begin{equation}
\label{eq:final-reg-61}
f(x',y') =  \sqrt{|x'-y'|^2 + (u(x',t)- u(y',t-h))^2}. 
\end{equation}
 In particular,  $d$ is Lipschitz continuous and we claim that 
 \begin{equation}
\label{eq:final-reg-6}
|\nabla d(x') | \leq C h^{\frac12 +\frac{\alpha}{8}}  \quad \text{a.e. } \, x' \in B_{\rho'/2}.  
\end{equation}

Let us assume $d$ is differentiable at $x'$ and $d(x')\geq 0$, the case $d(x')<0$ being similar. If $\nabla d(x')  = 0$ then \eqref{eq:final-reg-6} is trivially true. We assume $|\nabla d(x')|>0$ in which case 
$d(x' + \tau \omega) \geq 0$ for small $\tau >0$ and $\omega = \frac{\nabla d(x')}{|\nabla d(x')|}$. We apply Lemma \ref{lem:danskin} to the  function in \eqref{eq:final-reg-61} and obtain 
\[
|\nabla d(x') | \leq \frac{|(x'-y') + \big(u(x',t)- u(y',t-h)\big)\nabla u(x',t)|}{ \sqrt{|x'-y'|^2 + (u(x',t)- u(y',t-h))^2}},
\]
where $y'$ is a point where the minimum in $d$ is attained. We use the minimality of $y'$ and differentiate the function  in  \eqref{eq:final-reg-61} with respect to $y'$ and have
 \begin{equation}
\label{eq:final-reg-611}
(x'-y') + \big(u(x',t)- u(y',t-h)\big)\nabla u(y',t-h) = 0. 
\end{equation}
Therefore we conclude 
\[
|\nabla d(x') | \leq \frac{|u(x',t)- u(y',t-h)| |\nabla u(x',t)- \nabla u(y',t-h)|}{ \sqrt{|x'-y'|^2 + (u(x',t)- u(y',t-h))^2}} \leq |\nabla u(x',t)- \nabla u(y',t-h)|. 
\]
Next we use \eqref{eq:final-reg-41} and the fact that  $\nabla P(x',t-t_0) = Ax'$ for all $t$ and obtain 
 \begin{equation}
\label{eq:final-reg-62}
\begin{split}
|\nabla d(x') | &\leq  |\nabla u(x',t)- \nabla u(y',t-h)| \\
&\leq  |\nabla u(x',t)- Ax' |  +|Ax' - Ay'| +  |\nabla u(y',t-h)- Ay'| \\
&\leq   C h^{\tfrac12+\tfrac{\alpha}{8}} + C|x'-y'|. 
\end{split}
\end{equation}

In order to estimate the distance $x'-y'$ we use \eqref{eq:final-reg-611} and have  $|x'-y'| \leq  |u(x',t)- u(y',t-h)||\nabla u(y',t-h)|$. We use first \eqref{eq:final-reg-4}, together with $|\Lambda(t)- \Lambda(t-h)| = |\lambda^h(t)| h \leq C \sqrt{h}$, where the inequality follows from Proposition \ref{prop:apriori-est},  and $|P(x',t-t_0) |\leq C h $, and have 
\[
\begin{split}
 &|u(x',t)- u(y',t-h)| \\
&\leq  |u(x',t) - P(x',t-t_0) -\Lambda(t)|  +   |u(y',t-h) - P(y',t-t_0-h) -\Lambda(t-h)| + C \sqrt{h} \\
&\leq C \sqrt{h}.   
\end{split}
\]
Then we use \eqref{eq:final-reg-41} and $| \nabla P(y',t-h-t_0) | = |A x'|\leq C \sqrt{h}$ to deduce $|\nabla u(y',t-h)| \leq \sqrt{h}$. Combining the  previous estimates yield $|x'-y'| \leq C h$ and the claim \eqref{eq:final-reg-6} follows from \eqref{eq:final-reg-62}.

We   prove that the function $x' \mapsto \mathrm{H}_{E^h(t)}\big(x', u(x',t)\big)$ is H\"older continuous and  claim that  
 \begin{equation}
\label{eq:final-reg-5}
|\mathrm{H}_{E^h(t)}\big(x', u(x',t)\big)- \mathrm{H}_{E^h(t)}\big(y', u(y',t)\big)| \leq C |x' -y'|^{\frac{\alpha}{4}} \quad \text{for all } \, x', y' \in   B_{\rho'/2}. 
\end{equation}
Indeed, we differentiate the Euler-Lagrange equation \eqref{eq:final-reg-3}, use the notation \eqref{eq:final-reg-60} and the inequality \eqref{eq:final-reg-6} and  have for all $x',y' \in  B_{\rho'/2}$ (recall that $\rho' = C' \sqrt{h}$)
\[
\begin{split}
\mathrm{H}_{E^h(t)}\big(x', u(x',t)\big) &- \mathrm{H}_{E^h(t)}\big(y', u(y',t)\big)| = -  \frac{1}{h} \int_0^1 \nabla d(\tau x' + (1-\tau)y') \cdot  (x'-y') \, d \tau  \\
&\leq \frac{|x'-y'|}{h} \int_0^1 |\nabla d(\tau x' + (1-\tau)y')| \, d \tau \\
& \leq C |x'-y'| \,  h^{-\frac12 +\frac{\alpha}{8}} \\
&= C |x'-y'|^{\frac{\alpha}{4}}  \,   \big( \frac{|x' -y'|}{\sqrt{h}}\big)^{1 -\frac{\alpha}{4}} \leq C |x'-y'|^{\frac{\alpha}{4}} . 
\end{split}
\]
Hence, we have \eqref{eq:final-reg-5}. 

\medskip

\textbf{Step 4:} \, We have thus proved that the mean curvature is  H\"older continuous on  $\pa E^h(t) \cap C_{\rho'/2}$. Therefore the boundary can be written as a graph of a function $u(\cdot, t)$ which is a solution of a uniformly elliptic equation
\[
\text{Tr}\left(\tilde A(x)\nabla^2 u(\cdot, t) \right) = \tilde f \qquad \text{in } \, B_{\rho'/2} 
\]
with $\|\tilde f\|_{C^{\frac{\alpha}{4}}(B_{\rho'/2})}\leq C$.  We define for $\hat \rho = \rho'/4$
\[
w_{\hat \rho}(x') := \frac{u(\hat \rho x', t) - P(\hat \rho x',t-t_0)- \Lambda(t)}{\hat \rho^2} 
\]
and deduce that it is a solution of 
\[
\text{Tr}\left(\hat A(x)\nabla^2 w_{\hat \rho} \right) = \hat  f \qquad \text{in } \, B_{2},  
\]
where $\hat A$ is uniformly elliptic and H\"older continuous and by \eqref{eq:final-reg-5} it holds $\|\hat  f\|_{C^{\frac{\alpha}{4}}(B_2)} \leq C$. The estimate \eqref{prop:schauder-2} from Proposition \ref{prop:schauder} implies 
\[
\|w_{\hat \rho}\|_{C^{2 + \alpha/4}(B_1) } \leq C \big( \|w_{\hat \rho}\|_{C^{0}(B_2) }  + \|\hat f \|_{C^{\alpha/4}(B_2) } \big). 
\]
Using \eqref{eq:final-reg-4} and $\hat \rho = \hat C \sqrt{h}  \geq \sqrt{h}$ we have $ \|w_{\hat \rho}\|_{C^{0}(B_2)} \leq C h^{\frac{\alpha}{2}}$. This  and $\|\hat  f\|_{C^{\frac{\alpha}{4}}(B_2)} \leq C$ imply 
\[
\|w_{\hat \rho}\|_{C^{2}(B_1) } \leq \|w_{\hat \rho}\|_{C^{2 + \alpha/4}(B_1) } \leq C .
\]
Since  $\nabla^2 w_{\hat \rho} (x') = \nabla^2 u(\hat \rho x',t) - \nabla^2 P(\hat \rho x',t-t_0) = \nabla^2 u(\hat \rho x',t) - A$ for every $x' \in B_1$ , we finally conclude
\[
\|\nabla^2 u(\cdot ,t)\|_{C^2(B_{\hat \rho})} \leq C   
\]
and the claim follows. 
\end{proof}

We remark that with a little more work one may prove that the function $u$ in the statement of Proposition \ref{lem:final-reg} is in fact $C^{2,\gamma}$-regular for $\gamma = \frac{\alpha}{4}$. 

\subsection{Proof of the main theorem}

\begin{proof}[\textbf{Proof of Theorem \ref{mainthm}}]
Let $\{E(t)\}_{t \geq 0}$ be  volume-preserving flat flow as in the statement  and  let  $\{E^{h_n}(t)\}_{t \geq 0}$ be the associated  approximate  flat flow which converges to it. We simplify the notation by $E^{h}(t) =E^{h_n}(t)$.
By scaling we may assume that $|E^h(t)| = |E_0| = |B_1|$ for all $t>0$. By Proposition \ref{prop:JuMoOrSpa}  there is $\hat x \in \R^3$ such that 
\begin{equation}\label{eq:pf-mainthm-1}
\sup_{x \in E^{h}(t) \Delta B_1(\hat x)} \text{dist}(x, \partial B_1(\hat x))  + \big|P(E^{h}(t)) - P(B_1) \big|  \leq C e^{-c_1t} 
\end{equation}
and by translation we may assume that $\hat x = 0$. For any $T$ large we define  $\Gamma_T$  as in \eqref{def:gammaT} and $\Sigma_{T}$ as  in \eqref{def:sigmaT}.  Our goal is to show that there is $T_0$ large enough such that for every $t \in [T_0+1,\infty) \setminus \Sigma_{T_0}$ the second fundamental form is uniformly bounded $\|B_{E^h(t)}\|_{C^0} \leq C$.

Let $\eps_0>0$ be as in Theorem \ref{thm:decay-flatness}. As we discussed in Section 3, we may define the set $\Sigma_{T_0} \subset [T_0,\infty)$  as in \eqref{def:sigmaT} such that for all  $t_0 \in [T_0+1,\infty) \setminus \Sigma_{T_0}$  it holds \eqref{eq:density-point-pl} if $n=1$ or  \eqref{eq:density-sigmaT} if $n=2$, i..e, 
\begin{equation}\label{eq:pf-mainthm-2}
\inf_{r \in (0,1)} \frac{1}{r^2} \Big|\{ t \in [t_0-r^2,t_0] : \frac14 \| \mathrm{H}_{E^{h}(t) }\|_{L^2}^2 \leq 5 \pi  \} \Big| \geq 1 - \eps_0,
\end{equation}
 and by \eqref{eq:measure-sigmaT} $|\Sigma_{T_0}| \leq C e^{-\frac{c_1}{4}T_0}$. Let us  fix  $t_0 \in [T_0+1,\infty) \setminus \Sigma_{T_0}$.

We show that   the flatness assumption \eqref{eq:flatt-assumption} holds  for $r=r_0$ at every point $x_0 \in \partial E^h(t_0)$, when $T_0$ is large enough.  Let us fix $r_0$  and  $x_0 \in \partial E^{h}(t_0)$. By \eqref{eq:pf-mainthm-1}, with $\hat x = 0$,  we may by rotation assume that $x_0 = |x_0| e_{n+1}$ and  $\big| |x_0| -1\big| \leq C e^{-c_1t_0}$. 
We choose  caloric polynomial $P : \R^n \times \R \to \R$ as 
\[
P(x',t) = 1- \frac{1}{2}|x'|^2 - n(t - t_0). 
\]
It holds for $0 <r_0 < \frac14$
\begin{equation}\label{eq:pf-mainthm-3}
\sup_{|x'|<r_0}\big| \sqrt{1 - |x'|^2} -  (1- \frac{|x'|^2}{2} ) \big| \leq r_0^4. 
\end{equation}
Let $\Lambda(\cdot)$ be as in \eqref{def:Lambda}, i.e.,  $\Lambda(t) := \int_{t_0}^t \lambda^h(\tau+h) \, d \tau$. We claim that there is $c_3>0$ such that 
\begin{equation}\label{eq:pf-mainthm-4}
 |\Lambda(t) - n(t_0-t)| \leq Ce^{-c_3 T_0}
\end{equation} 
for all $t \in [t_0-1, t_0]$. 

Recalling the definition of $\Gamma_T$  in \eqref{def:gammaT} and using \eqref{def:gammaT-3} if $n=2$ and  \eqref{eq:gammaT-pl} if $n=1$, we have that $|\lambda^h(t) -n| \leq C e^{-\frac{c_1q}{2} T}$ for all $t \in [T,\infty) \setminus \Gamma_T$. On the other hand it holds $|\Gamma_T| \leq C e^{-\frac{c_1}{2} T}$. Therefore using Proposition \ref{prop:apriori-est} (iii)  we have 
\[
\begin{split}
 |\Lambda(t) - n(t_0-t)| &\leq \int_{t}^{t_0} |\lambda^h(\tau+h) -n|\, d \tau \\
&\leq  \int_{[t,t_0+h] \cap \Gamma_{T_0}} |\lambda^h(\tau) -n|\, d \tau + \int_{[t,t_0+h] \setminus  \Gamma_{T_0}} |\lambda^h(\tau) -n|\, d \tau \\
&\leq C |\Gamma_{T_0} |^{\frac12} + C  e^{-\frac{c_1q}{2} T_0}  \leq C e^{-c_3 T_0}.
\end{split}
\]
Hence, we have \eqref{eq:pf-mainthm-4}.

We use \eqref{eq:pf-mainthm-3} and \eqref{eq:pf-mainthm-4} to  conclude
\[
\begin{split}
\sup_{|x'|<r_0}\big| \sqrt{1 - |x'|^2} -  P(x',t) - \Lambda(t) \big| &\leq \sup_{|x'|<r_0}\big| \sqrt{1 - |x'|^2} -  (1- \frac{|x'|^2}{2} ) \big| +  |\Lambda(t) - n(t_0-t)|  \\
&\leq r_0^4 + C e^{-c_3 T_0}
\end{split}
\]
for every $t \in [t_0-r_0^2, t_0]$.  We define the set $\textbf{P}_t := \big{\{}  (x',x_{n+1}) \in \R^{n+1} :  x_{n+1} < P(x',t)  + \Lambda(t) \big{\}}$.  Then by \eqref{eq:pf-mainthm-1}, by $x_0 = |x_0|e_{n+1}$,  and by the above it holds 
\[
 \sup_{ x \in \big(E^h(t)   \Delta \textbf{P}_t\big)  \cap B_{r_0}(e_{n+1})}  \big| x_{n+1} - P(x',t) - \Lambda(t)   \big| \leq  r_0^4 + C e^{-c_3 T_0}+ C e^{-c_1 T_0} \leq r_0^{2+\alpha} 
\]
for all  $t \in [t_0-r_0^2, t_0]$, when $r_0$ is small enough and $T_0$ large enough. Hence, we have the assumption \eqref{eq:flatt-assumption}. We conclude by Proposition \ref{lem:final-reg} that the second fundamental form is bounded, i.e., 
$|B_{E^h(t_0)}(x_0)| \leq C$. Since this holds at every point we have  $\|B_{E^h(t_0)}\|_{C^0} \leq C$ for all $t_0 \in [T_0+1,\infty) \setminus \Sigma_{T_0}$.

We use the uniform bound on the second fundamental form and \eqref{eq:pf-mainthm-1} to conclude, for instance by using the Allard regularity theory \cite[Theorem 5.2]{Sim},  that for all  $t \in [T_0+1,\infty) \setminus \Sigma_{T_0}$ the set  $E^h(t) $ is nearly spherical, i.e.,  we may write 
\[
\partial E^h(t) = \{ (1+u(x,t))x : x \in \mathbb{S}^n \}, \qquad \|u(\cdot, t)\|_{C^2(\mathbb{S}^n)} \leq C
\]
and $\|u(\cdot, t)\|_{C^1(\mathbb{S}^n)} \leq C e^{-ct}$. In particular, we have that the set  $E^h(t)$ satisfies uniform ball condition with radius $\hat r>0$.  Then by Proposition \ref{prop:JN2} there  is $\hat \delta>0$ such that 
all sets $E^h(\tau)$ for  $\tau \in [t,t+\hat \delta]$ satisfy  uniform ball condition with radius $\hat r/2$. Since by \eqref{eq:measure-sigmaT} we have $|\Sigma_{T_0}| \leq C e^{- \frac{c_1}{4} T_0} \leq \frac{\hat \delta}{2}$, when $T_0$ is large enough, 
we conclude that the second fundamental form stays uniformly bounded, i.e.,   
\[
\sup_{t \geq T_0+1}   \|B_{E^h(t)}\|_{C^0} \leq C   \qquad \text{for all }\, t \geq T_0+1. 
\]
Using the  regularity estimate from Proposition \ref{prop:JN2}, we conclude that the sets   $E^h(t)$ are uniformly $C^k$ regular for all $t \geq T_0+2$ and all $k \geq 2$, i.e., 
\[
\partial E^h(t) = \{ (1+u(x,t))x : x \in \mathbb{S}^n \}  \qquad \text{for all }\, t \geq T_0+2
\]
for a function $u(\cdot, t) : \mathbb{S}^n \to \R$ with $\|u(\cdot, t)\|_{C^1(\mathbb{S}^n)} \leq \eps_0$ and $\|u(\cdot, t)\|_{C^k(\mathbb{S}^n)} \leq C_k$ for all $k \geq 2$. Since these estimate are independent of $h$, they hold at the limit as $h \to 0$. This means that the limiting flat flow  is smooth in $[T_0+2,\infty)$. The exponential convergence follows from the above  regularity estimates and from   \eqref{eq:pf-mainthm-1}  via interpolation. This concludes the proof. 
\end{proof}

\section*{Acknowledgments}
This reasearch was funded by  the Academy of Finland grant  347550. 
S. J. was also supported by the National Research Foundation of Korea(NRF) grant funded by the Korea government(MSIT) (RS-2025-24803159 and RS-2026-25474502).
V.J. would like to thank Michael Goldman for useful discussion on the article \cite{KS01}.

\end{document}